\documentclass{ws-ijbc}
\usepackage{ws-rotating}     
\allowdisplaybreaks
\usepackage{hyperref}
\usepackage{graphicx}
\usepackage{epstopdf,subfigure}
\usepackage{multirow}
\usepackage{footnote}
\usepackage{amsmath}
\begin{document}
\newcommand{\f}{\displaystyle\frac}
\newcommand{\il}{\int\limits}
\newcommand{\s}{\sum\limits}
\newcommand{\ba}{\begin{array}}
\newcommand{\ea}{\end{array}}
\newcommand{\ld}{\left\{}
\newcommand{\rd}{\right\}}
\newcommand{\lz}{\left[}
\newcommand{\rz}{\right]}
\newcommand{\lx}{\left(}
\newcommand{\rx}{\right)}
\newcommand{\og}{\omega}
\newcommand{\ra}{\rightarrow}
\newcommand{\di}{\displaystyle}
\newcommand{\al}{\alpha}
\newcommand{\va}{\varepsilon}
\newcommand{\la}{\lambda}
\newcommand{\be}{\beta}
\newcommand{\bp}{\begin{proof}}
\newcommand{\ep}{\end{proof}}
\newcommand{\beq}{\begin{equation}}
\newcommand{\eeq}{\end{equation}}
\newcommand{\bt}{\begin{theorem}}
\newcommand{\et}{\end{theorem}}

\catchline{}{}{}{}{} 

\markboth{S. Chen et al.}{Dynamical Analysis of an Allelopathic Phytoplankton Model with Fear Effect}

\title{Dynamical Analysis of an Allelopathic Phytoplankton Model \\ with Fear Effect}

\author{Shangming Chen}

\address{School of Mathematics and Statistics, Fuzhou University, No.2 Xueyuan Road\\
Fuzhou, Fujian  350108, P. R. China\\
210320019@fzu.edu.cn}

\author{Fengde Chen \footnote{Author for correspondence}}
\address{School of Mathematics and Statistics, Fuzhou University, No.2 Xueyuan Road\\
Fuzhou, Fujian  350108, P. R. China\\
fdchen@fzu.edu.cn}

\author{Vaibhava Srivastava}
\address{Department of Mathematics, Iowa State University, \\
	Ames, IA 50011, USA. \\
	vaibhava@iastate.edu}

 \author{Rana D. Parshad }
\address{Department of Mathematics, Iowa State University, \\
	Ames, IA 50011, USA. \\
	rparshad@iastate.edu}

\maketitle

\begin{history}
\received{(to be inserted by publisher)}
\end{history}

\begin{abstract}
This paper is the first to propose an allelopathic phytoplankton competition ODE model influenced by a fear effect based on natural biological phenomena. It is shown that the interplay of this fear effect and the allelopathic term cause rich dynamics in the proposed competition model, such as global stability, transcritical bifurcation, pitchfork bifurcation, and saddle-node bifurcation. We also consider the spatially explicit version of the model, and prove analagous results. Numerical simulations verify the feasibility of the theoretical analysis. The results demonstrate that the primary cause of the extinction of non-toxic species is the fear of toxic species compared to toxins. Allelopathy  only affects the density of non-toxic species. The discussion provides guidance for the conservation of species and the maintenance of bio-diversity.
\end{abstract}


\section{Introduction}
Phytoplankton are at the base of aquatic food webs and of global importance for ecosystem functioning and services \cite{Winder12}. Moreover, phytoplankton also contribute significantly to economic growth, which is advantageous for the biotechnology, pharmaceutical, and nutraceutical sectors \cite{Pradhan22}. Hence, the investigation of phytoplankton species density holds significant academic significance. A unique phenomenon among phytoplankton species is when secondary metabolites generated by one phytoplankton have an inhibiting influence on the development or physiological operation of another phytoplankton \cite{Legrand03}. This behavior is often called allelopathy when phytoplankton engage in competitive interactions with peers by releasing toxic compounds. Numerous studies have shown that allelopathy plays a crucial role in the competitive dynamics of phytoplankton. For example, Mulderij {\it et al.}\cite{Mulderij06} investigated the allelopathic potential of exudates from the aquatic macrophyte {\it Stratiotes aloides} on the growth of phytoplankton. The study results show that {\it Stratiotes aloides} exerts a chemosensitizing effect on phytoplankton, inhibiting the growth of other algae by releasing toxins.

Maynard-Smith \cite{Maynard-Smith74} added the allelopathic term into the classical two-species Lotka-Volterra competition model in order to account for the harmful impacts exerted by one species on the other:
\begin{equation}
\left\{\begin{array}{l}
\displaystyle\frac{\mathrm{d} N_1(t)}{\mathrm{d} t} =N_1(t)\left [ \alpha _1-\beta_1N_1(t)-v_1N_2(t)-\gamma_1N_1(t)N_2(t) \right ], \vspace{2ex}\\
\displaystyle\frac{\mathrm{d} N_2(t)}{\mathrm{d} t} =N_2(t)\left [ \alpha _2-\beta_2N_2(t)-v_2N_1(t)-\gamma_2N_1(t)N_2(t) \right ] ,
\end{array}\right.
\end{equation}
where $N_i(t)$ ($i=1,2$, the same below) is the density of two competing phytoplankton species, $\alpha_i$ represents the rate of daily cell proliferation, $\beta_i$ denotes the intraspecific competition rate of the i-th species, $v_i$ stands for the rate of interspecific competition, $\gamma_i$ represents the toxicity release rate from the other species to species $i$-th. The initial conditions $N_i(0)>0$.

Based on the work of Maynard-Smith, many scholars have considered the situation where only one species releases toxins. Chen {\it et al.} \cite{Chen13} proposed a discrete system for toxin release from single species:
\begin{equation}
\left\{\begin{array}{l}
x_1(n+1) =x_1(n)\mathrm{exp}\left [r_1(n)-a_{11}(n)x_1(n)-a_{12}(n)x_2(n)-b_1(n)x_1(n)x_2(n) \right], \vspace{2ex}\\
x_2(n+1) =x_2(n)\mathrm{exp}\left [r_2(n)-a_{21}(n)x_1(n)-a_{22}(n)x_2(n)\right].
\end{array}\right.
\label{a1}
\end{equation}
The authors proved the extinction and global stability conditions for system \eqref{a1}. It was found that the extinction of system \eqref{a1} is not affected at low rates of toxin release, meaning that the toxic species cannot extinguish non-toxic species. Further studies on single toxic species were conducted in \cite{Chen16, Chen23}.

However, in reality, a non-toxic species can go extinct even if it is only affected by lower concentrations of toxins. In other words, what factors other than degradation by actual toxins might affect the density of competing phytoplankton species, without additional external factors interfering? Since the effect of allelopathy is based on the classical Lotka-Volterra competition model, we will consider competitive fear.

In 2016, Wang {\it et al.} \cite{Wang16} considered the fear effect for the first time based on the classical two-species Lotka-Volterra predator-prey model:
\begin{equation}
\left\{\begin{array}{l}
\displaystyle\frac{d x}{d t}=r x f(k, y)-d x-a x^{2}-g(x) y, \vspace{2ex}\\
\displaystyle\frac{d y}{d t}=-m y+c g(x) y,
\end{array}\right.
\end{equation}
where $a$ represents the mortality rate due to intraspecific competition of the prey, $g(x)$ is the functional predation rate of the predator, and $f(k, y)=\displaystyle\frac{1}{1+ky}$ represents the anti-predation response of the prey due to the fear of the predator, i.e., the fear effect function. The researchers found that under conditions of Hopf bifurcation, an increase in fear level may shift the direction of Hopf bifurcation from supercritical to subcritical when the birth rate of prey increases accordingly. Numerical simulations also suggest that animals' anti-predator defenses increase as the predator attack rate increases. Further research on the fear effect of the predator-prey model can be seen in \cite{Lai20, Liu22}.

By studying the fear effect of the predator-prey model, scholars generally agree that the non-consumptive effect of fear on the density of bait species is more significant than depredation on them. In connection with natural biological phenomena, prey perceives the risk of predation and respond with a range of anti-predatory responses, such as changes in habitat selection and foraging behavior \cite{Polis89, Peckarsky08}. These changes in various forms, may ultimately affect the overall reproductive rate of the prey population.

The effect of fear on predator-prey systems has been extensively studied, but fear has been considered far less in competition systems. However, there is strong evidence that fear exists in purely competitive systems without predation effects or where predation effects are negligible \cite{Chesson08, Wiens14}. The Barred Owl (\emph{Strix varia}) is a species of Owl native to eastern North America. During the last century, they have expanded their range westward and have been recognized as an invasion of the western North American ecosystem—their range overlaps with the Spotted Owl (\emph{Strix occidentalis}). The Spotted Owl is native to northwestern and western North America, which has led to intense competition between two species \cite{Long19}. The Barred Owl has a strong negative impact on the Spotted Owl, and field observations have reported that barred owls frequently attack spotted owls \cite{Van Lanen11}. Evidence also shows that barred owls actively and unilaterally drive spotted owls out of shared habitat \cite{Wiens14}.

Such evidence motivates us to consider the fear effect in a purely competitive two-species model, in which one competitor causes fear to the other. Thus, Srivastava {\it et al.}\cite{Srivastava23} considered the classical two-group Lotka-Volterra competition model with only one competitor causing fear to the other competitor:
\begin{equation}
\left\{\begin{array}{l}
\displaystyle\frac{d u}{d t}=a_{1} u-b_{1} u^{2}-c_{1} u v, \vspace{2ex}\\
\displaystyle\frac{d v}{d t}=\displaystyle\frac{a_{2} v}{1+ku} - b_{2} v^{2}-c_{2} u v.
\end{array}\right.
\label{a2}
\end{equation}
Their study found that the fear effect leads to exciting dynamics such as saddle-node bifurcation and transcritical bifurcation in system \eqref{a2}. This is not found in the classical Lotka-Volterra competition model. In extension to this work, Chen {\it et al.}\cite{Chen2023} also proved several interesting dynamics for a two-species competitive ODE and PDE systems where an Allee and the fear effect are both present. 

Inspired by the above works, we aim to investigate how the fear parameter affects competitive allelopathic planktonic systems by introducing a fear effect term, where the non-toxic species is ``fearful" of the toxic population. Thus we propose the following model:
\begin{equation}
\left\{\begin{array}{l}
\displaystyle\frac{\mathrm{d} x_1}{\mathrm{d} \tau} =x_1\left( r_1-\alpha_1x_1-\beta_1x_2 \right), \vspace{2ex}\\
\displaystyle\frac{\mathrm{d} x_2}{\mathrm{d} \tau} =x_2\left( \displaystyle\frac{r_2}{1+\eta x_1} -\alpha_2x_2-\beta_2x_1-\xi x_1x_2 \right) ,
\end{array}\right.
\label{a3}
\end{equation}
where $\eta$ is the fear effect parameter and $\xi$ represents the toxic release rate.

In the current manuscript we perform a complete dynamical analysis of system \eqref{a3} with the following innovations:

\begin{itemlist}
\item System \eqref{a3} has at most two positive equilibria, while the global stability of positive equilibria is influenced by the fear effect parameter $\eta$ and the interspecific competition rate $\beta_1$.\vspace{0.1ex}\\
\item Changing the values of the fear effect $\eta$ and the interspecific competition rate $\beta_1$ will cause system \eqref{a3} to experience a transcritical bifurcation at the boundary. At the same time, the toxic release rate $\xi$ will transform the transcritical bifurcation into a pitchfork bifurcation.\vspace{0.1ex}\\
\item The toxic release rate $\xi$ causes system \eqref{a3} to undergo a saddle-node bifurcation in the quadrant. \vspace{0.1ex}\\
\item The toxic release rate $\xi$ only affects the non-toxic species density, while the fear effect $\eta$ can lead to the extinction of non-toxic species.
\item In the spatially explicit system or the PDE case, we analogously see that attraction to boundary equilibrium or an interior equilibrium are both possible depending on parametric restrictions and initial conditions, see theorems \ref{thm_CE1} $\&$ \ref{thm_CE2}. Furthermore strong competition type dynamics are also possible, again depending on parametric restrictions and initial conditions, see theorem \ref{thm_SComp}.

\end{itemlist}

The rest of this paper is organized as follows: The conditions for the system's permanence are laid forth in Section 2, which also demonstrates the solution's positivity and boundness. We examine the existence and types of all equilibria in Section 3 and Section 4. Also, the global stability of positive equilibria is studied in Section 5. In Section 6, we analyze the bifurcation of the system around the equilibria. Numerical simulations are performed in Section 7 to verify the theoretical analysis's feasibility, showing how fear effect and toxin release rate can affect species density. We end this paper with a brief conclusion.

\section{Preliminaries}
In order to reduce the parameters of system \eqref{a3}, the following dimensionless quantities are applied to the non-dimensionalize model system \eqref{a3}
$$
t=r_2\tau,\quad\frac{x_1}{k_1}=x,\quad\frac{x_2}{k_2}=y,\quad \eta k_1=k,\quad \f{\xi k_1 k_2}{r_2}=m,\quad \f{\beta_2 k_1}{r_2}=a,\quad \f{r_1}{r_2}=b,\quad \f{\beta_1 k_2}{r_1}=c,
$$
then system \eqref{a3} becomes the following system:
\begin{equation}
\left\{\begin{array}{l}
\displaystyle\frac{\mathrm{d} x}{\mathrm{d} t} =bx\left( 1-x-cy \right)=xf(x,y)\equiv F(x,y), \vspace{2ex}\\
\displaystyle\frac{\mathrm{d} y}{\mathrm{d} t} =y\left( \f{1}{1+kx}-y-ax-mxy \right)=yg(x,y)\equiv G(x,y),
\end{array}\right.
\label{2}
\end{equation}
all parameters in system \eqref{2} are positive. Based on biological considerations, the initial condition of system \eqref{2} satisfies
\begin{equation}
x(0)>0, y(0)>0.
\label{3}
\end{equation}
\subsection{Positivity and boundedness of the solutions}
\bt
All solutions of system \eqref{2} are positive.
\et
\bp
Since
$$
x(t)=x(0) \mathrm{exp} \left [ \int_{0}^{t} f(x(s),y(s))\mathrm{d}s\ \right ] >0,
$$
and
$$
y(t)=y(0) \mathrm{exp} \left [ \int_{0}^{t} g(x(s),y(s))\mathrm{d}s\ \right ] >0.
$$
So all solutions of system \eqref{2} with initial condition \eqref{3} are positive.

This completes the proof.
\ep
\begin{lemma}\label{lemma2.2}\cite{Chen05}
If $a,b>0$ and $x(0)>0$,
\begin{itemize}
\item $\limsup\limits_{t\rightarrow +\infty} x(t) \le \displaystyle\frac{a}{b}$ when $x^{'}(t)\le x(t)(a-bx(t))$,
\item $\liminf\limits_{t\rightarrow +\infty} x(t) \ge \displaystyle\frac{a}{b}$ when $x^{'}(t)\ge x(t)(a-bx(t))$.
\end{itemize}
\end{lemma}
\bt
The solutions of system \eqref{2} are bounded.
\et
\bp
According to the first equation of system \eqref{2},
$$
\f{\mathrm{d} x}{\mathrm{d} t} =bx \left( 1-x-cy \right) \le x(b -b x),
$$
by applying Lemma \ref{lemma2.2} to the above inequality, we have
\begin{equation}
\limsup\limits_{t\rightarrow +\infty} x(t) \le \frac{b}{b} =1.
\label{2.3}
\end{equation}

Similarly, according to the second equation of system \eqref{2}, we have
$$
\displaystyle\frac{\mathrm{d} y}{\mathrm{d} t} =y\left( \f{1}{1+kx}-y-ax-mxy \right)\le y(1-y),
$$
so
\begin{equation}
\limsup\limits_{t\rightarrow +\infty} y(t) \le 1.
\label{2.4}
\end{equation}

This completes the proof.
\ep
\subsection{Permanence of the system}
\begin{definition}
System \eqref{2} is considered to be permanent if there are two positive constants, denoted as $m$ and $M$, which are not dependent on the solutions of system \eqref{2}, such that each positive solution $(x(t,x_0,y_0),y(t,x_0,y_0))$ of system \eqref{2} with the initial condition $(x_0,y_0)\in Int(R_+^2)$ satisfies
$$
m\le \liminf\limits_{t\rightarrow +\infty} x(t,x_0,y_0)\le \limsup\limits_{t\rightarrow +\infty} x(t,x_0,y_0)\le M,
$$
$$
m\le \liminf\limits_{t\rightarrow +\infty} y(t,x_0,y_0)\le \limsup\limits_{t\rightarrow +\infty} y(t,x_0,y_0)\le M.
$$
\end{definition}

\bt\label{thm2.5}
System \eqref{2} is permanent if $0 < k < k^*$ and $0 < c < 1$.
\et
\bp
From \eqref{2.3} and \eqref{2.4}, for $\varepsilon >0$ small enough without loss of generality, there is $T>0$ such that, for $t>T$, we have
$$
x(t)\le 1+\varepsilon, \quad y(t)\le 1+\varepsilon.
$$
According to the first equation of system \eqref{2},
$$
\displaystyle\frac{\mathrm{d} x}{\mathrm{d} t} =x\left [ (b-b cy)-b x \right ] \ge x\left [ (b-b c(1+\varepsilon))-b x \right ],
$$
by applying Lemma \ref{lemma2.2} to above differential inequality, we have
$$
\liminf\limits_{t\rightarrow +\infty} x(t)\ge 1-c(1+\varepsilon).
$$
Setting $\varepsilon \rightarrow0$ in above inequality leads to
\begin{equation}
\liminf\limits_{t\rightarrow +\infty} x(t)\ge 1-c.
\label{2.5}
\end{equation}
Similarly, according to the second equation of system \eqref{2},
$$
\displaystyle\frac{\mathrm{d} y}{\mathrm{d} t} =y\left( \f{1}{1+kx}-y-ax-mxy \right) \ge y \left [ (\f{1}{1+k(1+\varepsilon)}-a(1+\varepsilon))-(1+m(1+\varepsilon))y\right ],
$$
by applying Lemma \ref{lemma2.2} to above differential inequality, we have
$$
\liminf\limits_{t\rightarrow +\infty} y(t)\ge \f{\f{1}{1+k(1+\varepsilon)}-a(1+\varepsilon)}{1+m(1+\varepsilon)}.
$$
Setting $\varepsilon \rightarrow0$ in above inequality leads to
\begin{equation}
\liminf\limits_{t\rightarrow +\infty} y(t)\ge \f{\f{1}{1+k}-a}{1+m}.
\label{2.6}
\end{equation}
In summary, we select $M=1$, $m=\mathrm{min}\left \{ 1-c, \f{\f{1}{1+k}-a}{1+m} \right \} $, which obviously independent of the solution of system \eqref{2}. Let $\f{1}{a}-1\triangleq k^*$. Then, \eqref{2.3}, \eqref{2.4}, \eqref{2.5} and \eqref{2.6} show that system \eqref{2} is permanent under the assumption of Theorem \ref{thm2.5}.

This completes the proof.
\ep

\section{Boundary Equilibria and Their Types}

It is obvious that system\eqref{2} includes two boundary equilibria $E_1(1,0)$, $E_2(0,1)$, as well as a constant equilibrium point $E_0(0,0)$. In the following, we will examine the types of them. The Jacobian matrix of system \eqref{2} is given by
\beq
J(E)=\begin{bmatrix}
-b(2x+cy-1) & -bcx\\
-y\left [ \displaystyle\frac{k}{(1+kx)^2}+a+my \right ] & \displaystyle\frac{1}{1+kx}-(2my+a)x-2y
\end{bmatrix}\triangleq\begin{bmatrix}
 B_1 & B_2\\
 B_3 &B_4
\end{bmatrix}.
\label{8}
\eeq
From this, we can obtain
\beq
J(E_0)=\begin{bmatrix}
b & 0\\
0 & 1
\end{bmatrix},
\eeq
\beq
J(E_1)=\begin{bmatrix}
-b & -bc\\
0 & \f{1}{1+k}-a
\end{bmatrix},
\eeq
\beq
J(E_2)=\begin{bmatrix}
b(-c+1) & 0\\
-a-k-m & -1
\end{bmatrix}.
\eeq
Then we get the following theorem.
\bt\label{Thm4}
The types of boundary equilibria are illustrated in the following:
\begin{enumerate}
\item $E_0$ is always a source.
\item
\begin{enumerate}
\item $E_1$ is a hyperbolic stable node when $k>k^*$.
\item When $k=k^*$,
\begin{enumerate}
\item $E_1$ is an attracting saddle-node, and the parabolic sector is on the upper half-plane if $m>m^*$ (Fig. 1(a)).
\item $E_1$ is an attracting saddle-node, and the parabolic sector is on the lower half-plane if $0<m<m^*$ (Fig. 1(b)).
\item $E_1$ is a nonhyperbolic saddle if $m=m^*$ (Fig. 1(c)).
\end{enumerate}
\item $E_1$ is a hyperbolic saddle when $0<k<k^*$.
\end{enumerate}
\item
\begin{enumerate}
\item $E_2$ is a hyperbolic stable node when $c>1$.
\item When $c=1$,
\begin{enumerate}
\item $E_2$ is an attracting saddle-node, and the parabolic sector is on the right half-plane if $0<m<m^{**}$ (Fig. 2(a)).
\item $E_2$ is an attracting saddle-node, and the parabolic sector is on the left half-plane if $m>m^{**}$ (Fig. 2(b)).
\item $E_2$ is a degenerate stable node if $m=m^{**}$ (Fig. 2(c)).
\end{enumerate}
\item $E_2$ is a hyperbolic saddle if $0<c<1$.
\end{enumerate}
\end{enumerate}
\et
\bp
Due to $\la^{E_0}_1=b>0$, $\la^{E_0}_2=1>0$, so $E_0$ is always a source.

For $E_1$, $\la^{E_1}_1=-b<0$. When $\la^{E_1}_2<0$, i.e., $k>k^*$, $E_1$ is a hyperbolic stable node. When $\la^{E_1}_2>0$, i.e., $0<k<k^*$, $E_1$ is a hyperbolic saddle. When $\la^{E_1}_2=0$, i.e., $k=k^*$, $E_1$ is a degenerate equilibrium point. We then have the following debate.
\begin{figure}[h]
\centering
\subfigcapskip=-25pt
\subfigure[$m>m^*$]
{\scalebox{0.45}[0.45]{
\includegraphics{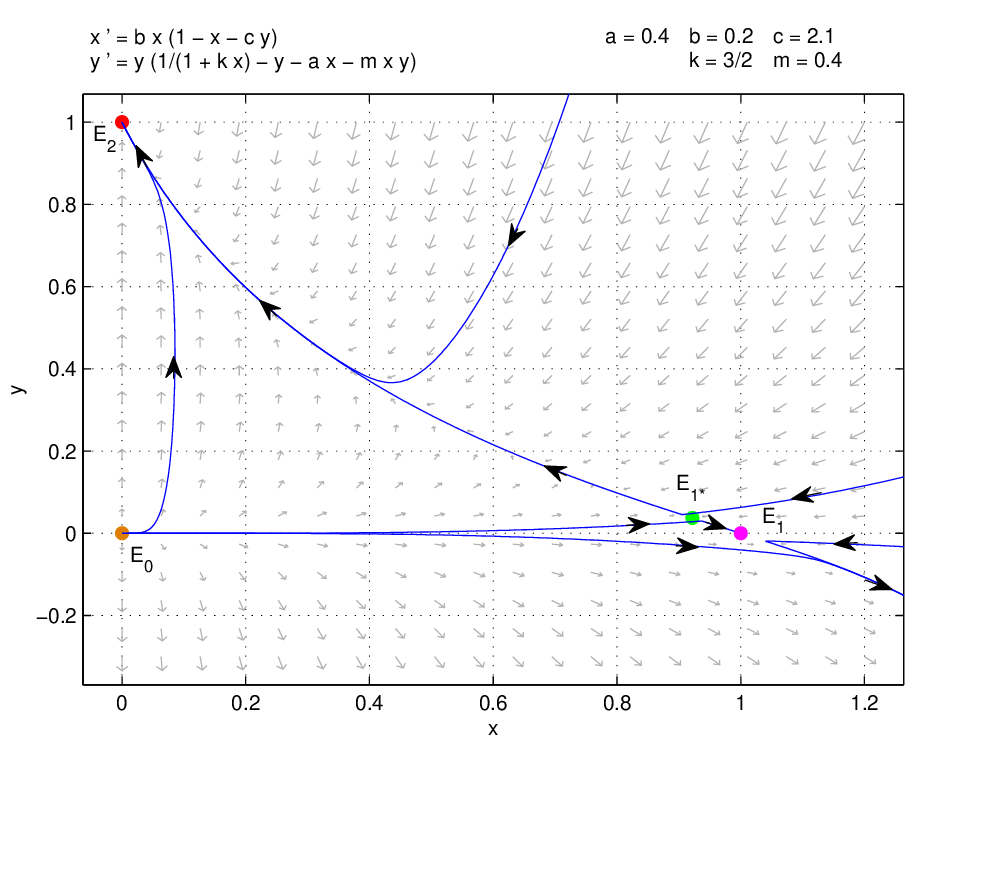}}}
\subfigure[ $0<m<m^*$]
{\scalebox{0.45}[0.45]{
    \includegraphics{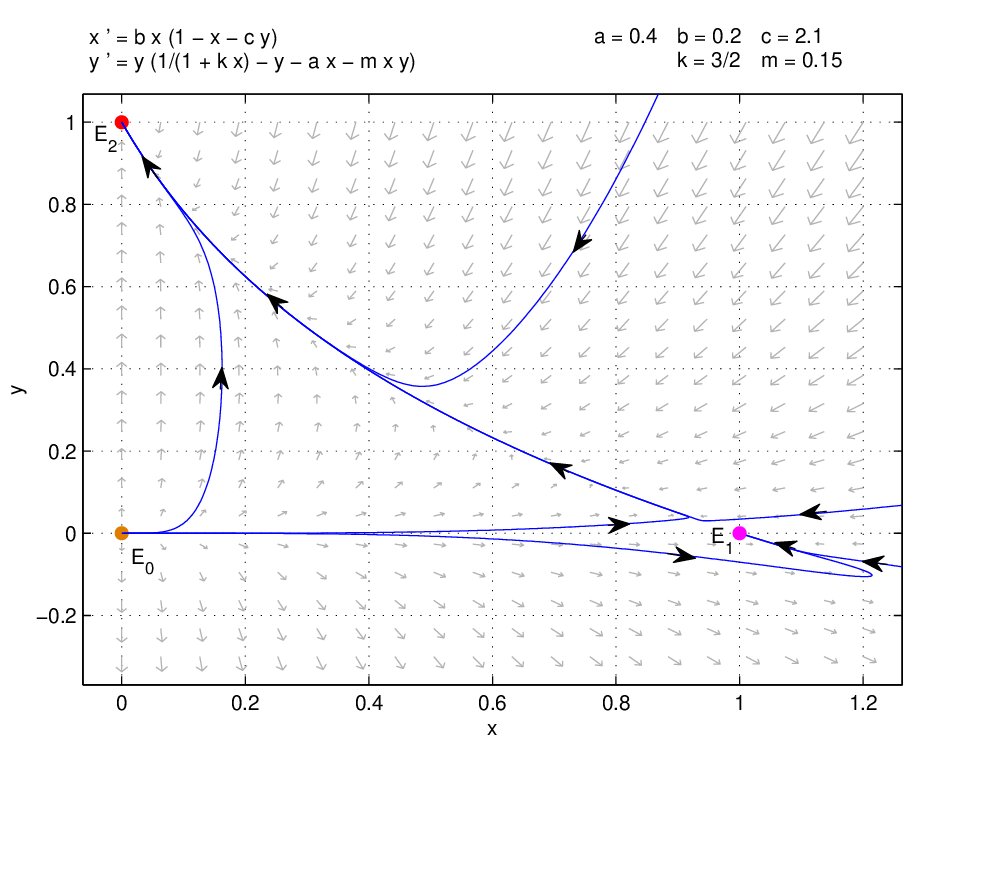}}}
    \subfigure[ $m=m^*$]
{\scalebox{0.45}[0.45]{
    \includegraphics{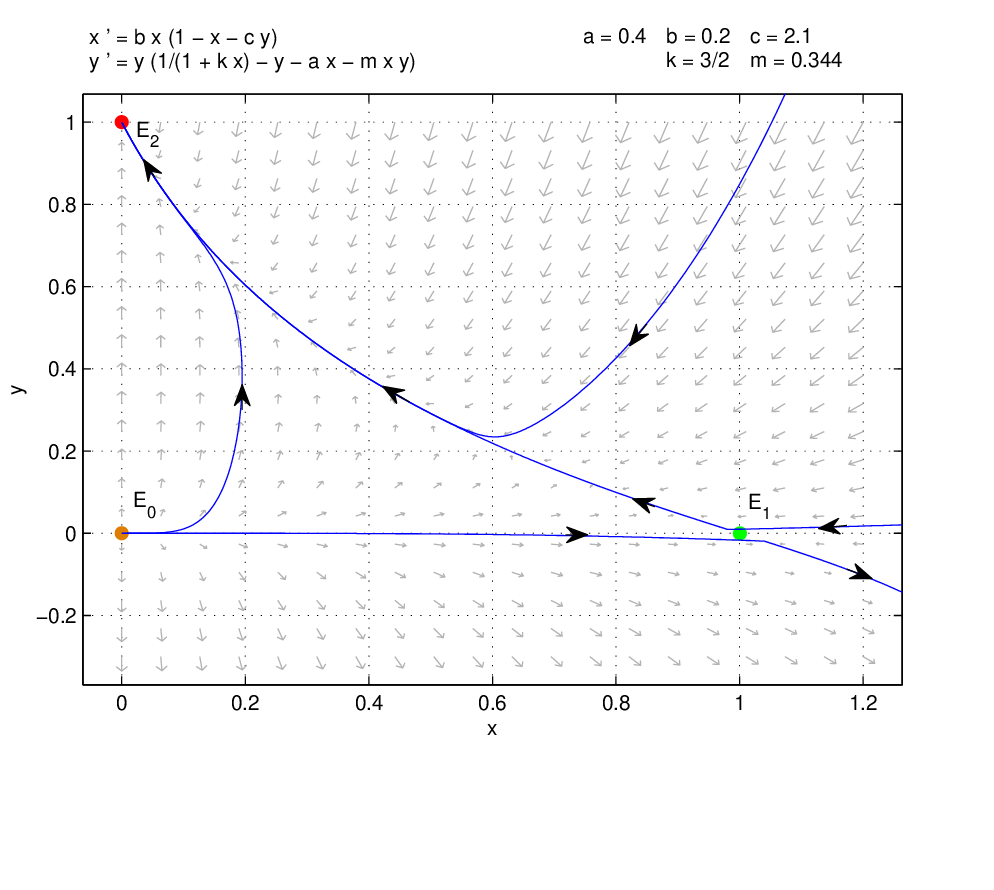}}}
\caption{Red, green, pink, and orange points indicate stable node, saddle, saddle-node, and unstable node (source), respectively. The value of the toxin release rate $m$ affects the solution orbit near the boundary equilibrium point $E_1$.}
\end{figure}

The equilibrium point $E_1$ is translated to the origin by applying the transformation $(X, Y) = (x-1, y)$. We perform a Taylor expansion around the origin, then system \eqref{2} becomes
$$
\left\{\begin{array}{l}
\displaystyle\frac{\mathrm{d} X}{\mathrm{d} t} =-b X-b c Y -b c XY -bX^{2}, \vspace{2ex}\\
\displaystyle\frac{\mathrm{d} Y}{\mathrm{d} t} =-\left(1+m \right) Y^{2}+a \left(-2+a \right) XY +a \left(-1+a \right)^{2} X^{2}Y -mXY^{2}  +P_1(X,Y),
\end{array}\right.
$$
where $P_i(X, Y)$ are power series in $(X, Y)$ with terms $X^IY^J$ satisfying $I+J \ge 4$ (the same below).

In the next step, we make the following transformations to the above system
$$
\begin{bmatrix}
X \\Y

\end{bmatrix}=\begin{bmatrix}
-bc & -b\\
b &0
\end{bmatrix}\begin{bmatrix}
X_1 \\Y_1

\end{bmatrix},
$$
and letting $\tau=-b t$, for which we will retain $t$ to denote $\tau$ for notational simplicity, we get
\begin{equation}
\left\{\begin{array}{l}
\displaystyle\frac{\mathrm{d} X_1}{\mathrm{d} t} =a_{20}X_1^2+a_{11}X_1Y_1+a_{30}X_1^3+a_{21}X_1^2Y_1+a_{12}X_1Y_1^2, \vspace{2ex}\\
\displaystyle\frac{\mathrm{d} Y_1}{\mathrm{d} t} =Y_1+b_{20}X_1^2+b_{11}X_1Y_1+b_{02}Y_1^2+b_{30}X_1^3+b_{21}X_1^2Y_1+b_{12}X_1Y_1^2+P_2(X_1,Y_1),
\end{array}\right.
\label{12}
\end{equation}
where
\begin{flalign}
\begin{split}
& a_{20}=\left(a^{2} c -2 a c +m +1\right), \quad a_{11}=a(-2+a), \quad a_{30}=- b c \left(a^{3} c -2 a^{2} c +a c +m \right), \\
&a_{21}=- b \left(2 a^{3} c -4 a^{2} c +2 a c +m \right), \quad a_{12}=- a \left(-1+a \right)^{2} b, \quad b_{20}=-\left(a^{2} c -2 a c +m +1\right) c,\\
&b_{11}=c \left(a^{2}-2 a +b \right), \quad b_{02}=-b, \quad b_{30}=b c^{2} \left(a^{3} c -2 a^{2} c +a c +m \right), \quad b_{12}=a \left(-1+a \right)^{2} b c,\\
&b_{21}=b \left(2 a^{3} c -4 a^{2} c +2 a c +m \right) c.\\
\nonumber
\end{split}&
\end{flalign}
Therefore, according to Theorem 7.1 in Chapter 2 of \cite{Zhang92}, if $a_{02}>0$, i.e., $m>-1+\left(-a^{2}+2 a \right) c\triangleq m^*$, $E_1$ is an attracting saddle-node, and the parabolic sector is on the upper half-plane (Fig. 1(a)). If $a_{02}<0$, i.e., $0<m<m^*$, $E_1$ is an attracting saddle-node, and the parabolic sector is on the lower half-plane (Fig. 1(b)). If $a_{02}=0$, i.e., $m=m^*$, system \eqref{12} becomes
\begin{equation}
\left\{\begin{array}{l}
\displaystyle\frac{\mathrm{d} X_1}{\mathrm{d} t} =a_{11}X_1Y_1+a_{30}X_1^3+a_{21}X_1^2Y_1+a_{12}X_1Y_1^2, \vspace{2ex}\\
\displaystyle\frac{\mathrm{d} Y_1}{\mathrm{d} t} =Y_1+b_{11}X_1Y_1+b_{02}Y_1^2+b_{30}X_1^3+b_{21}X_1^2Y_1+b_{12}X_1Y_1^2+P_2(X_1,Y_1).
\end{array}\right.
\label{3.6}
\end{equation}
By the existence theorem of the implicit function, it follows that $Y_1 = \phi (X_1)$ can be solved from the second equation of system \eqref{3.6} in a sufficiently small domain at the origin $(0,0)$ and satisfies $\phi(0)=\phi ^{'}(0)=0$.
Substituting 
$$
Y_1 = \phi (X_1)=-b_{30}X_1^3+\cdots\cdots
$$
into the first equation of system \eqref{3.6}, we get
$$
\displaystyle\frac{\mathrm{d} X_1}{\mathrm{d} t}=a_{30}X_1^3+\cdots \cdots
$$
where
$$
a_{30}=- b c \left(a^{3} c -3 a^{2} c +3 a c -1\right).
$$
From $m^*=-1+\left(-a^{2}+2 a \right) c>0$, we get 
$$
- b c \left(a^{3} c -3 a^{2} c +3 a c -1\right)<-abc^2(a-1)^2<0.
$$

According to Theorem 7.1 again, $E_1$ is a nonhyperbolic saddle since $a_{03}<0$ (Fig. 1(c)).

\begin{figure}[h]
\centering
\subfigcapskip=-25pt
\subfigure[$0<m<m^{**}$]
{\scalebox{0.45}[0.45]{
\includegraphics{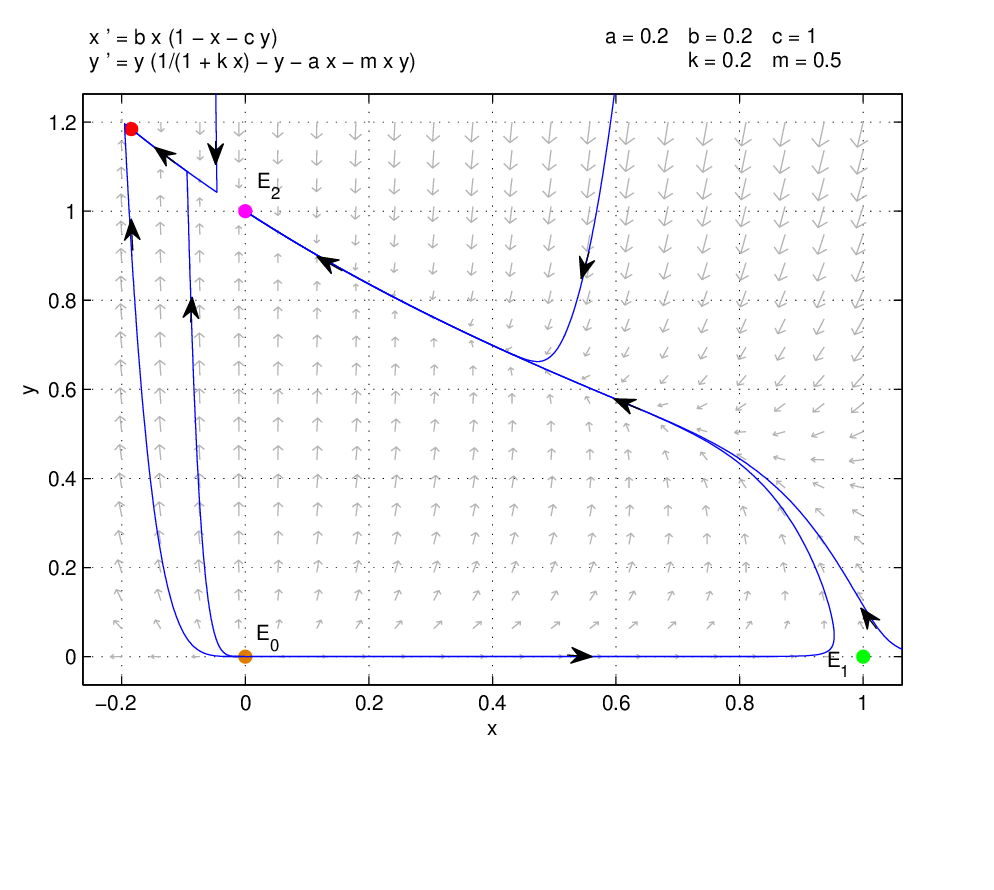}}}
\subfigure[ $m>m^{**}$]
{\scalebox{0.45}[0.45]{
    \includegraphics{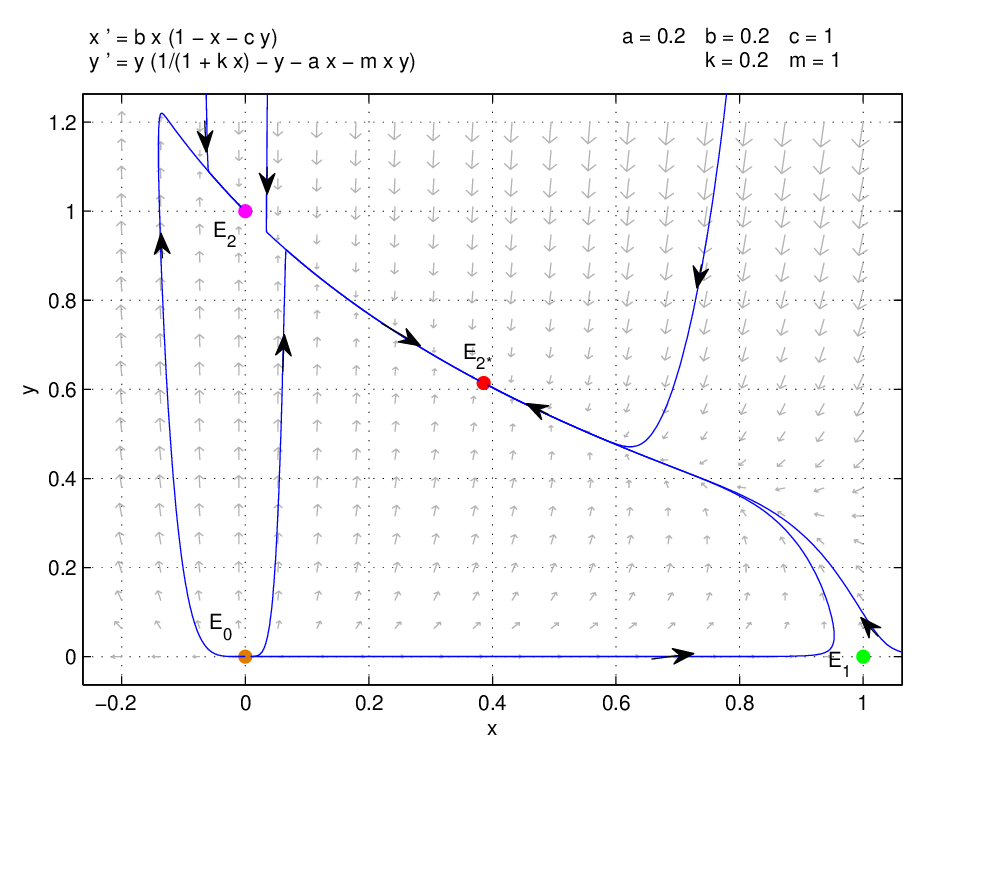}}}
    \subfigure[ $m=m^{**}$]
{\scalebox{0.45}[0.45]{
    \includegraphics{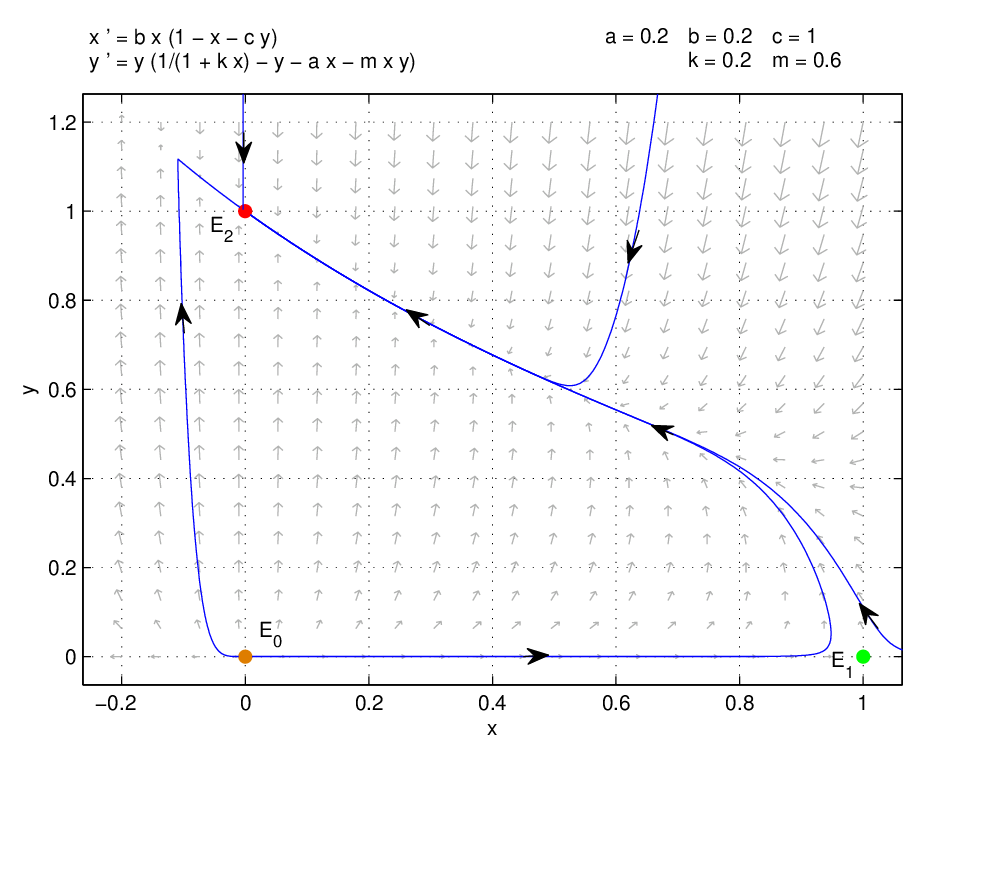}}}
\caption{Red, green, pink, and orange points indicate stable node, saddle, saddle-node, and unstable node (source), respectively. The value of the toxin release rate $m$ affects the solution orbit near the boundary equilibrium point $E_2$.}
\end{figure}

For $E_2$, $\la^{E_2}_2=-1<0$. When $\la^{E_2}_1<0$, i.e., $c>1$, $E_2$ is a hyperbolic stable node. When $\la^{E_2}_1>0$, i.e., $0<c<1$, $E_2$ is a hyperbolic saddle. When $\la^{E_2}_1=0$, i.e., $c=1$, $E_2$ is a degenerate equilibrium point. Then we conduct the following discussion.

We move equilibrium $E_2$ to the origin by transforming $(X_2, Y_2)=(x, y-1)$ and make Taylor's expansion around the origin, then system \eqref{2} becomes
$$
\left\{\begin{array}{l}
\displaystyle\frac{\mathrm{d} X_2}{\mathrm{d} t} =-bX_2^{2} -bX_2 Y_2 , \vspace{2ex}\\
\displaystyle\frac{\mathrm{d} Y_2}{\mathrm{d} t} =-\left(a +k +m \right) X_2 -Y_2 +k^{2}X_2^{2}-\left(2 m +a +k \right) X_2Y_2 -Y_2^{2}-k^{3}X_2^{3}+k^{2} X_2^{2}Y_2- m X_2Y_2^{2} +P_3(X_2,Y_2).
\end{array}\right.
$$

In the next step, we make the following transformations to the above system
$$
\begin{bmatrix}
X_2 \\Y_2

\end{bmatrix}=\begin{bmatrix}
-\f{1}{a+k+m} & 0\\
1 &1
\end{bmatrix}\begin{bmatrix}
X_3 \\Y_3

\end{bmatrix},
$$
and letting $\tau=-t$, for which we will retain $t$ to denote $\tau$ for notational simplicity, we get
\begin{equation}
\left\{\begin{array}{l}
\displaystyle\frac{\mathrm{d} X_3}{\mathrm{d} t} =-\f{b \left(-1+a +k +m \right) }{\left(a +k +m \right)^{2}}{X_3}^{2}-\f{b }{a +k +m}X_3Y_3, \vspace{2ex}\\
\displaystyle\frac{\mathrm{d} Y_3}{\mathrm{d} t} =Y_3+c_{20}X_3^2+c_{11}X_3Y_3+Y_3^2+c_{30}X_3^3+c_{21}X_3^2Y_3+c_{12}X_3Y_3^2+P_4(X_3,Y_3),
\end{array}\right.
\label{14}
\end{equation}
where
\begin{flalign}
\begin{split}
& c_{20}=\f{m a +k^{2}+m k +m^{2}}{\left(a +k +m \right)^{2}}, \quad c_{11}=\f{a +k}{a +k +m},\quad c_{21}=-\f{2 m a +k^{2}+2 m k +2 m^{2}}{\left(a +k +m \right)^{2}},\quad c_{12}=-\f{m}{a +k +m},\\
&c_{30}=-\f{a^{2} m +k^{2} a +2 a k m +2 a \,m^{2}+2 k^{3}+2 k^{2} m +2 k \,m^{2}+m^{3}}{\left(a +k +m \right)^{3}}.\nonumber
\end{split}&
\end{flalign}
We define $m^{**} = 1-a-k$. Hence by Theorem 7.1, if $0<m<m^{**}$, $E_2$ is an attracting saddle-node, and the parabolic sector is on the right half-plane (Fig. 2(a)). If $m>m^{**}$, $E_2$ is an attracting saddle-node, and the parabolic sector is on the left half-plane (Fig. 2(b)). If $m=m^{**}$, system \eqref{14} becomes
\begin{equation}
\left\{\begin{array}{l}
\displaystyle\frac{\mathrm{d} X_3}{\mathrm{d} t} =-bX_3Y_3, \vspace{2ex}\\
\displaystyle\frac{\mathrm{d} Y_3}{\mathrm{d} t} =Y_3+c_{20}X_3^2+c_{11}X_3Y_3+Y_3^2+c_{30}X_3^3+c_{21}X_3^2Y_3+c_{12}X_3Y_3^2+P_4(X_3,Y_3).
\end{array}\right.
\label{15}
\end{equation}
By using the second equation of system \eqref{15}, we obtain the implicit function
$$
Y_3=-c_{20}X_3^2+(c_{11}c_{20}-c_{30})X_3^3+\cdots \cdots
$$
and
$$
\displaystyle\frac{\mathrm{d} X_3}{\mathrm{d} t} =bc_{20}X_3^3+\cdots \cdots,
$$
where
$$
bc_{20}=\f{b(m a +k^{2}+m k +m^{2})}{\left(a +k +m \right)^{2}}>0.
$$
According to Theorem 7.1 again, $E_2$ is a degenerate stable node due to the negative time transformations (Fig. 2(c)).
\ep
\begin{remark}
The biological significance of the parameters $k$ and $c$ are the fear effect of non-toxic $(y)$ species and the interspecific competition rate of toxic $(x)$ species, respectively. By analyzing the type of boundary equilibria, non-toxic and toxic species will become extinct when $k>k^*$ and $c>1$, respectively.
\end{remark}
\begin{figure}[h]
\centering
\scalebox{0.35}[0.35]{
\includegraphics{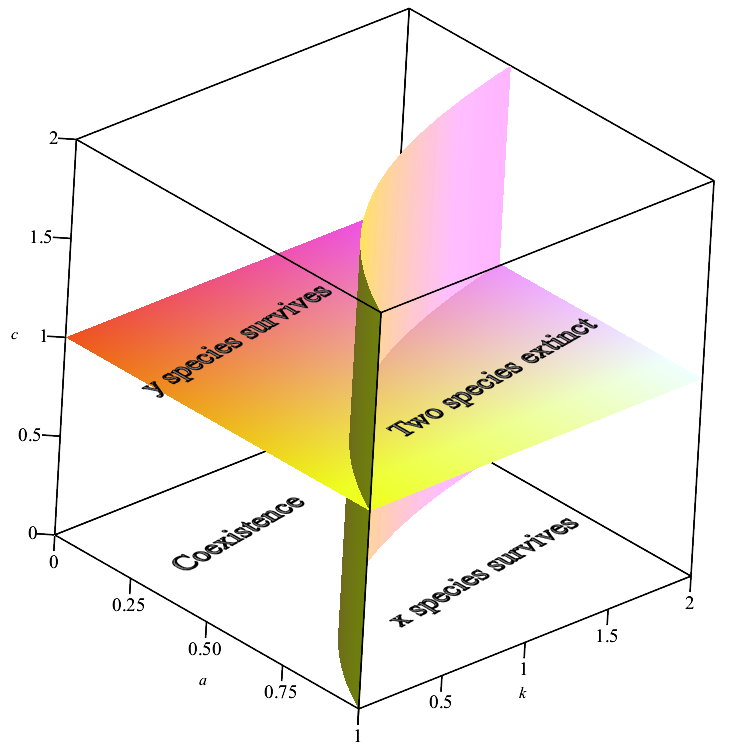}}
\caption{Schematic representation of the biological significance of parameters $k$ and $c$. Range of parameters: $a\in \left ( 0,1 \right )$, $k\in \left ( 1,2 \right )$, $c\in \left ( 0,2 \right )$.}
\label{boundary}
\end{figure}

\section{Positive Equilibria and Their Types}
The intersections of two isoclines $f(x,y)=0$, $g(x,y)=0$ in the first quadrant is the point of positive equilibria. Denote the positive equilibria of system \eqref{2} as $E_{i*}(x_{i},y_{i})$ (i=1, 2, 3), from $f(x,y)=g(x,y)$, we obtain
\begin{equation}
u(x)=A_1x^3+A_2x^2+A_3x+A_4,
\label{16}
\end{equation}
\begin{equation}
v(x)=u'(x)=3A_1x^2+2A_2x+A_3,
\label{17}
\end{equation}
where
$$
A_1=km>0,
$$
$$
A_2=(-ac-m+1)k+m=(A_3+k)k+m,
$$
$$
A_3=-ac-k-m+1,
$$
$$
A_4=c-1.
$$
Denote the discriminant of \eqref{17} as $\Delta=4A_2^2-12A_1A_3$. When $\Delta>0$, \eqref{17} has two real roots, which can be expressed as follows:
$$
x_{v1}=\displaystyle\frac{(ac+m-1)k-m-\sqrt{\Delta} }{3km}, \quad x_{v2}=\displaystyle\frac{(ac+m-1)k-m+\sqrt{\Delta} }{3km}.
$$
Let $u(x)=0$, we have
\beq
m=\f{a c kx^{2}+a c x -kx^{2}+k x -c -x +1}{\left(-1+x \right) \left(k x +1\right) x}.
\label{18}
\eeq
Substituting \eqref{18} into $\mathrm{det}(J(E)) $ and $v(x)$, we get
\beq
\mathrm{det}(J(E))=-\f{x \left(-1+x \right) b}{\left(k x +1\right) c} v(x).
\label{19}
\eeq

The positive of system \eqref{2} is $(x_i,y_i)$ where $y_i=\f{1-x_i}{c}$. Let $m_1\triangleq1-ac-k$ and $m_2\triangleq\f{2 a c k +a c -k -1}{1+k}$. From Theorem 1 and 2, we know that $0< x(t) < 1$ and $0< y(t) < 1$. By a simple analysis, we can obtain the following theorem.
\begin{figure}[h]
\centering
\subfigure[$m=m_1$]
{\scalebox{0.3}[0.3]{
\includegraphics{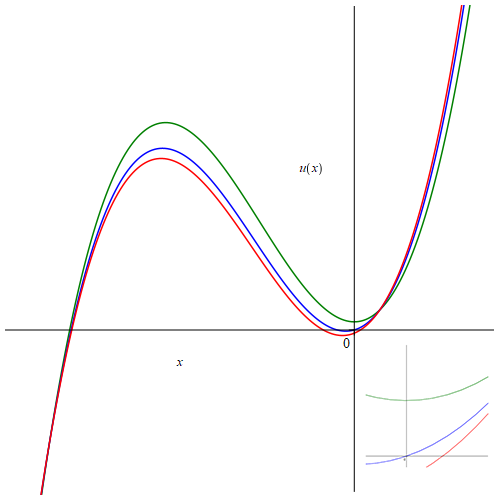}}}
\subfigure[ $m>m_1$, $c>1$]
{\scalebox{0.3}[0.3]{
    \includegraphics{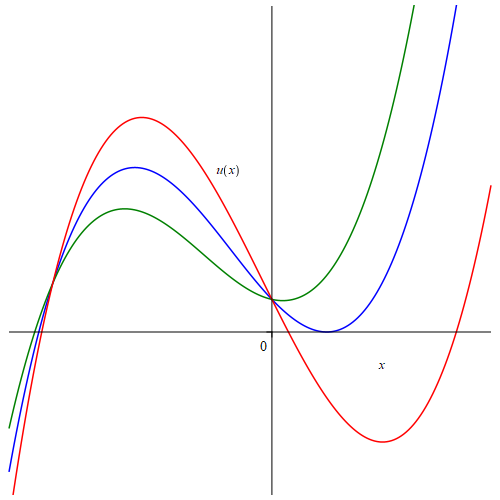}}}
    \subfigure[ $m>m_1$, $0<c\le1$]
{\scalebox{0.3}[0.3]{
    \includegraphics{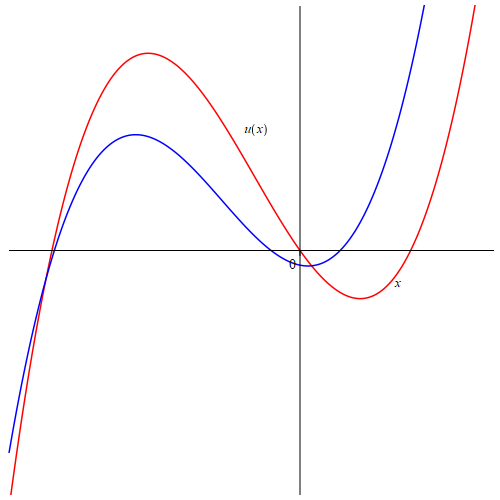}}}
    \subfigure[$0<m<m_1$]
{\scalebox{0.3}[0.3]{
    \includegraphics{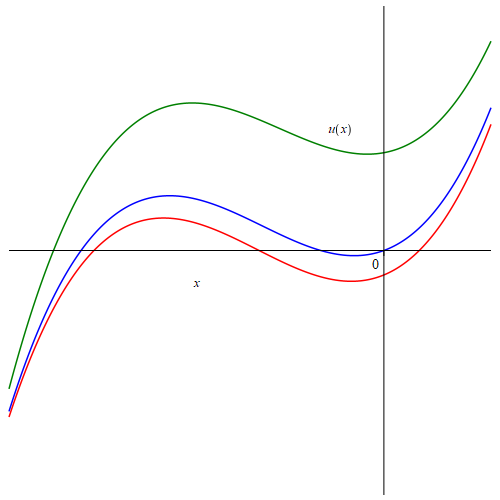}}}
\caption{The number of positive real roots of $u(x)$. }
\end{figure}
\bt\label{Thm5}
The existence of positive equilibria for system \eqref{2} is shown below:
\begin{enumerate}
\item $m=m_1$ (Fig. 4(a))
\begin{enumerate}
\item For $0<c<1$,
\begin{enumerate}
\item System \eqref{2} has a unique positive equilibrium $E_{2*}$ when $0<k<k^*$.
\end{enumerate}
\end{enumerate}
\item $m>m_1$
\begin{enumerate}
\item For $c>1$ (Fig. 4(b)),
\begin{enumerate}
\item If $u(x_{v2})=0$,
\begin{enumerate}
\item System \eqref{2} has a unique positive equilibrium $E_{3*}$ when $m>m_2$.
\end{enumerate}
\item If $u(x_{v2})<0$,
\begin{enumerate}
\item System \eqref{2} has a unique positive equilibrium $E_{1*}$ when $m>m_2$ and $k \ge k^*$.
\item System \eqref{2} has two positive equilibria $E_{1*}$ and $E_{2*}$ when $m>m_2$ and $0<k<k^*$.
\item System \eqref{2} has a unique positive equilibrium $E_{1*}$ when $m=m_2$.
\item System \eqref{2} has a unique positive equilibrium $E_{1*}$ when $0<m<m_2$ and $k>k^*$.
\end{enumerate}
\end{enumerate}
\item For $0<c\le 1$ (Fig. 4(c)),
\begin{enumerate}
\item System \eqref{2} has a unique positive equilibrium $E_{2*}$ when $0<k< k^*$.
\end{enumerate}
\end{enumerate}
\item $0<m<m_1$ (Fig. 4(d))
\begin{enumerate}
\item  For $0<c<1$,
\begin{enumerate}
\item  System \eqref{2} has a unique positive equilibrium $E_{2*}$ when and $0<k<k^*$.
\end{enumerate}
\end{enumerate}
\end{enumerate}
\et

Next, we analyze the types of positive equilibria. Since $-\f{x \left(-1+x \right) b}{\left(k x +1\right) c}>0$, we can easily determine the sign of $\mathrm{det}(J(E_*))$ by \eqref{19}. We conclude that $\mathrm{det}(J(E_{1*}))<0$, $\mathrm{det}(J(E_{2*}))>0$, $\mathrm{det}(J(E_{3*}))=0$. Therefore $E_{1*}$ is a saddle point. For $E_{2*}$, we have
$$
\mathrm{det}(J(E_{2*}))=B_1B_4-B_2B_3>0.
$$
The signs of $B_2$, $B_3$, and $B_4$ have been determined, and we can thus know that $B_1<0$. Finally, we can determine that $\mathrm{tr}(J(E_{2*}))=B_1+B_4<0$ by the above analysis. From $\mathrm{det}(J(E_{2*}))>0$, $\mathrm{tr}(J(E_{2*}))<0$, we know that $E_{2*}$ is a stable node.

Since $\mathrm{det}(J(E_{3*}))=0$, the positive equilibrium point $E_{3*}$ is clearly a degenerate equilibrium point. Next, we analyze the specific type of degenerate equilibrium point $E_{3*}$.

First, it is clear from Theorem 5 and Fig. 4(b) that if $E_{3*}$ exists then the parametric condition needs to satisfy $u(E)=v(E)=0$, where $E=x_{v2}$. From this, we get
$$
a=\frac{3 k mx^{2}-2 k m x +2 k x +2 m x -k -m +1}{k^{2} m x^{4}+2 k m x^{3}+k^{2} x^{2}+mx^{2}+2 k x +1}\triangleq a^*,
$$
$$
c=\frac{k^{2} mx^{4}+2 k mx^{3}+k^{2} x^{2}+mx^{2}+2 k x +1}{2 k x +1}\triangleq c^*.
$$

We move equilibrium $E_{3*}$ to the origin by transforming $(X, Y)=(x-E, y-\f{1-E}{c})$, make Taylor's expansion around the origin, and substitute $a=a^*$, $c=c^*$, then system \eqref{2} becomes
\begin{equation}
\left\{\begin{array}{l}
\displaystyle\frac{\mathrm{d} X}{\mathrm{d} t} =e_{10}X+e_{01}Y+e_{20}X^2+e_{11}XY, \vspace{2ex}\\
\displaystyle\frac{\mathrm{d} Y}{\mathrm{d} t} =d_{10}X+d_{01}Y+d_{20}X^2+d_{02}Y^2+d_{11}XY+P_5(X,Y),
\end{array}\right.
\label{20}
\end{equation}
\begin{table}[h]
\tbl{Positive Equilibria of System  \eqref{2}.}
{\begin{tabular}{l c c c c c c}\\[-10pt]
\toprule
 \multicolumn{1}{c}{$m\sim m_1$}&$c$ &$u(E)$&$m\sim m_2$&$k$&Positive Equilibria\\[3pt]
\hline\\[0pt]
$m=m_1$&$0<c<1$ &/&/&$0<k<k^* $&$E_{2*}$\\[5pt]
\hline\\[0pt]
 \multirow{6}{*}{$m>m_1$}  &\multirow{5}{*}{$c>1$} &$u(E)=0$&$m>m_2$&/ &$E_{3*}$ \\[5pt]
\cline{3-6}\\
&&\multirow{4}{*}{$u(E)<0$} &\multirow{2}{*}{$m>m_2$}&$k\ge k^*$&$E_{1*}$ \\[5pt]
\cline{5-6}\\
&&&&$0<k<k^*$&$E_{1*}$, $E_{2*}$ \\[5pt]
\cline{4-6}\\
&&&$m=m_2$&/&$E_{1*}$ \\[5pt]
\cline{4-6}\\
&&&$0<m<m_2$&$k>k^*$&$E_{1*}$ \\[5pt]
\cline{2-6}\\
&$0<c\le 1$&/&/&$0<k<k^*$&$E_{2*}$\\[5pt]
\hline\\
$0<m<m_1$&$0<c<1$&/&/&$0<k<k^*$&$E_{2*}$\\[5pt]
\botrule
\end{tabular}}
\begin{tabnote}
$E_{1*}$ is a saddle, $E_{2*}$ is a stable node, and $E_{3*}$ is a saddle-node, where $m_1=1-ac-k$, $m_2=\frac{2 a c k +a c -k -1}{1+k}$, $k^*=\frac{1}{a}-1$.
\end{tabnote}
\end{table}
where 
\begin{flalign}
\begin{split}
& e_{10}=-bE, \quad e_{01}=-\frac{b E \left(E k +1\right)^{2} \left(E^{2} m +1\right)}{2 E k +1}, \quad e_{20}=-b, \quad e_{11}=-\frac{b \left(E k +1\right)^{2} \left(E^{2} m +1\right)}{2 E k +1},\\
& d_{10}=\frac{\left(E m +1\right) \left(2 E k +1\right)^{2} \left(-1+E \right)}{\left(E k +1\right)^{4} \left(E^{2} m +1\right)^{2}}, \quad d_{01}=\frac{\left(-1+E \right) \left(E m +1\right) \left(2 E k +1\right)}{\left(E k +1\right)^{2} \left(E^{2} m +1\right)},\\
& d_{20}=-\frac{\left(-1+E \right) \left(2 E k +1\right) k^{2}}{\left(E k +1\right)^{5} \left(E^{2} m +1\right)}, \quad d_{02}=-Em-1, \quad  d_{11}=-\frac{\left(m +1\right) \left(2 E k +1\right)}{\left(E k +1\right)^{2} \left(E^{2} m +1\right)}.
\nonumber
\end{split}&
\end{flalign}

We make the following transformations to system \eqref{20}
$$
\begin{bmatrix}
X \\Y

\end{bmatrix}=\begin{bmatrix}
e_{01} &e_{10} \\
-e_{10} &d_{10}
\end{bmatrix}\begin{bmatrix}
X_4 \\Y_4
\end{bmatrix},
$$
and letting $\tau=L t$, where
$$
\begin{aligned}
L= &-\frac{E^{5} b \,k^{2} m +2 E^{4} b k m +E^{3} b \,k^{2}+E^{3} b m -2 E^{3} k m +2 E^{2} b k +2 E^{2} k m}{\left(E k +1\right)^{2} \left(E^{2} m +1\right)}\\
&+\frac{+2 E^{2} k+E^{2} m -b E -2 E k-E m +E-1}{\left(E k +1\right)^{2} \left(E^{2} m +1\right)},
\end{aligned}
$$
for which we will retain $t$ to denote $\tau$ for notational simplicity. We get
\begin{equation}
\left\{\begin{array}{l}
\displaystyle\frac{\mathrm{d} X}{\mathrm{d} t} =g_{20}X^2+g_{02}Y^2+g_{11}XY, \vspace{2ex}\\
\displaystyle\frac{\mathrm{d} Y}{\mathrm{d} t} =Y+f_{20}X^2+f_{02}Y^2+f_{11}XY+P_6(X,Y),
\end{array}\right.
\label{21}
\end{equation}
where
$$
g_{20}=\frac{\left(E^{2} m +1\right)^{2} \left(E k +1\right)^{3} \left(-1+E \right) \left(3 E^{2} k^{2} m +3 E k m +k^{2}+m \right) E^{2} b^{2}}{H^{2} \left(2 E k +1\right)},
$$
and
$$
\begin{aligned}
H=&E^{5} bk^{2} m +2 E^{4} b k m +E^{3} bk^{2}+E^{3} b m -2 E^{3} k m +2 E^{2} b k\\
&+2 E^{2} k m -2 E^{2} k -E^{2} m +b E +2 E k +E m -E +1,
\end{aligned}
$$
 please see Appendix A for the rest of the parameters.

We note that $g_{20}<0$. Hence by Theorem 7.1 in Chapter 2 in \cite{Zhang92}, $E_{3*}$ is a saddle-node. In summary, together with Theorem 5, we obtain Table 1.

\section{Global Stability of Positive Equilibria}
\begin{lemma}
{\it Bendixson-Dulac} Criteria \cite{Ma15}:

If in a single connected domain $O$, there exists a function $B(x,y) \in C^1(O)$, such that
$$
\frac{\partial (BF)}{\partial x} +\frac{\partial (BG)}{\partial y}\ge 0(\le 0), \quad \forall (x,y)\in O,
$$
and is not constant to zero in any subregion of O. Then system \eqref{2.3} does not have closed trajectories that all lie within O and singular closed trajectories with finitely many singular points. The function $B(x,y)$ is often called the Dulac function.
\end{lemma}
\begin{theorem}
System \eqref{2} cannot have any limit cycle in the interior of the positive quadrant $R_+^2$.
\end{theorem}
\begin{proof}
We use the {\it Bendixson-Dulac} criteria \cite{Ma15} to prove Theorem 6. Construct a Dulac function $B(x,y)=\displaystyle\frac {1}{xy}$. Then it is clear that $B(x,y)$ is positive and so is smooth in a connected domain:
$$
\mathrm{Int} (R_+^2)=\left \{ (x,y)\in R^2 \mid  x>0, y>0 \right \}.
$$

Let
$$
\Delta (x,y)=\frac{\partial (BF)}{\partial x} +\frac{\partial (BG)}{\partial y}=-\frac{b}{y}+\frac{-m x -1}{x} <0.
$$
Thus, $\Delta (x,y)$ is neither identically zero nor changing sign in the interior of the positive quadrant of the $xy$-plane. Using the {\it Bendixson-Dulac} criteria \cite{Ma15}, system \eqref{2} has no closed trajectory, so there is no periodic solution in the first quadrant. 

The proof of Theorem 6 is finished. 
\end{proof}

From Theorem 5 and Table 1, when system \eqref{2} satisfies $0<c<1$, $0<k<k^*$, the boundary equilibria are all unstable, and there is a unique stable positive equilibrium $E_{2*}$ in system \eqref{2}. Since Theorem 6 has proved that system \eqref{2} cannot have any limit cycle in the interior of the positive quadrant, we can obtain the following theorem.

\begin{theorem}\label{Thm:7}
The locally stable positive equilibria $E_{2*}$ is globally stable when $0<c<1$, $0<k<k^*$.
\end{theorem}

\section{Bifurcation Analysis}
\subsection{Transcritical bifurcation}
In proving Theorem 5, we found an interesting phenomenon: when $u(1)=0$, i.e., $k=k^*$, the positive equilibrium point $E_{2*}$ will merge with the boundary equilibrium point $E_1$. Also, the stability of the boundary equilibrium point $E_1$ will change when the parameter $k$ is in different intervals $(0,\frac{1}{a}-1)$ and $(\frac{1}{a}-1,+\infty )$, respectively. Moreover, we find a similar phenomenon for the boundary equilibrium point $E_2$. From this, we conjecture that system \eqref{3} experiences transcritical bifurcations around $E_1$ and $E_2$. We proceed to a rigorous proof below.
\begin{figure}[h]
\centering
\subfigcapskip=-25pt
\subfigure[$k<k^*$]
{\scalebox{0.45}[0.45]{
\includegraphics{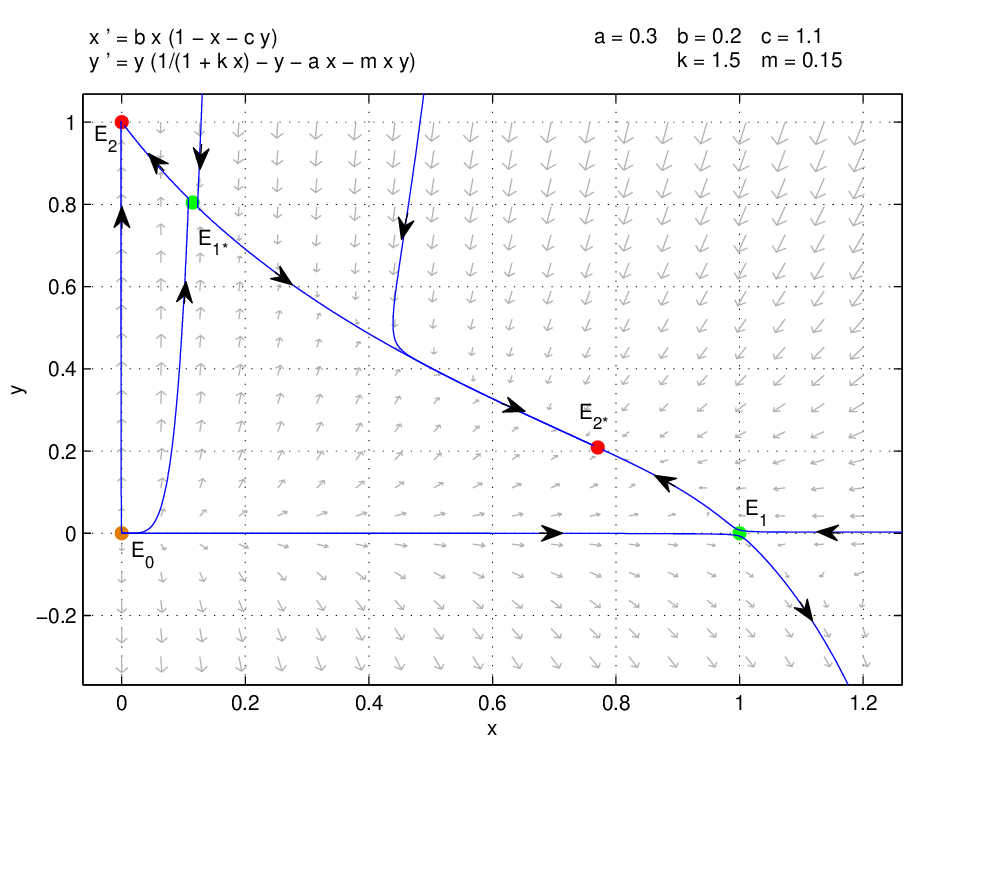}}}
\subfigure[ $k=k^*$]
{\scalebox{0.45}[0.45]{
\includegraphics{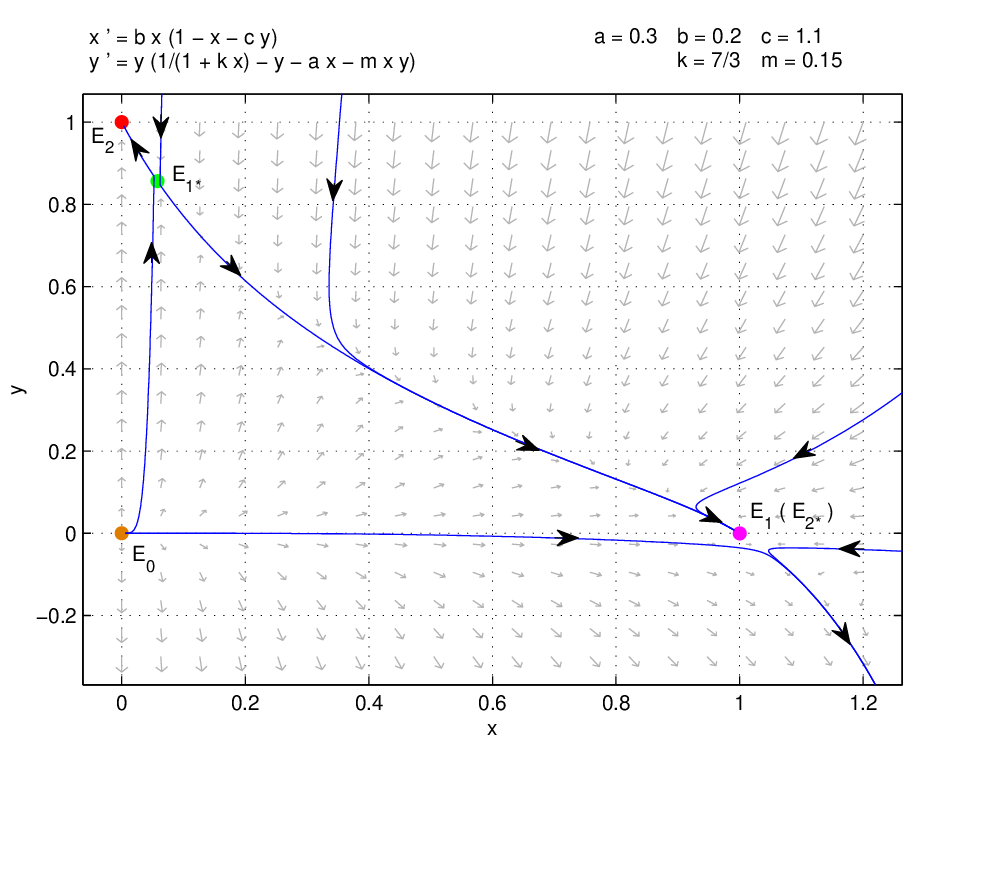}}}
\subfigure[ $k>k^*$]
{\scalebox{0.45}[0.45]{
\includegraphics{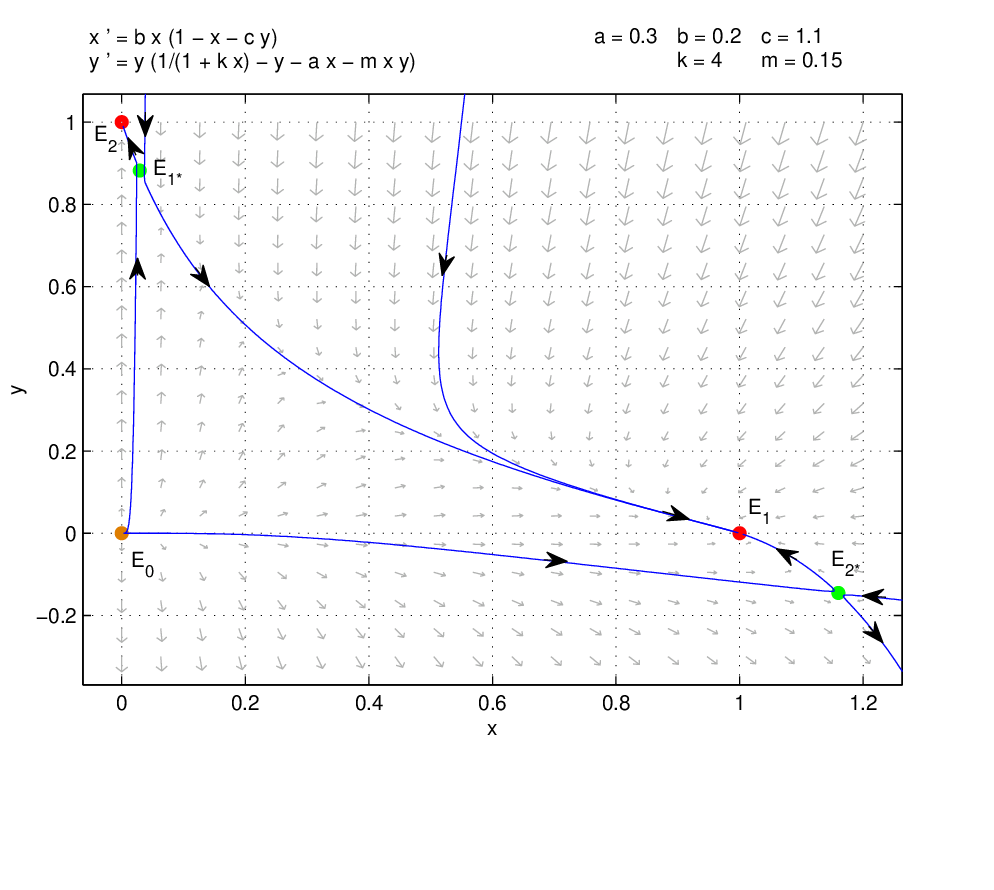}}}
\caption{Red, green, pink, and orange points indicate stable node, saddle, saddle-node, and unstable node (source), respectively. System \eqref{2} undergoes a transcritical bifurcation around $E_1$.}
\end{figure}
\begin{theorem}
System \eqref{2} undergoes a transcritical bifurcation around $E_1$ at the bifurcation parameter threshold $k_{TR} = k^*$ when $u(E)<0$ and $m\ne-a^{2} c +2 a c -1$ (Fig. 5).
\end{theorem}
\begin{proof}
From Theorem 4, we know that the eigenvalues of $J(E_1)$ are $\lambda_1^{E_1}=-b$, $\lambda_2^{E_1}=0$ if $k=k_{TR}=k^*$. Now, let $\mathbf{V_1} = (v_1, v_2)^T$ and $\mathbf{W_1} = (w_1, w_2)^T$ be the eigenvectors of $J(E_1)$ and $J^T(E_1)$ corresponding to $\lambda_1^{E_1}=0$, respectively. By calculating, we obtain
\begin{equation}
\mathbf{V_1}=\begin{bmatrix}
v_1\\v_2

\end{bmatrix} =\begin{bmatrix}
-c \\1

\end{bmatrix},\mathbf{W_1}=\begin{bmatrix}
w_1\\w_2

\end{bmatrix} =\begin{bmatrix}
0 \\1
\end{bmatrix}.
\label{22}
\end{equation}
We assume that
$$
Q(x,y)=\begin{bmatrix}
F(x,y) \\G(x,y)

\end{bmatrix}=\begin{bmatrix}
b x \left(-c y -x +1\right)\\y \left(\f{1}{k x +1}-y -a x -m x y \right)
\end{bmatrix}.
$$
Furthermore,
$$
Q_k(E_1;k_{TR})=\begin{bmatrix}
\displaystyle\frac{\partial F}{\partial k} \vspace{2ex}\\\displaystyle\frac{\partial G}{\partial k}

\end{bmatrix}=\begin{bmatrix}
0\\0

\end{bmatrix},
$$
$$
\left. DQ_k(E_1;k_{TR})\mathbf{V_1}=\left[\begin{array}{cc}
0 & 0
\\
\f{2 y x k}{\left(k x +1\right)^{3}}-\f{y}{\left(k x +1\right)^{2}} & -\f{x}{\left(k x +1\right)^{2}}
\end{array}\right] \right |_{(E_1;k_{TR})}\begin{bmatrix}
-c \\1

\end{bmatrix}=\begin{bmatrix}
0 \\-a^2
\end{bmatrix},
$$
$$
\left.D^2Q(E_1;k_{TR})(\mathbf{V_1}, \mathbf{V_1})=\begin{bmatrix}
\displaystyle\frac{\partial^2F}{\partial x^2}v_1^2+ 2\displaystyle\frac{\partial^2F}{\partial x\partial y}v_1v_2+ \displaystyle\frac{\partial^2F}{\partial y^2}v^2_2\vspace{2ex}\\
\displaystyle\frac{\partial^2G}{\partial x^2}v_1^2+ 2\displaystyle\frac{\partial^2G}{\partial x\partial y}v_1v_2+ \displaystyle\frac{\partial^2G}{\partial y^2}v^2_2
\end{bmatrix}\right|_{(E_1;k_{TR})}=\begin{bmatrix}
0\\ \left(-2 a^{2}+4 a \right) c -2 m -2
\end{bmatrix}.
$$
Thus, we have
$$
\mathbf{W_1}^TQ_k(E_1;k_{TR})=0,
$$
$$
\mathbf{W_1}^T\left[ DQ_k(E_1;k_{TR})\mathbf{V_1} \right]=-a^2\ne0,
$$
$$
\mathbf{W_1}^T\left[ D^2Q(E_1;c_{TR})(\mathbf{V_1}, \mathbf{V_1}) \right]=\left(-2 a^{2}+4 a \right) c -2 m -2\ne0.
$$
Based on {\it Sotomayor's Theorem} \cite{Perko13} , all the transversality conditions for system \eqref{2} to experience a transcritical bifurcation are satisfied. Consequently, system \eqref{2} undergoes a transcritical bifurcation around $E_1$ at the bifurcation parameter threshold $k_{TR} = k^*$.
\end{proof}

\begin{theorem}
System \eqref{2} undergoes a transcritical bifurcation around $E_2$ at the bifurcation parameter threshold $c_{TR} =1$ when $u(E)<0$ and $m\ne1-a-k$ (Fig. 6).
\end{theorem}
\begin{figure}[h]
\centering
\subfigcapskip=-25pt
\subfigure[$0<c<c_{TR}$]
{\scalebox{0.45}[0.45]{
\includegraphics{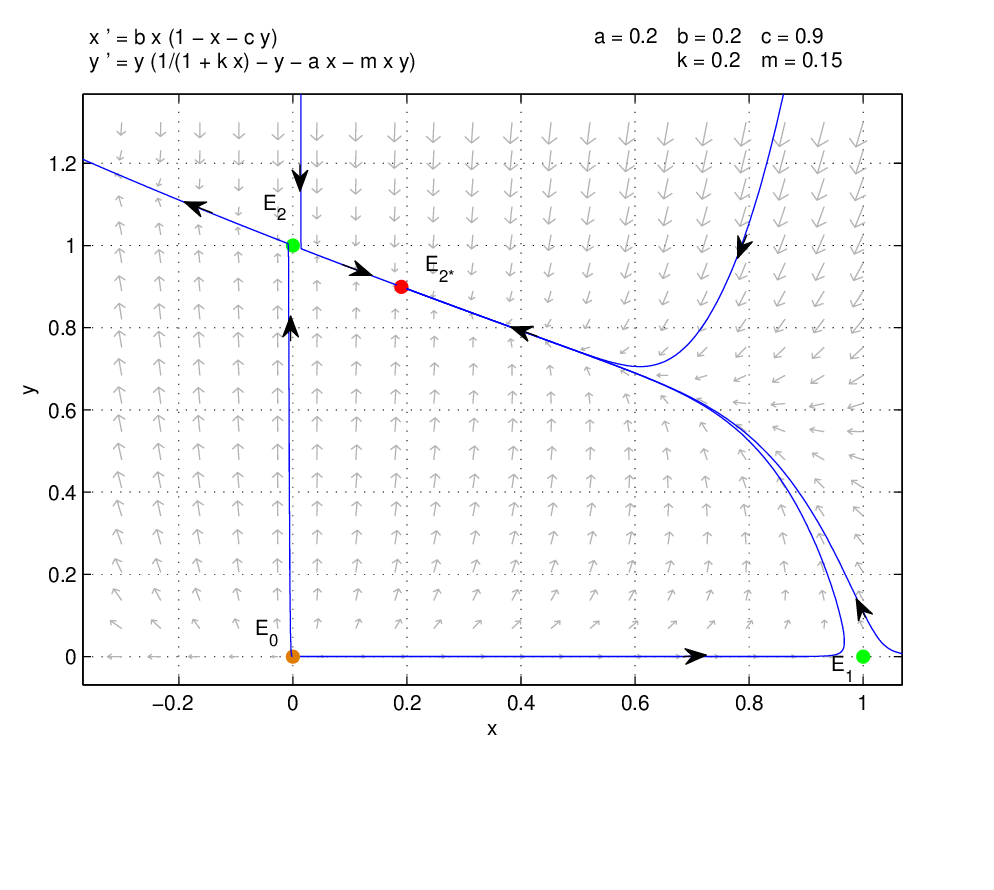}}}
\subfigure[ $c=c_{TR}$]
{\scalebox{0.45}[0.45]{
\includegraphics{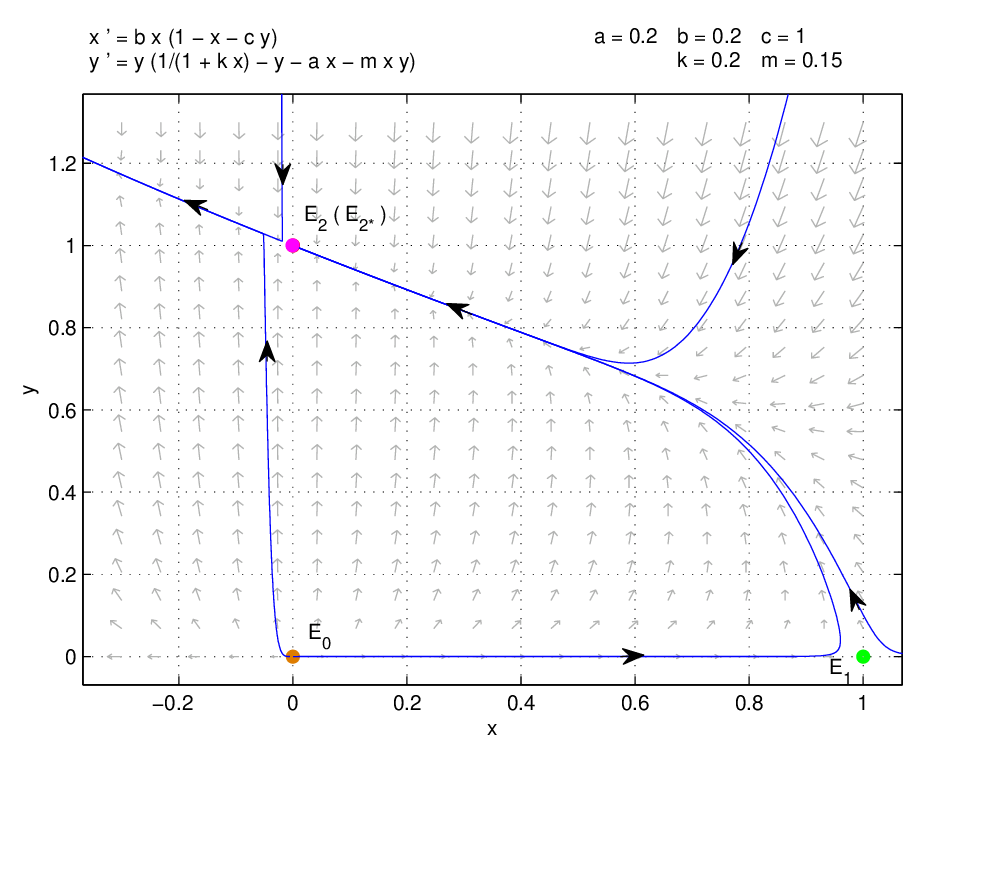}}}
\subfigure[ $c>c_{TR}$]
{\scalebox{0.45}[0.45]{
\includegraphics{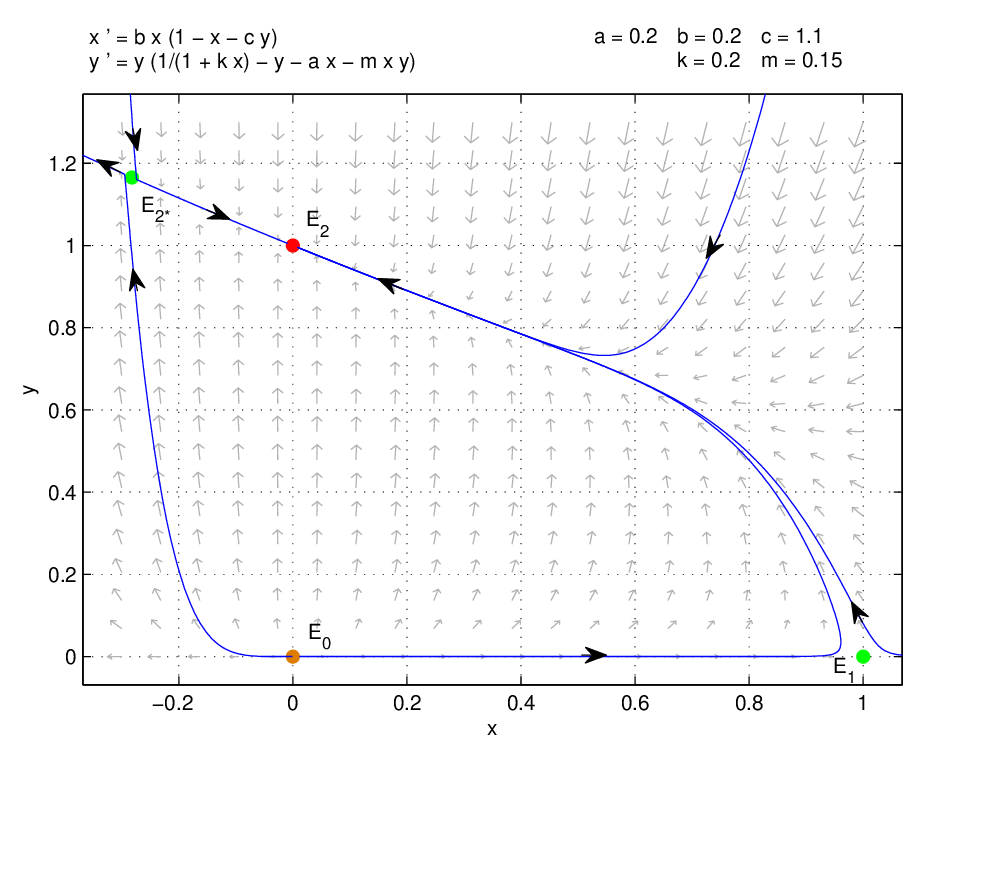}}}
\caption{Red, green, pink, and orange points indicate stable node, saddle, saddle-node, and unstable node (source), respectively. System \eqref{2} undergoes a transcritical bifurcation around $E_2$.}
\end{figure}
\begin{proof}
From Theorem 4, we know that the eigenvalues of $J(E_1)$ are $\lambda_{1}^{E_2}=-1$, $\lambda_{2}^{E_2}=0$ if $c=c_{TR}=1$. Now, let $\mathbf{V_2} = (v_3, v_4)^T$ and $\mathbf{W_2} = (w_3, w_4)^T$ be the eigenvectors of $J(E_2)$ and $J^T(E_2)$ corresponding to $\lambda_{2}^{E_2}=0$, respectively. By calculating, we obtain
\begin{equation}
\mathbf{V_1}=\begin{bmatrix}
v_3\\v_4

\end{bmatrix} =\begin{bmatrix}
-\f{1}{a+k+m} \vspace{2ex}\\1

\end{bmatrix},\mathbf{W_1}=\begin{bmatrix}
w_3\\w_4

\end{bmatrix} =\begin{bmatrix}
1 \\0
\end{bmatrix}.
\end{equation}
Furthermore,
$$
Q_c(E_2;c_{TR})=\begin{bmatrix}
\displaystyle\frac{\partial F}{\partial c} \vspace{2ex}\\\displaystyle\frac{\partial G}{\partial c}

\end{bmatrix}=\begin{bmatrix}
0\\0

\end{bmatrix},
$$
$$
\left. DQ_c(E_2;c_{TR})\mathbf{V_2}=\left[\begin{array}{cc}
-by & -bx
\\
0 & 0
\end{array}\right] \right |_{(E_2;c_{TR})}\begin{bmatrix}
-\f{1}{a+k+m}\vspace{2ex} \\1

\end{bmatrix}=\begin{bmatrix}
\f{b}{a+k+m}\vspace{2ex} \\0
\end{bmatrix},
$$
$$
\left.D^2Q(E_2;c_{TR})(\mathbf{V_2}, \mathbf{V_2})=\begin{bmatrix}
\displaystyle\frac{\partial^2F}{\partial x^2}v_3^2+ 2\displaystyle\frac{\partial^2F}{\partial x\partial y}v_3v_4+ \displaystyle\frac{\partial^2F}{\partial y^2}v^2_4\vspace{2ex}\\
\displaystyle\frac{\partial^2G}{\partial x^2}v_3^2+ 2\displaystyle\frac{\partial^2G}{\partial x\partial y}v_3v_4+ \displaystyle\frac{\partial^2G}{\partial y^2}v^2_4
\end{bmatrix}\right|_{(E_2;c_{TR})}=\begin{bmatrix}
\f{2 b \left(-1+a +k +m \right)}{\left(a +k +m \right)^{2}} \vspace{2ex}\\ \f{2 m^{2}+\left(2 a +2 k \right) m +2 k^{2}}{\left(a +k +m \right)^{2}}
\end{bmatrix}.
$$
Thus, we have
$$
\mathbf{W_2}^TQ_c(E_2;c_{TR})=0,
$$
$$
\mathbf{W_2}^T\left[ DQ_c(E_2;c_{TR})\mathbf{V_2} \right]=\f{b}{a+k+m}\ne0,
$$
$$
\mathbf{W_2}^T\left[ D^2Q(E_2;c_{TR})(\mathbf{V_2}, \mathbf{V_2}) \right]=\f{2 b \left(-1+a +k +m \right)}{\left(a +k +m \right)^{2}}\ne0.
$$
Based on {\it Sotomayor's Theorem} \cite{Perko13} , all the transversality conditions for system \eqref{2} to experience a transcritical bifurcation are satisfied, so system \eqref{2} undergoes a transcritical bifurcation around $E_2$ at the bifurcation parameter threshold $c_{TR} = 1$.
\end{proof}

\begin{figure}[h]
\centering
\subfigcapskip=-25pt
\subfigure[$0<c<c_{TR}$]
{\scalebox{0.45}[0.45]{
\includegraphics{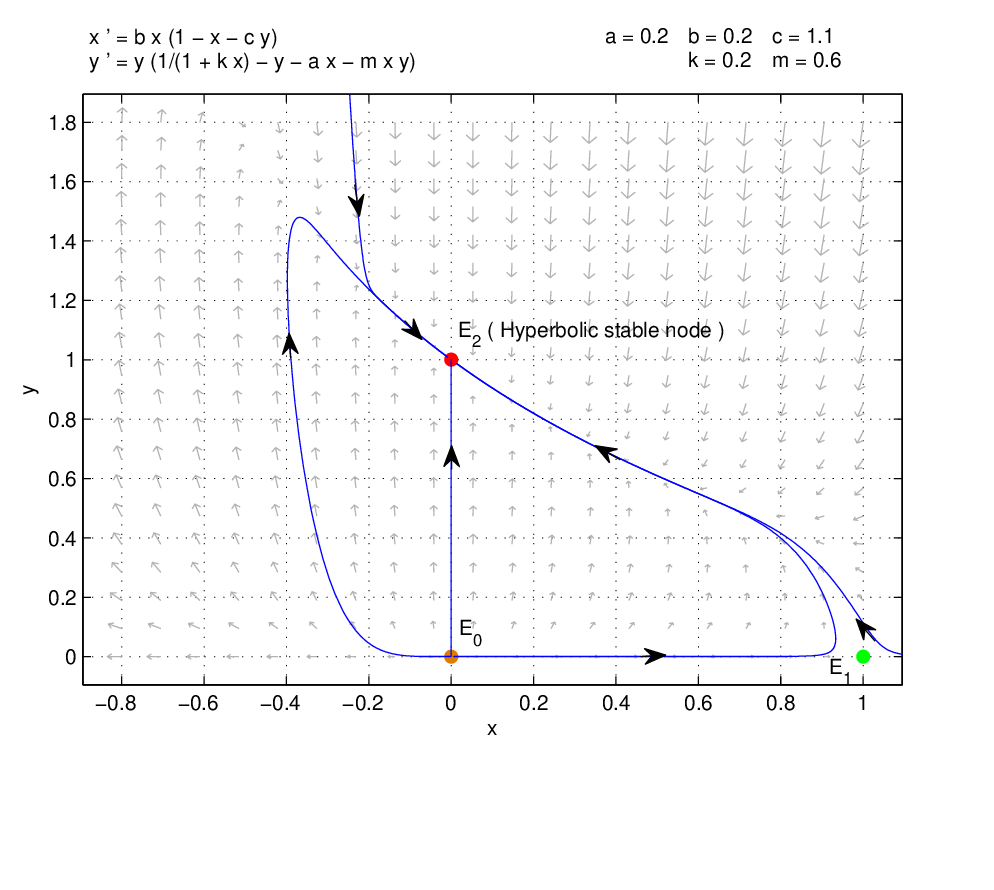}}}
\subfigure[ $c=c_{TR}$]
{\scalebox{0.45}[0.45]{
\includegraphics{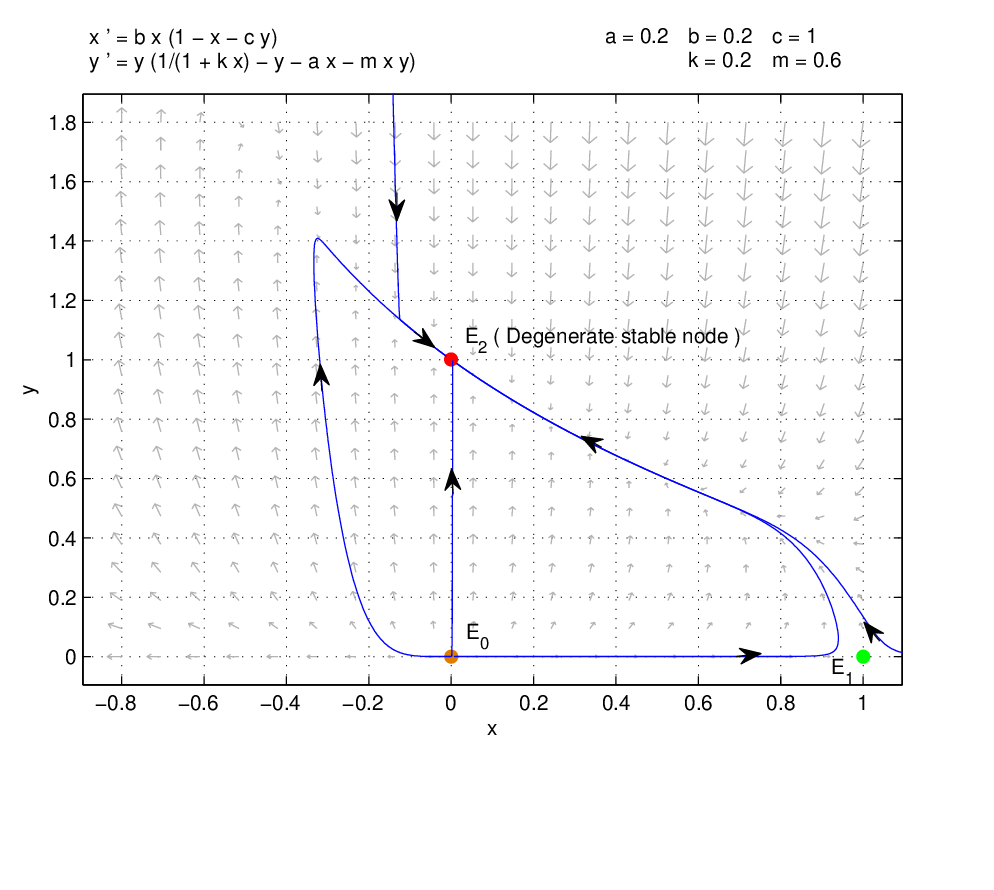}}}
\subfigure[ $c>c_{TR}$]
{\scalebox{0.45}[0.45]{
\includegraphics{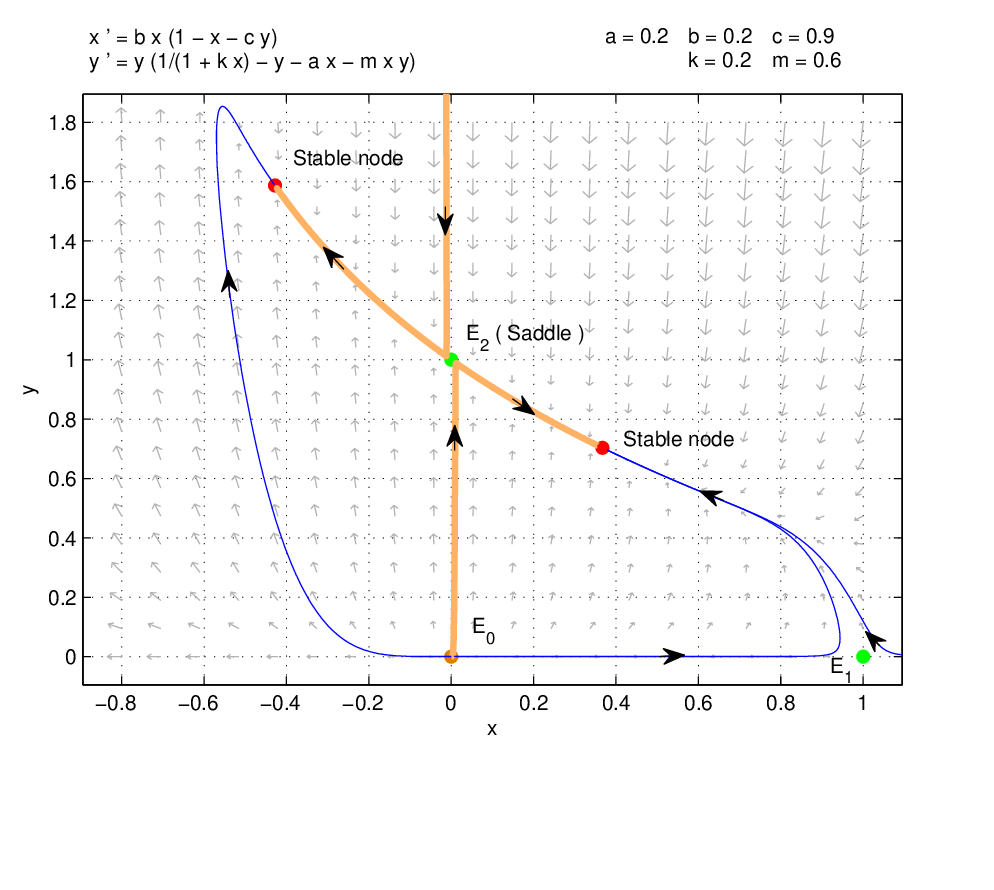}}}
\caption{System \eqref{2} undergoes a pitchforkbifurcation around $E_2$ when $c=c_{TR}$ and $m=1-a-k$.}
\end{figure}

\subsection{Pitchfork bifurcation}

According to Theorem 9, the third transversality condition about transcritical bifurcation on $E_2$ will equal $0$ when $m=1-a-k$, i.e., $m=m^{**}$. Also by Theorem 4, $E_2$ is a degenerate stable node when $m=m^{**}$. We select $a=0.2$, $b=0.2$, $k=0.2$, $m=0.6$, $c=1 \pm 0.1$. By numerical simulation, we find that the number of equilibria near $E_2$ undergoes a $1-1-3$ transformation. From this we conclude that system \eqref{2} will experience a pitchfork bifurcation around $E_2$ when $c=c_{TR}$ and $m=m^{**}$ (Fig. 7).

\subsection{Saddle-node bifurcation}
Under the condition $m>m_1$, $c>1$, and $0<k<k^*$, we note that when $u(E)>0$, $u(E)=0$, and $u(E)<0$, system \eqref{2} has 0, 1, and 2 positive equilibria, respectively. Therefore we consider system \eqref{2} undergoing a saddle-node bifurcation around the positive equilibrium point $E_{3*}$. We selected the toxic release rate $m$ as the bifurcation parameter. By calculating $u(E)=v(E)=0$, we obtain the bifurcation parameter threshold $m=\frac{-E^{2} k^{2}+2 E c k -2 E k +c -1}{E^{2} \left(E^{2} k^{2}+2 E k +1\right)}\triangleq m_{SN}$, and $a=\frac{E^{4} k^{2}-2 E^{3} k^{2}+2 E^{3} k +3 E^{2} c k +E^{2} k^{2}-4 E^{2} k -2 E c k +E^{2}+2 E c +2 E k -2 E -c +1}{c \,E^{2} \left(E^{2} k^{2}+2 E k +1\right)}\triangleq a_{1}$. Next, we use {\it Sotomayor's Theorem} \cite{Perko13}  to verify that the transversality conditions for saddle-node bifurcation are satisfied.
\begin{figure}[h]
\centering
\subfigcapskip=-25pt
\subfigure[$u(E)<0$]
{\scalebox{0.45}[0.45]{
\includegraphics{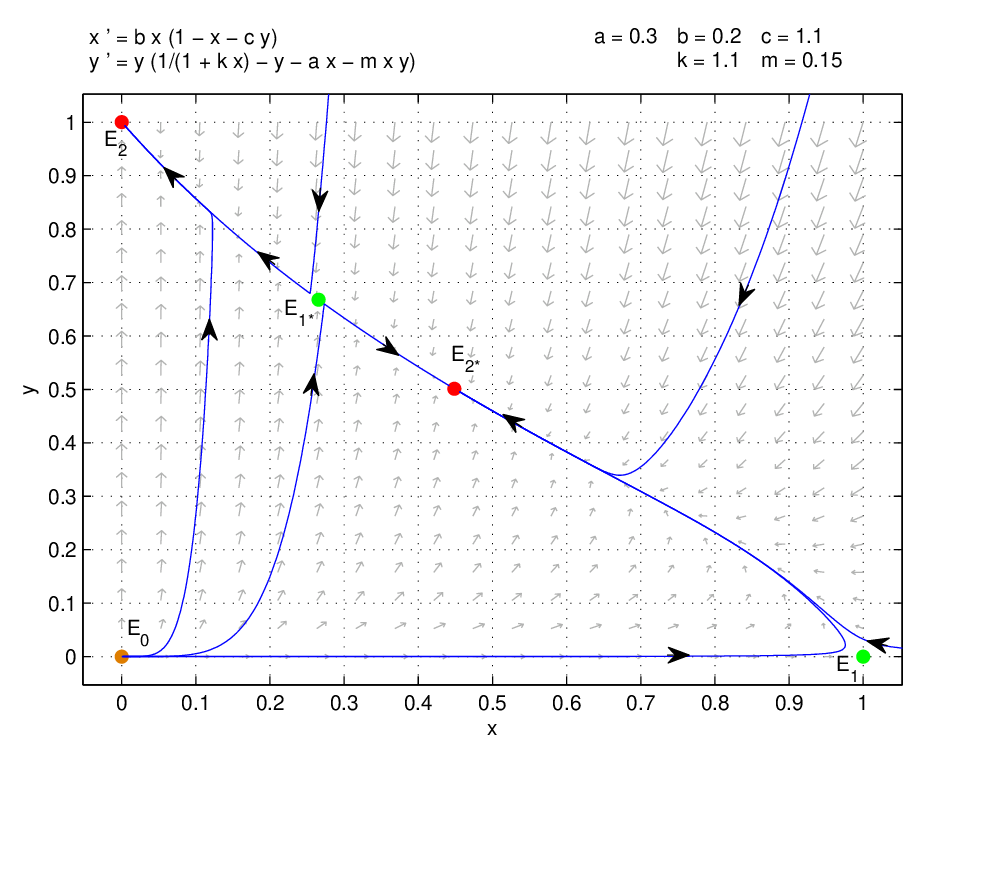}}}
\subfigure[ $u(E)=0$]
{\scalebox{0.45}[0.45]{
\includegraphics{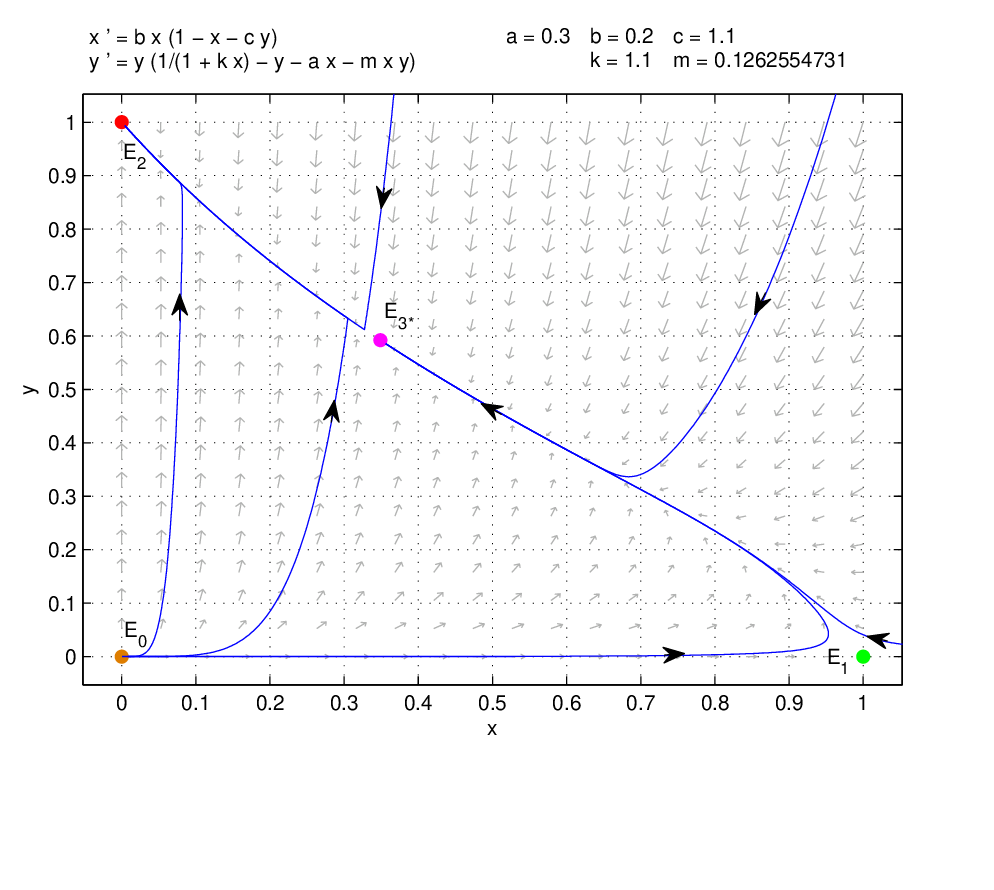}}}
\subfigure[ $u(E)>0$]
{\scalebox{0.45}[0.45]{
\includegraphics{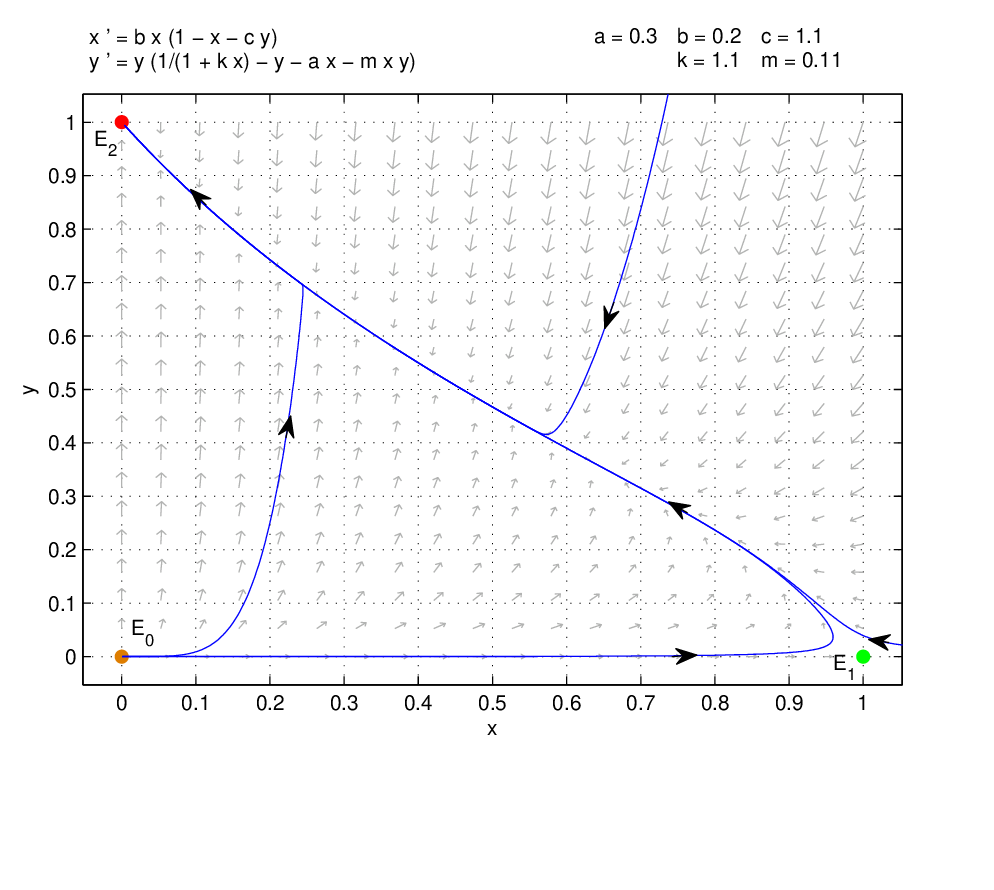}}}
\caption{Red, green, pink, and orange points indicate stable node, saddle, saddle-node, and unstable node (source), respectively. System \eqref{2} undergoes a saddle-node bifurcation around $E_{3*}$. }
\end{figure}
\begin{theorem}
System \eqref{2} undergoes a saddle-node bifurcation around $E_{3*}$ at the bifurcation parameter threshold $m = m_{SN}$ when $m>m_1$, $0<k<k^*$, $c>1$, and $c \ne\frac{E^{3} k^{3}+3 E^{2} k^{2}+3 E k +1}{3 E^{2} k^{2}+3 E k +1}$ (Fig. 8, Fig. 9).
\end{theorem}
\begin{figure}[h]
\centering
\scalebox{0.6}[0.6]{
\includegraphics{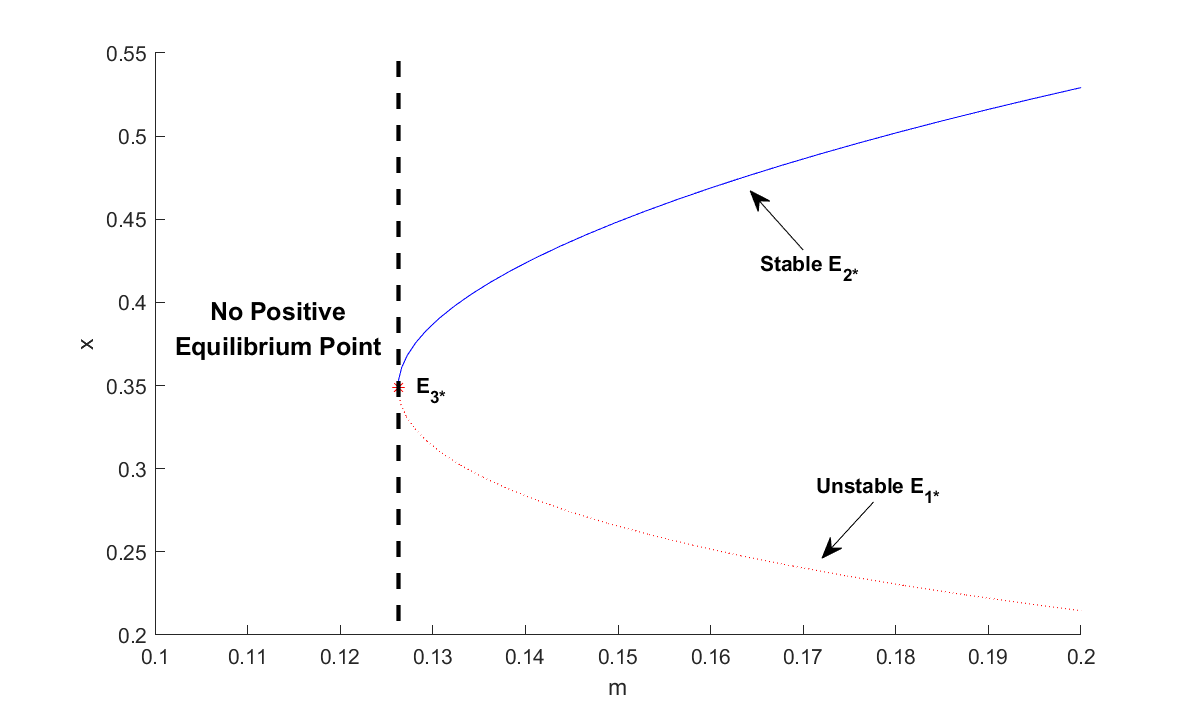}}
\caption{Select $a=0.3$, $b=0.2$, $c=1.1$, $k=1.1$, $m=m_{SN}=0.1262554731$, system \eqref{2} undergoes a saddle-node bifurcation around $E_{3*}$.}
\label{sn}
\end{figure}
\begin{proof}
According to \eqref{8}, we know that the Jacobi matrix of the positive equilibrium point $E_{3*}$ can be expressed in the following form by substituting $m=m_{SN}$, $a=a_1$, and one of the eigenvalues is $\lambda = 0$.
$$
J(E_{3*})=\left[\begin{array}{cc}
-bE &-bEc
\vspace{2ex}\\
\frac{\left(E^{3} k^{2}+\left(-k^{2}+2 k \right) E^{2}+\left(1+\left(2 c -2\right) k \right) E +c -1\right) \left(-1+E \right)}{E \left(E k +1\right)^{2} c^{2}} & \frac{\left(E^{3} k^{2}+\left(-k^{2}+2 k \right) E^{2}+\left(1+\left(2 c -2\right) k \right) E +c -1\right) \left(-1+E \right)}{c E \left(E k +1\right)^{2}}
\end{array}\right].
$$
Now, let $\mathbf{V_3} = (v_5, v_6)^T$ and $\mathbf{W_3} = (w_5, w_6)^T$ be the eigenvectors of $J(E_{3*})$ and $J^T(E_{3*})$ corresponding to $\lambda=0$, respectively. By calculating, we obtain
\begin{equation}
\mathbf{V_3}=\begin{bmatrix}
v_5\\v_6

\end{bmatrix} =\begin{bmatrix}
-c \\1

\end{bmatrix},\mathbf{W_3}=\begin{bmatrix}
w_5\\w_6

\end{bmatrix} =\begin{bmatrix}
\frac{\left(E^{3} k^{2}+\left(-k^{2}+2 k \right) E^{2}+\left(1+\left(2 c -2\right) k \right) E +c -1\right) \left(-1+E \right)}{E^{2} b \,c^{2} \left(E k +1\right)^{2}}\vspace{2ex} \\1
\end{bmatrix}.
\label{23}
\end{equation}
Furthermore,
$$
Q_m(E_{3*};m_{SN})=\begin{bmatrix}
\displaystyle\frac{\partial F}{\partial m} \vspace{2ex}\\\displaystyle\frac{\partial G}{\partial m}

\end{bmatrix}=\begin{bmatrix}
0\vspace{2ex}\\-\f{\left(1-E \right)^{2} E}{c^{2}}
\end{bmatrix},
$$
$$
D^2Q(E_{3*};m_{SN})(\mathbf{V_3}, \mathbf{V_3})=\begin{bmatrix}
\displaystyle\frac{\partial^2F}{\partial x^2}v_5^2+ 2\displaystyle\frac{\partial^2F}{\partial x\partial y}v_5v_6+ \displaystyle\frac{\partial^2F}{\partial y^2}v^2_6\vspace{2ex}\\
\displaystyle\frac{\partial^2G}{\partial x^2}v_5^2+ 2\displaystyle\frac{\partial^2G}{\partial x\partial y}v_5v_6+ \displaystyle\frac{\partial^2G}{\partial y^2}v^2_6
\end{bmatrix}=\begin{bmatrix}
0 \vspace{2ex}\\\frac{2 \left(-1+E \right) \left(E^{3} k^{3}-3 k^{2} \left(c -1\right) E^{2}-3 k \left(c -1\right) E -c +1\right)}{\left(E k +1\right)^{3} E^{2}}
\end{bmatrix},
$$
Thus, we have
$$
\mathbf{W_3}^TQ_m(E_{3*};m_{SN})=-\f{\left(1-E \right)^{2} E}{c^{2}}\ne0,
$$
$$
\mathbf{W_3}^T\left[ D^2Q(E_{3*};q_{SN})(\mathbf{V_3}, \mathbf{V_3}) \right]=\frac{2 \left(-1+E \right) \left(E^{3} k^{3}-3 k^{2} \left(c -1\right) E^{2}-3 k \left(c -1\right) E -c +1\right)}{\left(E k +1\right)^{3} E^{2}}\ne0.
$$
According to {\it Sotomayor's Theorem} \cite{Perko13} , all the transversality conditions for system \eqref{2} to experience a saddle-node bifurcation are satisfied, so system \eqref{2} undergoes a saddle-node bifurcation around $E_{3*}$ at the bifurcation parameter threshold $m_{SN} = \frac{2 \left(-1+E \right) \left(E^{3} k^{3}-3 k^{2} \left(c -1\right) E^{2}-3 k \left(c -1\right) E -c +1\right)}{\left(E k +1\right)^{3} E^{2}}$.
\end{proof}

\begin{remark}
This section discusses all possible bifurcations of system \eqref{2}. Through the above analysis, we find that by varying the value of the fear effect $k$ for non-toxic species or the interspecific competition rate $c$ for toxic species, both cause system \eqref{2} to undergo a transcritical bifurcation on the boundary. When a particular value is taken for the toxin release rate $m$, this may also cause system \eqref{2} to experience a pitchfork bifurcation on the boundary. In addition, parameter $m$ will also lead system \eqref{2} to undergo a saddle-node bifurcation in the first quadrant. Thus we can determine that the fear effect and toxic release rate can cause complex dynamics in the classical Lotka-Volterra competition model.
\end{remark}

\section{Effect of Toxic Release Rate and Fear}
Through the studies in Section 6, we learned that the toxic release rate $m$ and the fear effect $k$ produce rich bifurcations in system \eqref{2}. Then returning to the biological significance, how exactly do $m$ and $k$ affect the species? Observing Table 1, we note that whenever $k$ falls in the interval $(0,k^*)$, there must be a stable positive equilibrium point $E_{2*}$ in system \eqref{2}. Therefore we conclude that, regardless of the value of the toxic release rate, the only factor that can affect the survival of non-toxic species is the competition fear. Next, we use numerical simulation to verify this through the time-course plots of solutions.
\begin{example}
For $m>m_1$ and $0<c<1$. We select $a=0.8$, $b=0.5$, $c=0.5$, $k_1=0.2$, $k_2=0.4$, $m=0.5$. Through numerical simulations, we obtain time-course plots of solutions (Fig. 10). When $m>m_1$, the value of fear effect $k$ leads to the extinction of the non-toxic species $(x)$ if it satisfies $k>k^*$, and is not in contrast.
\begin{figure}[h]
\centering
\subfigure[$0<k_1<k^*$]
{\scalebox{0.45}[0.45]{
\includegraphics{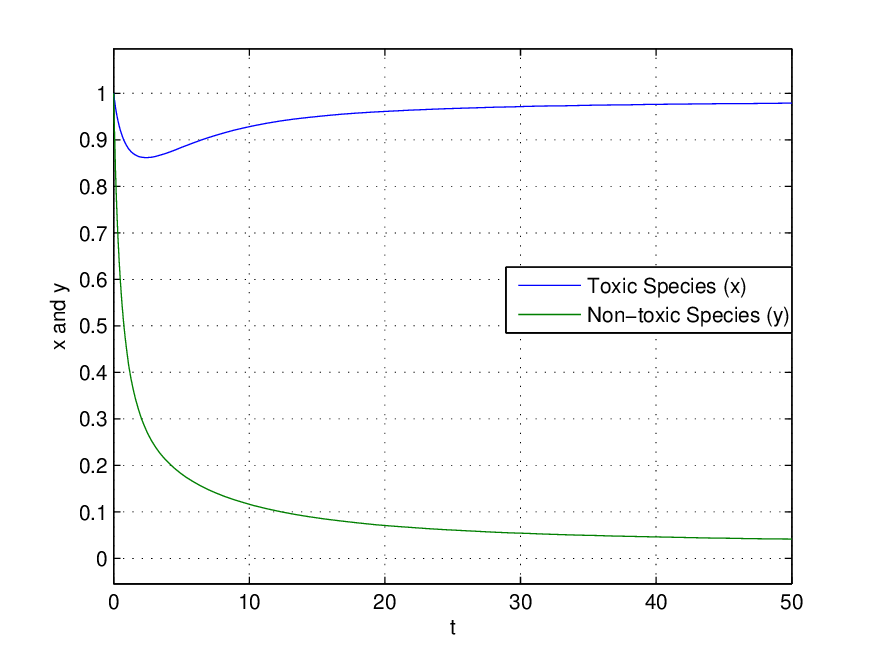}}}
\subfigure[ $k_2>k^*$]
{\scalebox{0.45}[0.45]{
\includegraphics{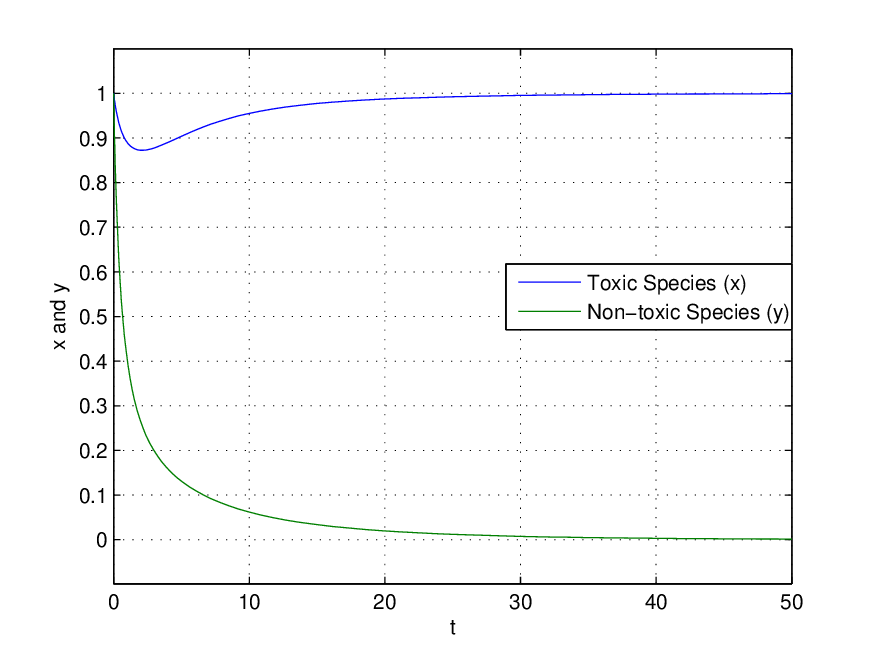}}}
\caption{$a=0.8$, $b=0.5$, $c=0.5$, $k_1=0.2$, $k_2=0.4$, $m=0.5$. (a) Non-toxic species survives when $0<k_1<k^*$. (b) Non-toxic species become extinct when $k_2>k^*$. }
\end{figure}
\end{example}
\begin{example}
For $m>m_1$ and $c>1$. We select $a=0.3$, $b=0.5$, $c=1.1$, $k_1=1.1$, $k_2=4$, $m=0.15$. Through numerical simulations, we obtain time-course plots of solutions (Fig. 11). We note that although toxic species are subject to an interspecific competition rate $c > 1$, they still survive by releasing toxins and by causing fear in competitor.
\begin{figure}[h]
\centering
\subfigure[$0<k_1<k^*$]
{\scalebox{0.45}[0.45]{
\includegraphics{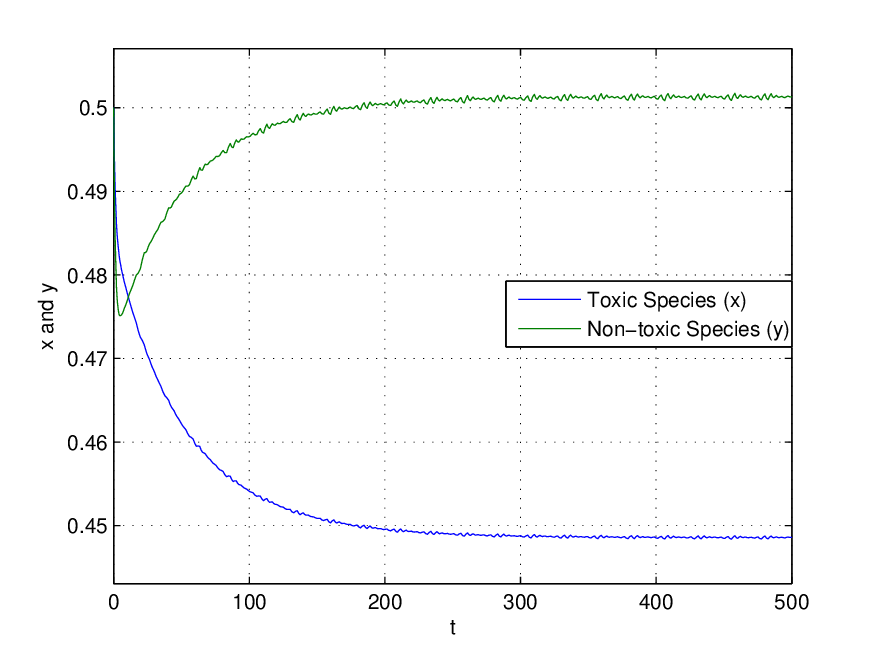}}}
\subfigure[ $k_2>k^*$]
{\scalebox{0.45}[0.45]{
\includegraphics{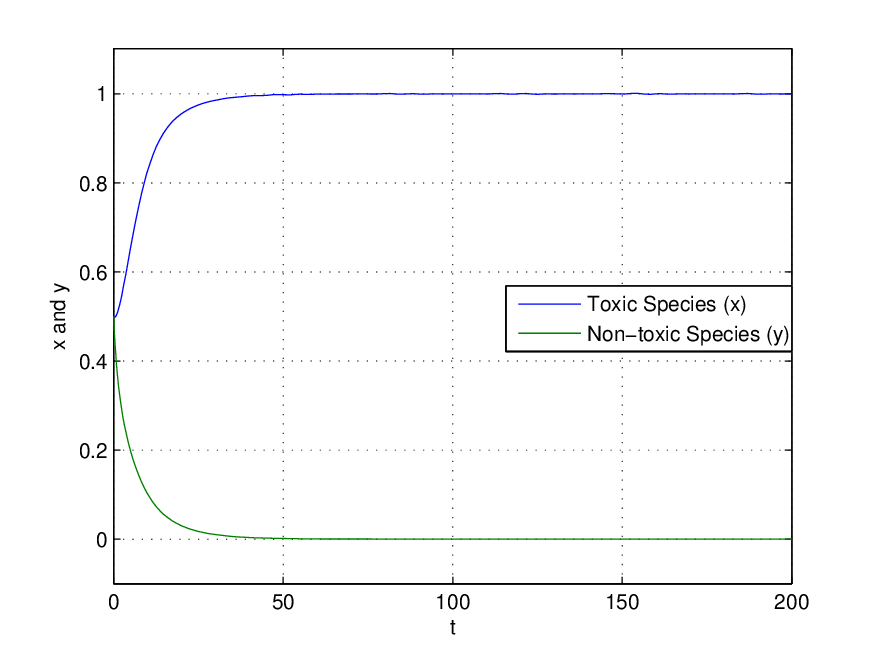}}}
\caption{$a=0.3$, $b=0.5$, $c=1.1$, $k_1=1.1$, $k_2=4$, $m=0.15$.  (a) Non-toxic species survives when $0<k_1<k^*$. (b) Non-toxic species become extinct when $k_1>k^*$. }
\end{figure}
\end{example}
\begin{example}
For $0<m<m_1$. We select $a=0.8$, $b=0.5$, $c=0.5$, $k_1=0.2$, $k_2=0.3$, $m=0.1$. Through numerical simulations, we obtain time-course plots of solutions (Fig. 12). 
\begin{figure}[h]
\centering
\subfigure[$0<k_1<k^*$]
{\scalebox{0.45}[0.45]{
\includegraphics{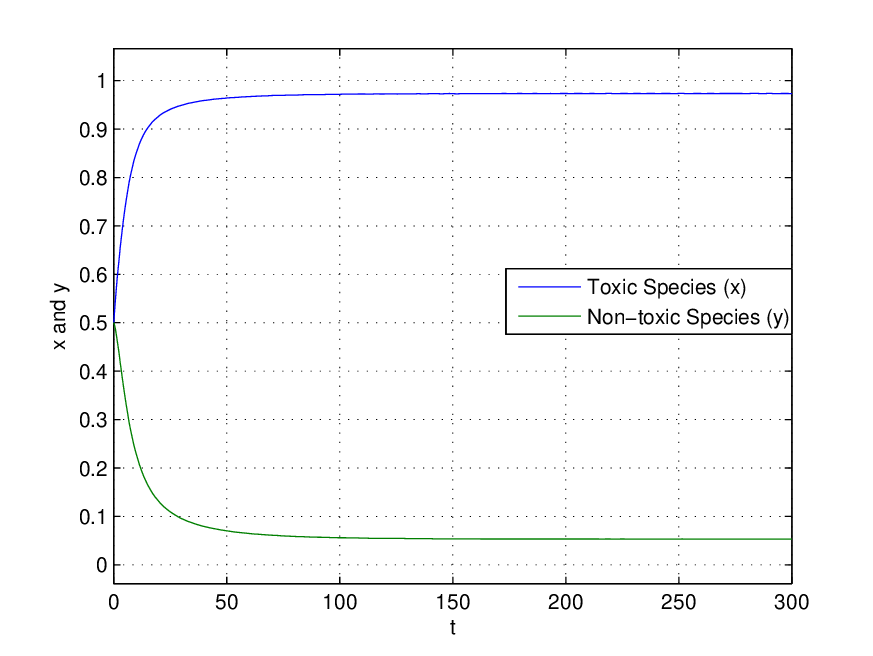}}}
\subfigure[ $k_2>k^*$]
{\scalebox{0.45}[0.45]{
\includegraphics{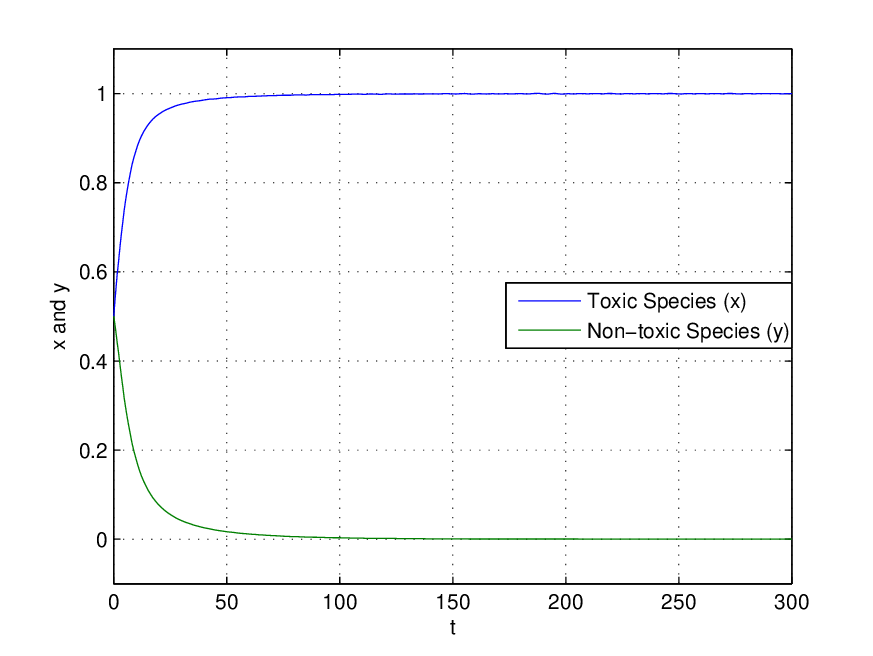}}}
\caption{$a=0.8$, $b=0.5$, $c=0.5$, $k_1=0.2$, $k_2=0.3$, $m=0.1$. (a) Non-toxic species survive when $0<k_1<k^*$. (b) Non-toxic species become extinct when $k_2>k^*$. }
\end{figure}
\end{example}
\begin{remark}
Comparing Example 7.1-3, we find that non-toxic species can survive regardless of the level of toxic release rate. When the fear effect on the non-toxic species is too large, it leads to extinction. Numerical simulation effectively verifies the correctness of our above analysis.
\end{remark}

\section{The PDE Case}
We will now cover several preliminary concepts that will pave the way for proving the global existence of solutions to \eqref{PDEmodel}. To achieve this objective, it is sufficient to establish a uniform estimate on the $\mathbb{L}^{p}$ norms of the right-hand side of \eqref{PDEmodel}, where $p$ exceeds $\frac{n}{2}$. By doing so, we can then apply classical theory, as outlined in \cite{henry}, to guarantee global existence.

In this context, the standard norms in the spaces $\mathbb{L}^{p}(\Omega)$, $\mathbb{L}^{\infty}(\Omega)$, and $\mathbb{C}(\overline{\Omega})$ are denoted as follows:

\begin{equation}
	\label{(2.2)}
	\left\| u\right\| _{p}^{p} = 
	\int_{\Omega }\left| u(x)\right|^{p}dx, \ \left\| u\right\| _{\infty }\text{=}\underset{x\in \Omega }{ess \sup}\left|
	u(x)\right| .  
\end{equation}

We employ well-established techniques, as described in \cite{Morgan89}, to accomplish this. Initially, let's revisit classical results that ensure the non-negativity of solutions and establish both local and global existence, as outlined in \cite{P10} and \cite{Morgan89}:

\begin{lemma}\label{lem:class1}
	Let us consider the following $m\times m$ - reaction diffusion system: for all $i=1,...,m,$ 
	\begin{equation}
		\label{eq:class1}
		\partial_t u_i-d_i\Delta u_i=f_i(u_1,...,u_m)~in~ \mathbb{R}_+\times \Omega,~ \partial_\nu u_i=0~ \text{on}~ \partial \Omega, u_i(0)=u_{i0},
	\end{equation}
	where $d_i \in(0,+\infty)$, $f=(f_1,...,f_m):\mathbb{R}^m \rightarrow \mathbb{R}^m$ is $C^1(\Omega)$ and $u_{i0}\in L^{\infty}(\Omega)$. Then there exists a $T>0$ and a unique classical solution of (\ref{eq:class1}) on $[0,T).$ If $T^*$ denotes the greatest of these $T's$, then 
	\begin{equation*}
		\Bigg[\sup_{t \in [0,T^*),1\leq i\leq m} ||u_i(t)||_{L^{\infty}(\Omega)} < +\infty \Bigg] \implies [T^*=+\infty].
	\end{equation*}
	If the nonlinearity $(f_i)_{1\leq i\leq m}$ is moreover quasi-positive, which means 
	$$\forall i=1,..., m,~~\forall u_1,..., u_m \geq 0,~~f_i(u_1,...,u_{i-1}, 0, u_{i+1}, ..., u_m)\geq 0,$$
	then $$[\forall i=1,..., m, u_{i0}\geq 0]\implies [\forall i=1,...,m,~ \forall t\in [0,T^*), u_i(t)\geq 0].$$
\end{lemma}

\begin{lemma}\label{lem:class2}
	Using the same notations and hypotheses as in Lemma \ref{lem:class1}, suppose moreover that $f$ has at most polynomial growth and that there exists $\mathbf{b}\in \mathbb{R}^m$ and a lower triangular invertible matrix $P$ with nonnegative entries such that  $$\forall r \in [0,+\infty)^m,~~~Pf(r)\leq \Bigg[1+ \sum_{i=1}^{m} r_i \Bigg]\mathbf{b}.$$
	Then, for $u_0 \in L^{\infty}(\Omega, \mathbb{R}_+^m),$ the system (\ref{eq:class1}) has a strong global solution.
\end{lemma}

Under these assumptions, the following local existence result is well known, see \cite{henry}.

\begin{theorem}
	\label{thm:class3}
	The system (\ref{eq:class1}) admits a unique, classical solution $(u,v)$ on $%
	[0,T_{\max }]\times \Omega $. If $T_{\max }<\infty $ then 
	\begin{equation}
		\underset{t\nearrow T_{\max }}{\lim }\Big\{ \left\Vert u(t,.)\right\Vert
		_{\infty }+\left\Vert v(t,.)\right\Vert _{\infty } \Big\} =\infty ,  
	\end{equation}%
	where $T_{\max }$ denotes the eventual blow-up time in $\mathbb{L}^{\infty }(\Omega ).$
\end{theorem}

The next result follows from the application of standard theory \cite{kish88}.

\begin{theorem}
	\label{thm:km1}
	Consider the reaction-diffusion system (\ref{eq:class1}). For spatially homogenous initial data $u_{0} \equiv c, v_{0} \equiv d$, with $c,d>0$, then the dynamics of (\ref{eq:class1}) and its resulting kinetic (ODE) system, when $d_{1}=d_{2}=0$ in (\ref{eq:class1}), are equivalent.
\end{theorem}

Our current aim is to explore the scenario where the fear function exhibits spatial heterogeneity. This perspective finds motivation in various ecological and sociological contexts. For instance, it is quite common for prey to exhibit higher fear levels in proximity to a predator's lair but lower fear levels in regions of refuge, as mentioned in \cite{Zhang19}. Additionally, areas with high population density may lead to reduced fear due to group defense mechanisms, as discussed in \cite{Samsal20}. Given these considerations, it is plausible to assume that the fear coefficient $k$ is not a constant but varies across the spatial domain $\Omega$, i.e., $k=k(x)$. The specific form of $k(x)$ may differ depending on the particular application, aligning with the concept of the Landscape of Fear (LOF) \cite{Brown99}. Consequently, we now consider the following spatially explicit version of \eqref{2}, featuring a heterogeneous fear function $k(x)$, which results in the following reaction-diffusion system:

\begin{equation}\label{PDEmodel}
	\left\{\begin{array}{l}
		u_t =d_1 \Delta u + bu\left( 1-u-cv \right), \quad x \in  \Omega, \vspace{2ex}\\
		v_t =d_2 \Delta v +v\left( \f{1}{1+k(x)u}-v-au-muv \right), \quad x \in  \Omega, \vspace{2ex}\\
		\dfrac{\partial u}{\partial \nu} = \dfrac{\partial v}{\partial \nu} =0, \quad \text{on} \quad \partial \Omega. \vspace{2ex}\\
		u(x,0)=u_0(x)\equiv c >0, \quad v(x,0)=v_0(x) \equiv d>0,
	\end{array}\right.
\end{equation}

Furthermore, we impose the following restrictions on the fear function $k(x)$,

\begin{align}\label{eq:as1}
	\begin{split}
		&(i) \quad k(x)  \in C^{1}(\Omega),
		\\
		&(ii) \quad k(x) \geq 0,
		\\
		& (iii)\quad  \mbox{If} \  k(x) \equiv 0 \ \mbox{on}  \ \Omega_{1} \subset \Omega, \ \mbox{then} \ |\Omega_{1}| = 0.
		\\
		& (iv)\quad  \mbox{If} \  k(x) \equiv 0 \ \mbox{on}  \ \cup^{n}_{i=1}\Omega_{i} \subset \Omega, \ \mbox{then} \ \Sigma^{n}_{i=1}|\Omega_{i}| = 0.
	\end{split}
\end{align}

\begin{remark}
	If $k(x) \equiv 0$ on $\Omega_{1} \subset \Omega$, with $|\Omega_{1}| > \delta > 0$, or $q(x) \equiv 0$ on $\cup^{n}_{i=1}\Omega_{i} \subset \Omega$, with $\Sigma^{n}_{i=1}|\Omega_{i}| > \delta > 0$, that is, on non-trivial parts of the domain, the analysis is notoriously difficult, as one now is dealing with a \emph{degenerate} problem. See \cite{Du02a,Du02b} for results on this problem. This case is not in the scope of the current manuscript. 
\end{remark}
Since the nonlinear right hand side of (\ref{PDEmodel}) is continuously
differentiable on $\mathbb{R}^{+}\times $ $\mathbb{R}^{+}$, then for any
initial data in $\mathbb{C}\left( \overline{\Omega }\right) $ or $\mathbb{L}%
^{p}(\Omega ),\;p\in \left( 1,+\infty \right) $, it is standard to 
estimate the $\mathbb{L}^{p}-$norms of the solutions and thus deduce global existence. The standard theory will apply even in the case of a bonafide fear function $k(x)$ because due to our assumptions on the form of $k$,  standard comparison arguments will apply \cite{Gil77}. Thus applying the classical methods above, via Theorem \ref{thm:class3}, and Lemmas \ref{lem:class1}-\ref{lem:class2}, we can state the following lemmas:

\begin{lemma}
	\label{lem:pos1}
	Consider the reaction-diffusion system \eqref{PDEmodel}, for $k(x)$ such that the assumptions via \eqref{eq:as1} hold. Then, the solutions to \eqref{PDEmodel} are non-negative as long as they initiate from positive initial conditions.
\end{lemma}

\begin{lemma}
	\label{lem:cl1}
	Consider the reaction-diffusion system \eqref{PDEmodel}. For $k(x)$ such that the assumptions via \eqref{eq:as1} hold. The solutions to \eqref{PDEmodel} are classical. That is for $(u_{0},v_{0}) \in \mathbb{L}^{\infty }(\Omega )$,  $(u,v) \in C^{1}(0,T; C^{2}(\Omega))$, $\forall T$.
\end{lemma}
Our goal in this section is to investigate the dynamics of \eqref{PDEmodel}. Herein, we will use the comparison technique and compare it to the ODE cases of classical competition or the constant fear function case, where the dynamics are well known. 

\begin{remark}
	This section's analysis primarily focuses on the choice of spatially homogenous (flat) initial data.
\end{remark}

Let's define some PDE systems,

\begin{align}\label{eq:lv_model}
	\begin{split}
		\overline{u}_t &= d_1 \overline{u}_{xx} + b \overline{u}\left( 1-\overline{u}-c \overline{v} \right) , \\
		\overline{v}_t &= d_2 \overline{v}_{xx} +\overline{v}\left( 1-v-au-muv \right),
	\end{split}
\end{align}

\begin{align}\label{eq:upper}
	\begin{split}
		\widehat{u}_t &= d_1 \widehat{u}_{xx} + b\widehat{u}\left( 1-\widehat{u}-c\widehat{v} \right) , \\
		\widehat{v}_t &= d_2 \widehat{v}_{xx} +\widehat{v}\left( \f{1}{1+ \mathbf{\widehat{k}}\widehat{u}}-\widehat{v}-a\widehat{u}-m\widehat{u}\widehat{v} \right),
	\end{split}
\end{align}

\begin{align}\label{eq:lower}
	\begin{split}
		\widetilde{u}_t &= d_1 \widetilde{u}_{xx} + bu\left( 1-\widetilde{u}-c\widetilde{v} \right) , \\
		\widetilde{v}_t &= d_2 \widetilde{v}_{xx} + \widetilde{v}\left( \f{1}{1+ \mathbf{\widetilde{k}}\widetilde{u}}-\widetilde{v}-a\widetilde{u}-m\widetilde{u} \widetilde{v} \right),
	\end{split}
\end{align}

\begin{align}\label{eq:lowest}
	\begin{split}
		\tilde{u}_t &= d_1 \tilde{u}_{xx} + b\tilde{u}\left( 1-\tilde{u}-c\tilde{v} \right) , \\
		\tilde{v}_t &= d_2 \tilde{v}_{xx} + \tilde{v}\left( \f{1}{1+ \mathbf{\widetilde{k}}}-\tilde{v}-a\tilde{u}-m\tilde{u}\tilde{v} \right),
	\end{split}
\end{align}

where
\begin{align}\label{eq:lowest1}
	\mathbf{\widehat{k}} = \min_{x\in \Omega} k(x), \quad \, \quad   \mathbf{\widetilde{k} }= \max_{x\in \Omega} k(x).
\end{align}

We assume Neumann boundary conditions for all of the reaction diffusion systems \eqref{eq:lv_model}-\eqref{eq:lowest}. Also, we prescribe spatially homogenous (flat) initial conditions in each system:
$u(x,0)=u_0 (x) \equiv c>~0, \quad v(x,0)=v_0(x) \equiv d > 0.$

\begin{theorem}\label{lem:com1}
    For the reaction-diffusion system (\ref{PDEmodel}) of Allelopathic Phytoplankton with a fear function $k(x)$, as well as the reaction-diffusion systems (\ref{eq:upper})-(\ref{eq:lower}). Then the following point-wise comparison holds,
    \[ \widetilde{v} \le v \le \widehat{v}.\]
\end{theorem}

\begin{proof}
	From the positivity of the solutions to reaction-diffusion systems \eqref{eq:upper}-\eqref{eq:lowest} and via comparison of \eqref{PDEmodel} to the logistic equation to get upper bound for second species, i.e., $v \le 1$. Hence, we have
	\[ \dfrac{1}{1+\mathbf{\widetilde{k}}} \le \dfrac{1}{1+ \mathbf{\widetilde{k}} \hspace{.051in}  \widetilde{u}} \le  \dfrac{1}{1+ k(x) u} \le \dfrac{1}{1+\mathbf{\widehat{k}} \hspace{.051in}  \widehat{u}} \le 1, \quad x \in \Omega. \]
	Hence, the result follows from the standard comparison theory \cite{Gil77}.
\end{proof}

\subsection{Attraction to boundary or interior equilibrium}
\begin{theorem}\label{thm_CE1}
For the reaction-diffusion system (\ref{PDEmodel}) of Allelopathic Phytoplankton with a fear function $k(x)$ that satisfies the parametric restriction
\begin{equation}
    \mathbf{\widehat{k}} > \dfrac{1}{a}-1,
\end{equation}
then there exits some flat initial data such that solution $(u,v)$ to (\ref{PDEmodel})
 converges uniformly to the spatially homogeneous state $(1,0)$ as $t \to \infty$.
\end{theorem}

\begin{proof}
	Consider the reaction-diffusion system given by equation \eqref{eq:upper}. Since the parameter $\mathbf{\widehat{k}}$ satisfies the specified condition, we can apply Theorem \ref{Thm4}. This allows us to select initial values $[u_0, v_0]$ where $v_0$ is significantly smaller than $u_0$ point wise, resulting in the convergence $(\widehat{u}, \widehat{v})$ towards $(1,0)$.

Furthermore, for the reaction-diffusion system given by Equation \eqref{eq:lower}, due to the inequality $\mathbf{\widetilde{k}} > \mathbf{\widehat{k}}$, the parameter $\mathbf{\widetilde{k}}$ also adheres to the imposed conditions. Consequently, Theorem \ref{Thm4} is applicable again, leading to the conclusion that for the same initial values $[u_0, v_0]$ with $v_0$ much smaller than $u_0$ point wise, the system $(\widetilde{u}, \widetilde{v})$ converges to  $(1,0)$.

Moreover, employing Lemma \ref{lem:com1}, we establish the relation $\widetilde{v} \leq v \leq \widehat{v}$. This implies:

\begin{align*}
	\lim_{t \rightarrow \infty}(\widetilde{u}, \widetilde{v})	 \leq  \lim_{t \rightarrow \infty} (u,v)    \leq \lim_{t \rightarrow \infty} (\widehat{u},\widehat{v}),
\end{align*}

and consequently:

\begin{align*}
	\left(1,0\right)	 \leq  \lim_{t \rightarrow \infty}   (u,v)    \leq 	(1,0).
\end{align*}

By employing a squeezing argument, as $t$ tends towards infinity, for initial data $[u_0,v_0]$, we can deduce the uniform convergence of solutions for the Equation \eqref{PDEmodel}. This leads to the assertion that:

\[ (u,v) \to (1,0)\] 

as $t$ approaches infinity.
\end{proof}

\begin{figure}[h]
\subfigure[$k(x)=3+sin^2 (10x)$]
{\scalebox{0.45}[0.45]{
\includegraphics[width=\linewidth,height=5in]{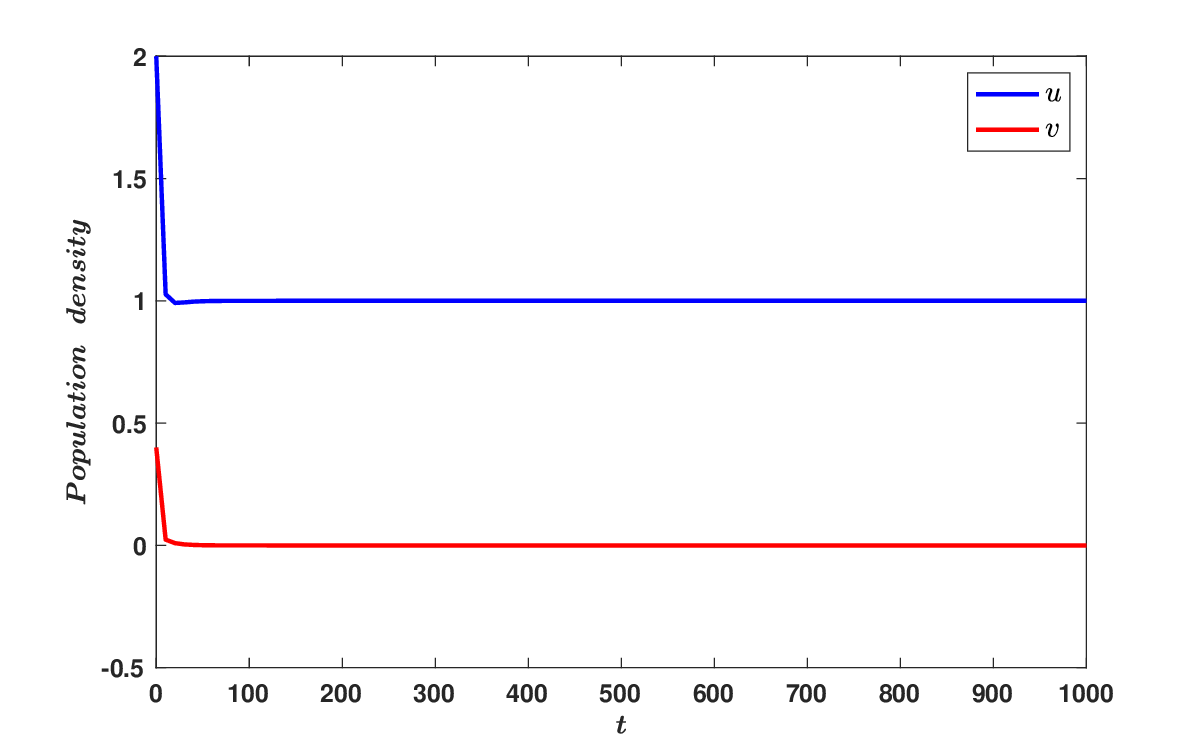}}}
\subfigure[$k(x)=4+sin^2 (10x)$]
{\scalebox{0.45}[0.45]{
\includegraphics[width=\linewidth,height=4.9in]{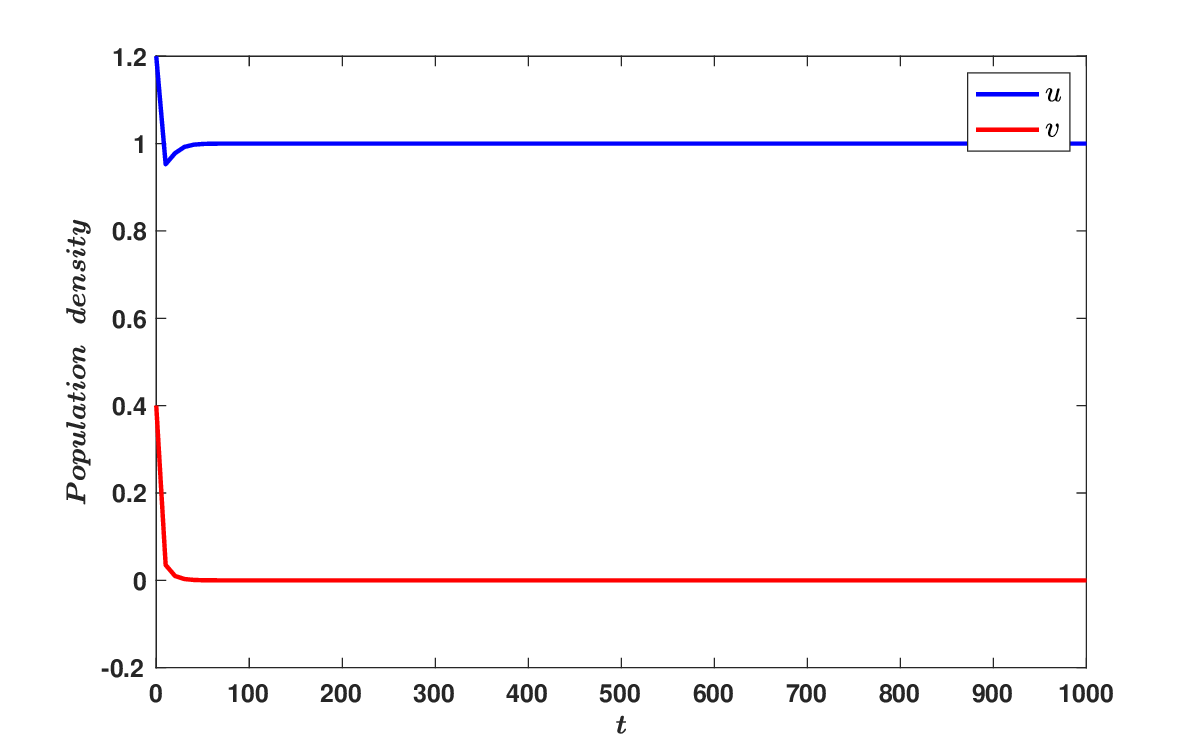}}}
\caption{Numerical simulation of \eqref{PDEmodel} for the case of competition exclusion in $\Omega=[0,\pi]$. The parameters are chosen as
$d_1 = 1, d_2 = 1, a = 0.3, b = 0.2, c=1.1$ and
$m=0.15.$ The initial data are chosen (a) $[u_0, v_0] = [2, 0.4]$ (b) $[u_0, v_0] = [1.2, 0.4]$.}
\label{fig:CE1}
\end{figure}

\begin{theorem}\label{thm_CE2}
For the reaction diffusion system (\ref{PDEmodel}) of Allelopathic Phytoplankton with a fear function $k(x)$, and $c>1$,
then there exits some flat initial data such that solution $(u,v)$ to (\ref{PDEmodel})
 converges uniformly to the spatially homogeneous state $(0,1)$ as $t \to \infty$.
\end{theorem}
\begin{proof}
Consider the reaction-diffusion system \eqref{eq:upper}. Since $c>1$ satisfies the parametric restriction, from  Theorem \ref{Thm4}, we can pick some initial data $[u_1,v_1] (u_1 \ll v_1$ pointwise) such that
	\[ (\widehat{u},\widehat{v}) \to (0,1).\]
    Similarly consider the reaction-diffusion system \eqref{eq:lower}, from Theorem \ref{Thm4}, for same set of initial data $[u_1,v_1] (u_1 \ll v_1 pointwise)$, we have 
	\[ (\widetilde{u},\widetilde{v}) \to (0,1).\]
	Moreover, on using Lemma \ref{lem:com1} we have,
	\[ \widetilde{v}	 \leq   v    \leq \widehat{v},\]
	which entails,
	\begin{align*}
		\lim_{t \rightarrow \infty}(\widetilde{u}, \widetilde{v})	 \leq  \lim_{t \rightarrow \infty} (u,v)    \leq \lim_{t \rightarrow \infty} (\widehat{u},\widehat{v}),
	\end{align*}
	subsequently,
	\begin{align*}
		\left(0,1\right)	 \leq  \lim_{t \rightarrow \infty}   (u,v)    \leq 	(0,1).
	\end{align*}
	Now using a squeezing argument, in the limit that $t \rightarrow \infty$, for initial data $[u_1,v_1] (u_1 \ll v_1 pointwise)$, we have uniform convergence of solutions of \eqref{PDEmodel}, i.e.,
	\[ (u,v) \to (0,1)\] as $t \rightarrow \infty$.
\end{proof}

\begin{figure}[h]
\subfigure[$k(x)=1.5+sin^2 (10x)$]
{\scalebox{0.45}[0.45]{
\includegraphics[width=\linewidth,height=5in]{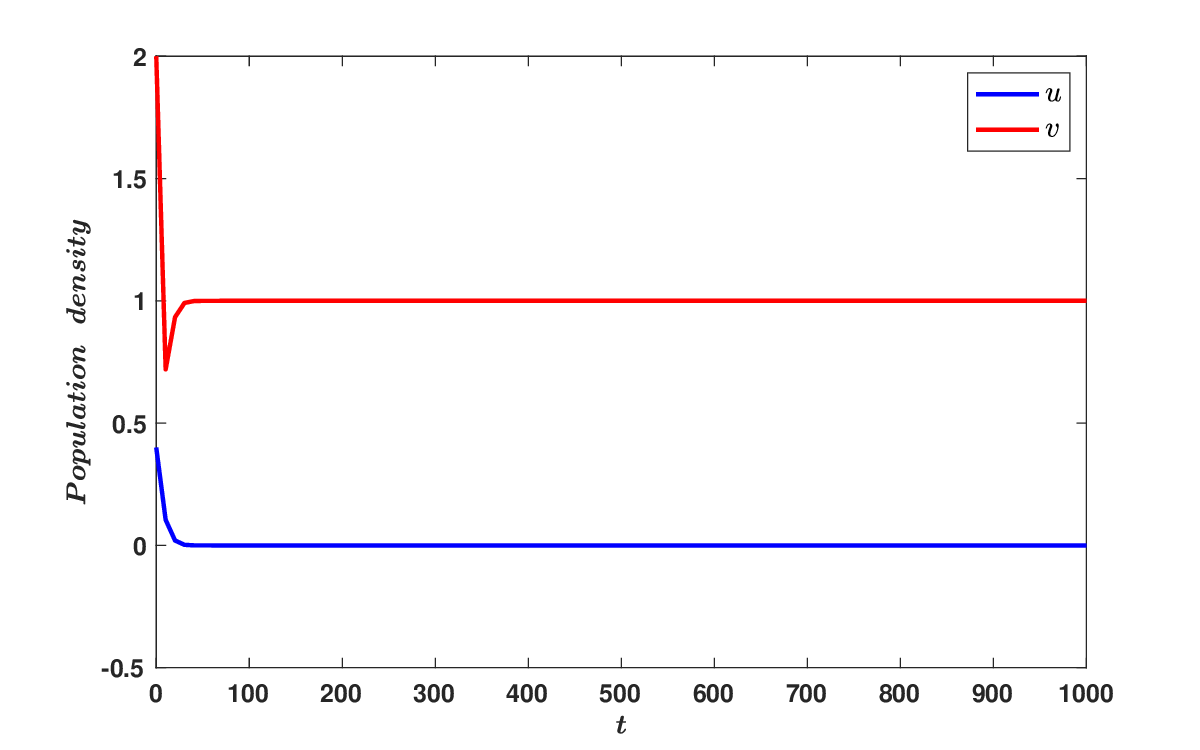}}}
\subfigure[$k(x)=2+sin^2 (10x)$]
{\scalebox{0.45}[0.45]{
\includegraphics[width=\linewidth,height=4.9in]{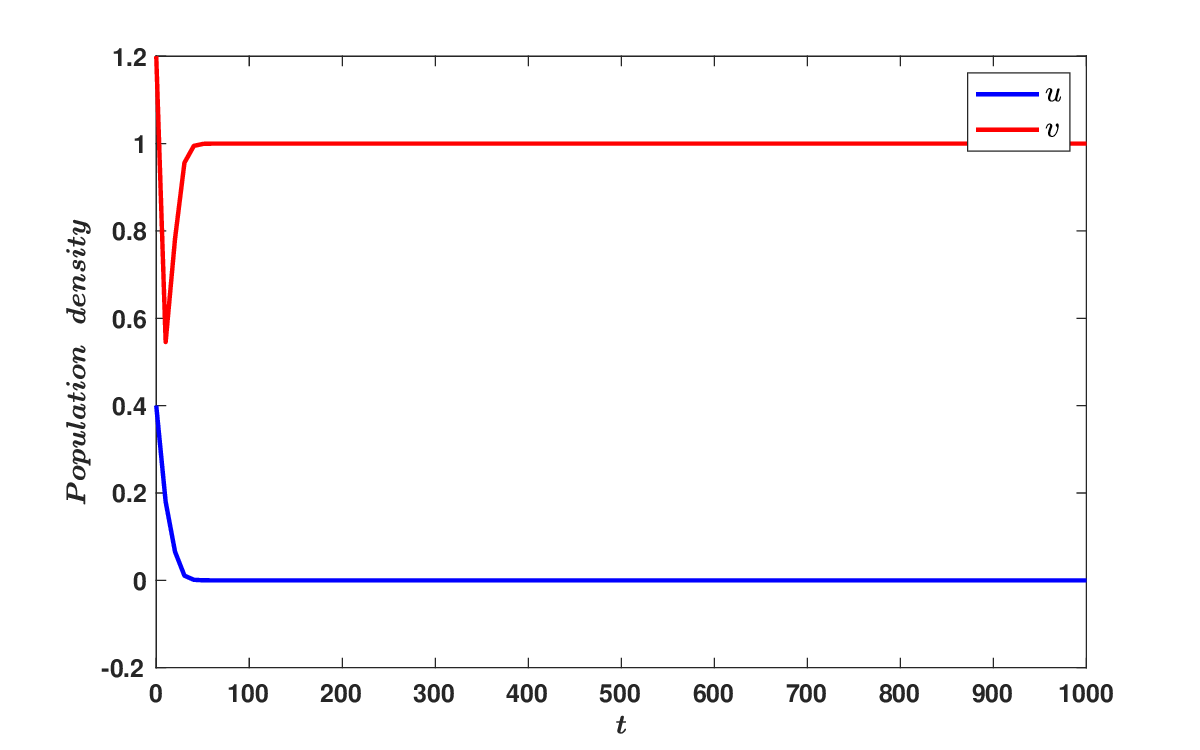}}}
\caption{Numerical simulation of \eqref{PDEmodel} for the case of competition exclusion in $\Omega=[0,\pi]$. The parameters are chosen as
$d_1 = 1, d_2 = 1, a = 0.4, b = 0.2, c=2.1$ and
$m=0.4.$ The initial data are chosen (a) $[u_0, v_0] = [0.4,2]$ (b) $[u_0, v_0] = [0.4,1.2]$.}
\label{fig:CE2}
\end{figure}

\begin{remark}
We see that via theorems \ref{thm_CE1} $\&$ \ref{thm_CE2}, that attraction to boundary equilibrium is possible for certain initial data. For other (positive) initial data, depending on parametric restrictions, one could have attraction to an interior state as well. 
    \end{remark}

\subsection{A case of strong competition}
\begin{theorem}\label{thm_SComp}
    For the reaction-diffusion system (\ref{PDEmodel}) of Allelopathic Phytoplankton with a fear function $k(x)$ that satisfies the parametric restriction
    \[ m>1-ac-\mathbf{k},\quad c>1,\quad u(E)<0, \quad m>\dfrac{2ac\mathbf{k}+ac-\mathbf{k}+1}{1+\mathbf{k}}, \quad \mathbf{k}\ge\dfrac{1}{a}-1,\]
    for $\mathbf{k}=\mathbf{\widehat{k}},   \mathbf{\widetilde{k} },$ and $u$ is a cubic polynomial given by \eqref{16}.
Then there exists sufficiently small initial data $[u_0(x),v_0(x)]$ $(v_0(x)<<u_0(x)$ pointwise$)$,  such that the solution $(u,v)$ to (\ref{PDEmodel}) converges uniformly to the spatially homogeneous state $(1,0)$ as $t \to \infty$, while there exits also sufficiently large intial data $[u_1 (x),v_1 (x)]$ $(u_1(x)<<v_1(x)$ pointwise$)$ for which the solution $(u,v)$ to (\ref{PDEmodel}) converges uniformly to the spatially homogeneous state $(0,1)$ as $t \to \infty$.
\end{theorem}

\begin{proof}
	Consider the reaction-diffusion system \eqref{eq:upper}. Since the $\mathbf{\widehat{k}}$ satisfies the parametric restriction, from Theorem~\ref{Thm5} , there exists a interior saddle equilibrium $E_{1*}$ to the kinetic (ODE) system \eqref{eq:upper}. On making use of the stable manifold theorem \cite{Perko13}, i.e., $\exists \hspace{0.1in} W^1_{s}(E_{1*}) \in \mathcal{C}^{1}$ separatrix, such that for initial data $(\widehat{u}_0,\widehat{v}_0)$ chosen right to $W^1_{s}(E_{1*})$ the solution $(\widehat{u},\widehat{v}) \to (1,0)$ and for initial data chosen left to $W^1_{s}(E_{1*})$, $(\widehat{u},\widehat{v}) \to (0,1)$.

	Moreover, notice that $\dfrac{1}{1 + \mathbf{\widetilde{k}} u} \le \dfrac{1}{1 + \mathbf{\widehat{k}} u}$, we have that for the kinetic (ODE) system \eqref{eq:lower}, we still remain in the strong competition case, and via standard theory again, $\exists \hspace{0.1in} W_{s}(E_{1**}) \in \mathcal{C}^{1}$ separatrix, such that for initial data $(\widetilde{u}_0,\widetilde{v}_0)$ chosen left to $W_{s}(E_{1**})$ the solution $(\widetilde{u},\widetilde{v}) \to (0,1)$ and for initial data chosen right to $W_{s}(E_{1**})$, $(\widetilde{u},\widetilde{v})\to (1,0)$. Here $E_{1**}$ is the interior saddle equilibrium to the kinetic (ODE) system for \eqref{eq:lower}.
	
	Now since $\dfrac{1}{1 + \mathbf{\widetilde{k}} u} \le \dfrac{1}{1 + \mathbf{\widehat{k}} u}$, the $v$ component of $E_{1**}$ is more than the $v$ component of $E_{1*}$. Now using the $\mathcal{C}^{1}$ property of the separatricies $W^{1}_{s}(E_{1*}) , W_{s}(E_{1**})$, we have the existence of a wedge $\mathbb{V}$ emanating from $E_{1*}$, s.t within $\mathbb{V}$ we have $W^{1}_{s}(E_{1*}) \leq W_{s}(E_{1**})$. Note via Lemma \ref{lem:com1} we have $ \widetilde{v} \leq v \leq \widehat{v}$.	Let us consider positive initial data $(u_0,v_0)$ chosen small enough, within $\mathbb{V}$ s.t. $  (u_0,v_0) < W^{1}_{s}(E_{1*})  \le W_{s}(E_{1**})$, we will have 
	\begin{align*}
		\Big\{  (1,0) \Big\} \le \Big\{ (u,v) \Big\} \le \Big\{ (1,0) \Big\}.
	\end{align*}
	On the other hand,  for sufficiently large initial data $(u_1,v_1)$ via an analogous construction we will have 
	\begin{align*}
		\Big\{  (0,1)\Big\} \le \Big\{ (u,v) \Big\} \le \Big\{ (0,1) \Big\}.
	\end{align*}
	This proves the theorem.
\end{proof}

\begin{figure}[h]
\subfigure[$k(x)=4+sin^2 (10x)$]
{\scalebox{0.45}[0.45]{
\includegraphics[width=\linewidth,height=5in]{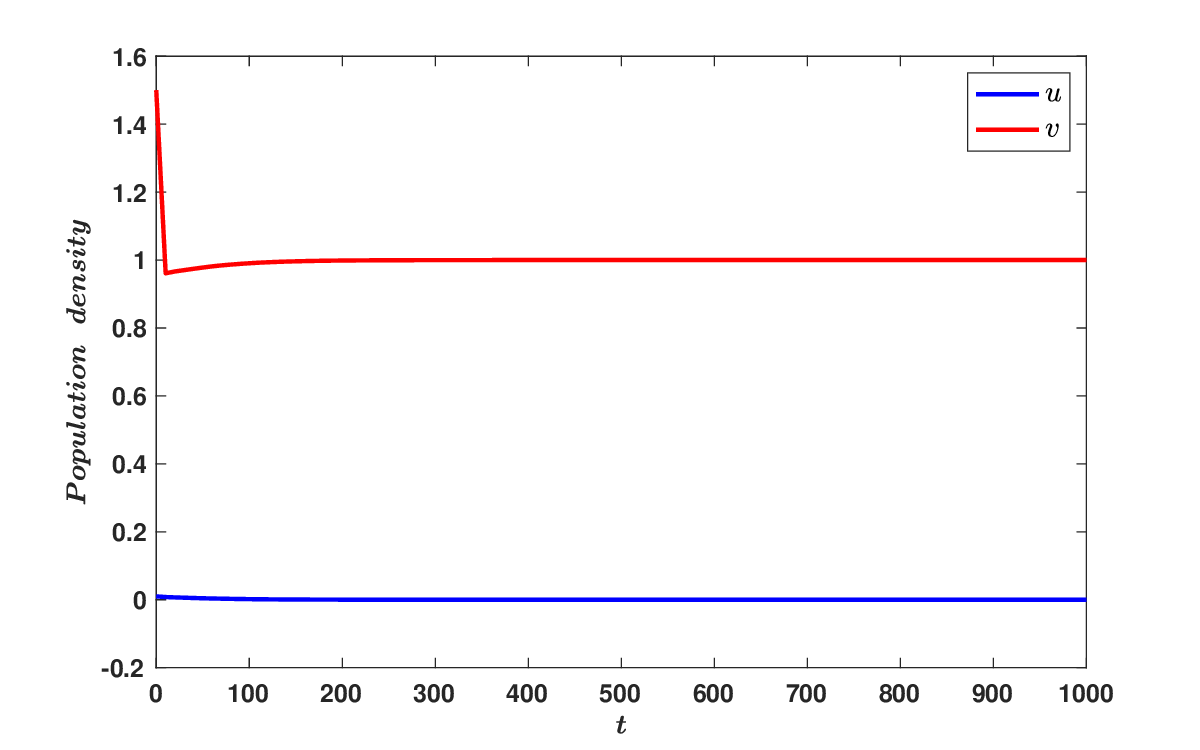}}}
\subfigure[$k(x)=4+sin^2 (10x)$]
{\scalebox{0.45}[0.45]{
\includegraphics[width=\linewidth,height=4.9in]{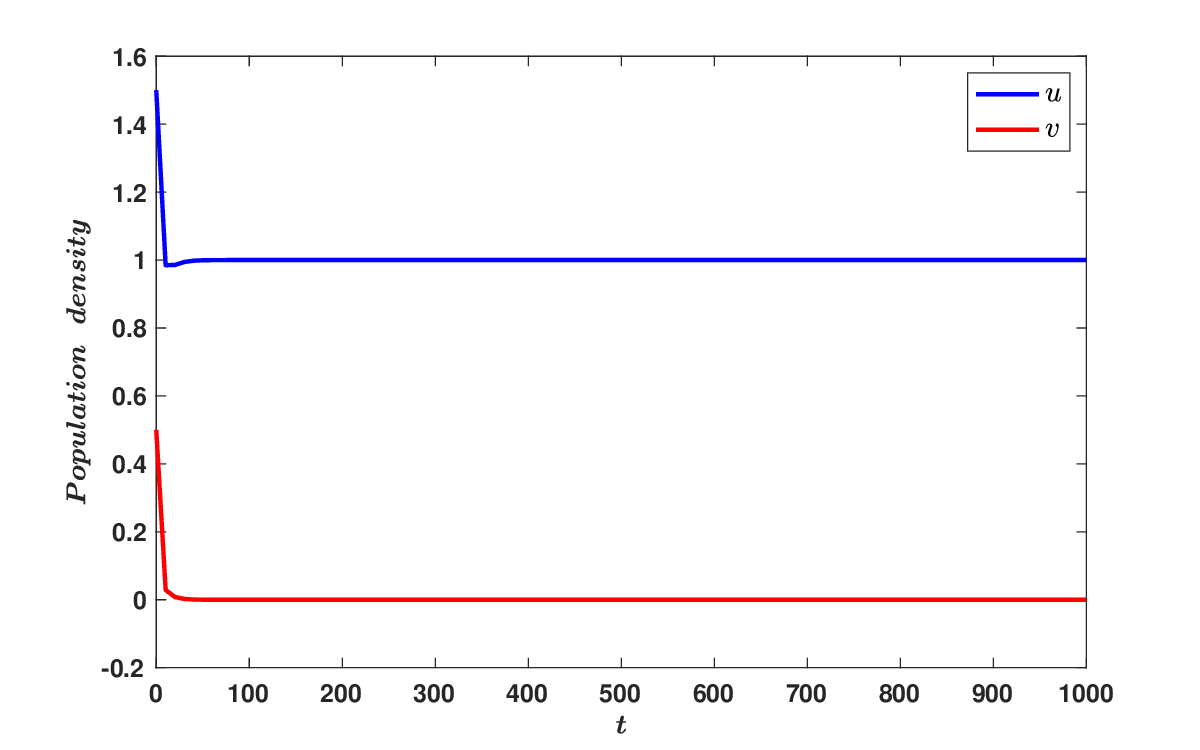}}}
\caption{Numerical simulation of \eqref{PDEmodel} for the case of strong comp in $\Omega=[0,\pi]$. The parameters are chosen as
$d_1 = 1, d_2 = 1, a = 0.2, b = 0.2, c=1.1$ and
$m=0.15.$ The initial data are chosen (a) $[u_0, v_0] = [0.01,1.5]$ (b) $[u_0, v_0] = [1.5,0.5]$.}
\label{fig:SComp}
\end{figure}

\begin{figure}[h]
\subfigure[]
{\scalebox{0.45}[0.45]{
\includegraphics[width=\linewidth,height=5in]{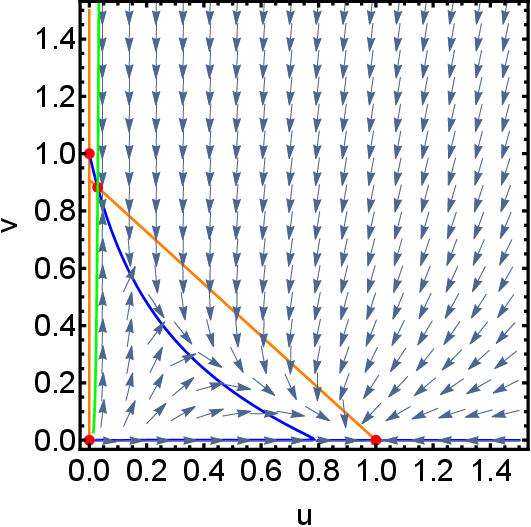}}}
\subfigure[]
{\scalebox{0.45}[0.45]{
\includegraphics[width=\linewidth,height=4.9in]{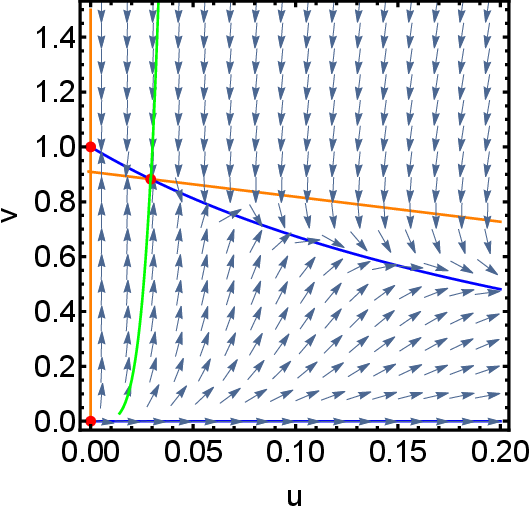}}}
\caption{Phase plots showing various dynamics under strong competition parametric restrictions. The parameters are chosen as
$a = 0.2, b = 0.2, c=1.1$ and
$m=0.15.$ }
\label{fig:ODE1}
\end{figure}

\begin{figure}[h]
\begin{center}
\subfigure[]
  {\scalebox{0.5}[0.5]{
\includegraphics[width=\linewidth,height=5in]{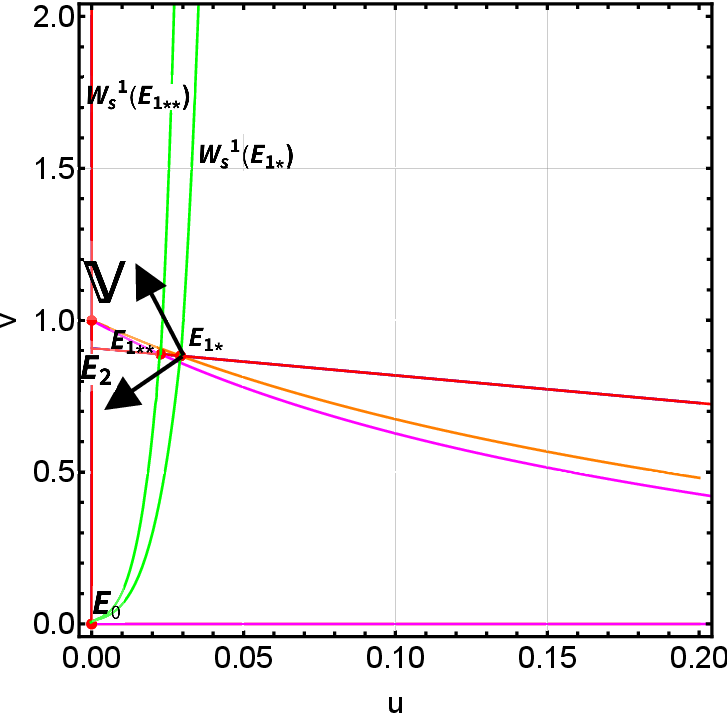}}}  
\end{center}
\caption{Numerical simulation for the reaction diffusion system (\ref{PDEmodel}) of Allelopathic Phytoplankton with a fear function $k(x)$ in $\Omega=[0,\pi]$. The parameters are chosen as $d_1 = 1, d_2 = 1, a = 0.2, b = 0.2, c=1.1$ and
$m=0.15.$ Equilibria: $E_{1*}=(0.029,0.882),E_{1**}=(0.022,0.888), E_{2}=(0,1), E_1=(1,0)$ and $E_0=(0,0)$. $W_{s}^{1}(E_{1*})$ ($k=4$) and $W_{s} (E_{1**)}$ ($k=5$) are two sepratrices passing through $E_{1*}$ and  $E_{1**}$ respectively. The $\mathcal{C}^1$ property of the separatrices, $W^{1}_{s}(E_{1*}) , W_{s}(E_{1**})$, shows a wedge $\mathbb{V}$ emanating from $E_{1*}$, such that within $\mathbb{V}$ we have $W^{1}_{s}(E_{1*})\le W_{s}(E_{1**})$. The $u$-nullcline is in red for $k=4$ and $k=5$. For 
$k=4$, $v$-nullcline is in orange. For $k=5$, $v$-nullcline is in magenta.}
\label{fig:ODE2}
\end{figure}

\subsection{The weak competition case}

The Theorems \ref{Thm5} and \ref{Thm:7}, along with numerical simulations (Fig~\ref{fig:WC}) motivate the following conjecture:
\begin{conjecture}\label{conj:WC}
    For the reaction-diffusion system (\ref{PDEmodel}) of Allelopathic Phytoplankton with a fear function $k(x)$ that satisfies the parametric restriction
    \[m>1-ac-\mathbf{k}, \quad 0<c<1, \quad 0<\mathbf{k}<\frac{1}{a}-1,\]
    for $\mathbf{k}=\mathbf{\widehat{k}},   \mathbf{\widetilde{k} }.$
Then for any positive set of initial data $[u_0(x),v_0(x)]$, the solution $(u,v)$ to (\ref{PDEmodel}) converges uniformly to the spatially homogeneous state $(u^*,v^*)$ as $t \to \infty$.
\end{conjecture}
\begin{figure}[h]
\subfigure[$k(x)=0.1sin^2 (10x)$]
{\scalebox{0.45}[0.45]{
\includegraphics[width=\linewidth,height=5in]{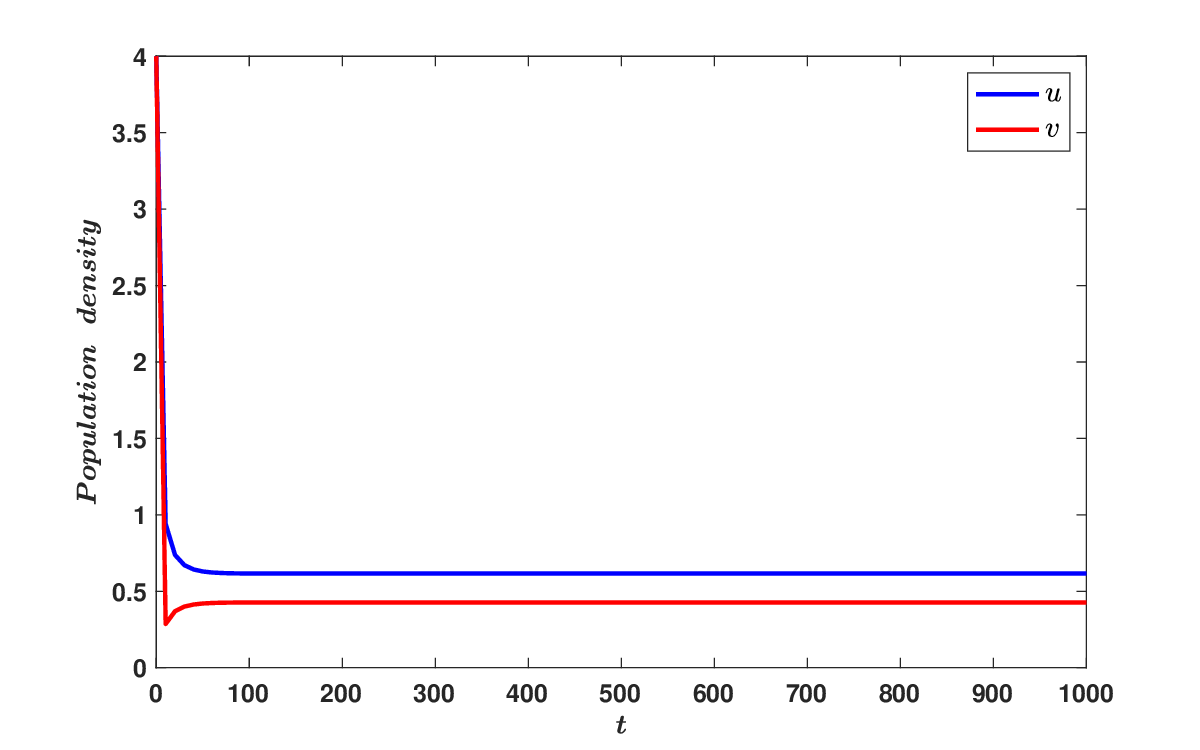}}}
\subfigure[$k(x)=0.1sin^2 (10x)$]
{\scalebox{0.45}[0.45]{
\includegraphics[width=\linewidth,height=4.9in]{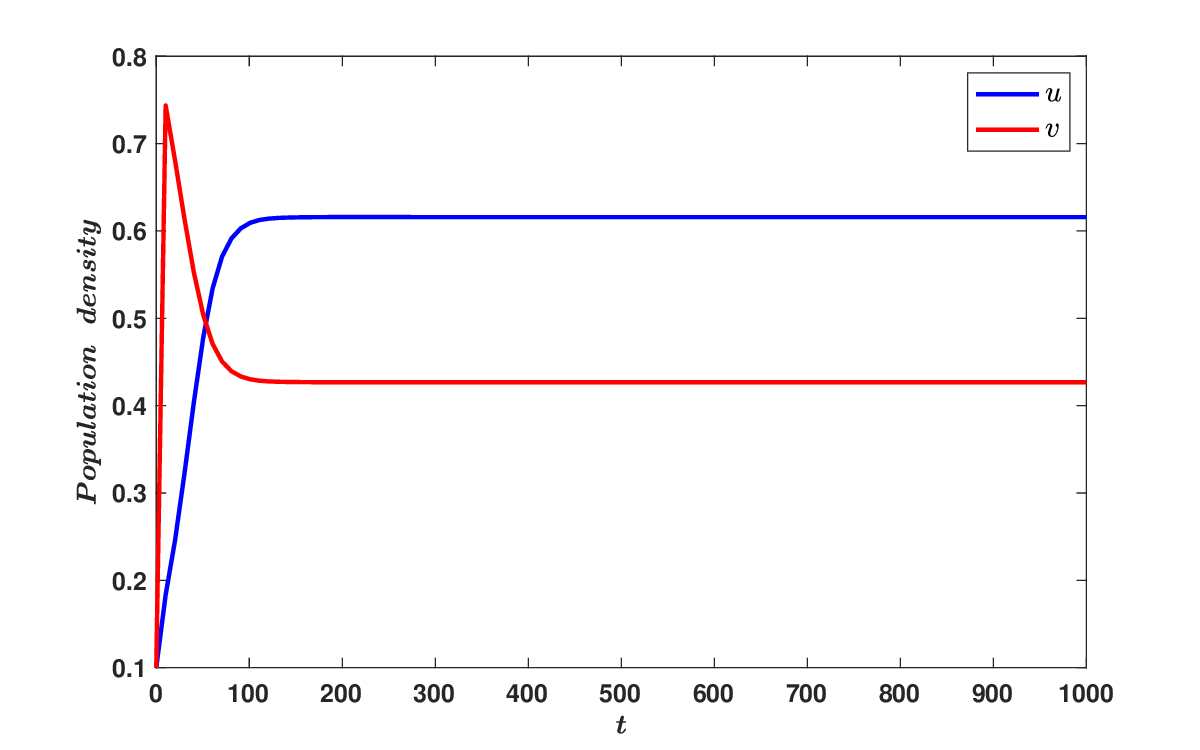}}}
\caption{Numerical simulation of \eqref{PDEmodel} for the case of weak comp in $\Omega=[0,\pi]$. The parameters are chosen as
$d_1 = 1, d_2 = 1, a = 0.2, b = 0.2, c=0.9$ and
$m=1.6.$ The initial data are chosen (a) $[u_0, v_0] = [4,4]$ (b) $[u_0, v_0] = [0.1,0.1]$.}
\label{fig:WC}
\end{figure}

\subsection{The case of multiple interiors}
The numerical simulations Fig~\ref{fig:Bi_Stab} motivate the following conjecture:
\begin{conjecture}\label{conj:Bi_stab}
    For the reaction-diffusion system (\ref{PDEmodel}) of Allelopathic Phytoplankton with a fear function $k(x)$ that satisfies the parametric restriction
    \[m>1-ac-\mathbf{k}, \quad m>\dfrac{2ac\mathbf{k}+ac-\mathbf{k}+1}{1+\mathbf{k}}, \quad u(E)<0, \quad c>1, \quad 0<\mathbf{k}<\frac{1}{a}-1,\]
    for $\mathbf{k}=\mathbf{\widehat{k}},   \mathbf{\widetilde{k} },$ and $u$ is a cubic polynomial given by \eqref{16}.
    Then there exists sufficiently small initial data $[u_0(x),v_0(x)]$ $(v_0(x)<<u_0(x)$ pointwise$)$,  such that the solution $(u,v)$ to (\ref{PDEmodel}) converges uniformly to the spatially homogeneous state $(u^*,v^*)$ as $t \to \infty$, while there exits also sufficiently large intial data $[u_1 (x),v_1 (x)]$ $(u_1(x)<<v_1(x)$ pointwise$)$ for which the solution $(u,v)$ to (\ref{PDEmodel}) converges uniformly to the spatially homogeneous state $(0,1)$ as $t \to \infty$.
\end{conjecture}

\begin{figure}[h]
\subfigure[$k(x)=1.5+0.1sin^2 (10x)$]
{\scalebox{0.45}[0.45]{
\includegraphics[width=\linewidth,height=5in]{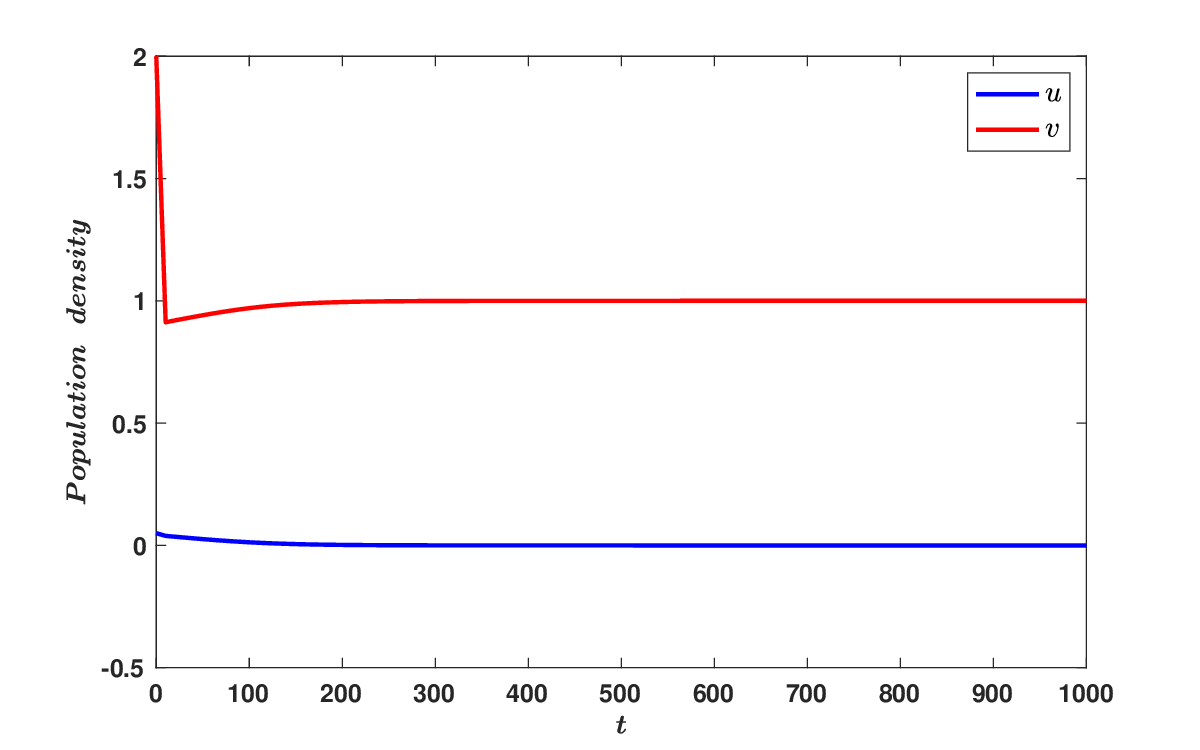}}}
\subfigure[$k(x)=1.5+0.1sin^2 (10x)$]
{\scalebox{0.45}[0.45]{
\includegraphics[width=\linewidth,height=4.9in]{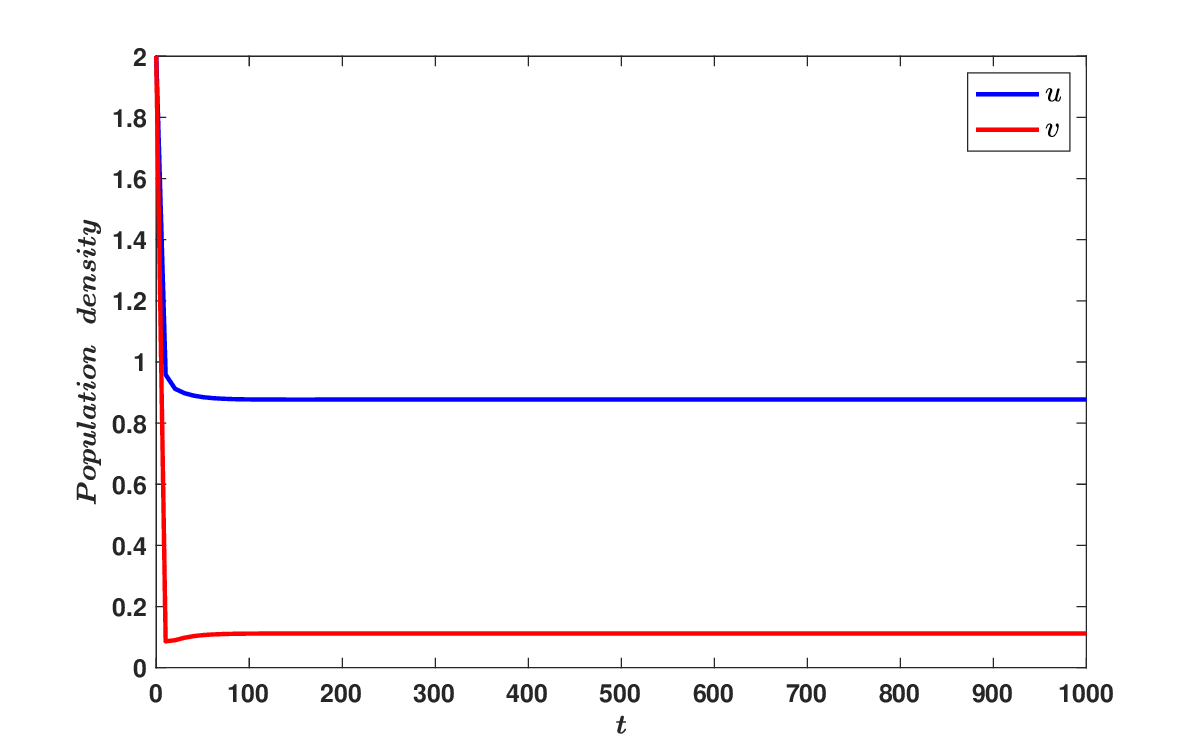}}}
\caption{Numerical simulation of \eqref{PDEmodel} for the case of bi-stability in $\Omega=[0,\pi]$. The parameters are chosen as
$d_1 = 1, d_2 = 1, a = 0.3, b = 0.2, c=1.1$ and
$m=0.5.$ The initial data are chosen (a) $[u_0, v_0] = [0.05,2]$ (b) $[u_0, v_0] = [2,2]$.}
\label{fig:Bi_Stab}
\end{figure}

\section{Numerical Simulations}
The \verb|MATLAB| R2021b software was employed to conduct a PDE simulation for a reaction-diffusion system (\ref{PDEmodel}) modeling Allelopathic Phytoplankton. This simulation considered spatially heterogeneous fear functions, denoted as $k(x)$. The solution was obtained using the pdepe function to solve 1-D initial boundary value problems in a single spatial dimension. The computational task was performed on an 8-core CPU within an Apple M1 Pro-based workstation, taking approximately $5-7$ seconds to complete when applied to the spatial domain interval $[0,\pi]$, which was divided into 1000 sub-intervals.

Our theoretical findings and conjectures, specific to the spatially explicit context, were substantiated through a time series analysis conducted over an extended duration. Simulations were executed with parameters conforming to the constraints established by the theorems. In the spatially explicit setting, we used the standard comparison theory to derive point-wise constraints on the fear function $k(x)$. This analysis observed competitive exclusion, strong competition, and multiple equilibria-type dynamics within the reaction-diffusion system featuring a spatially heterogeneous fear function.

The outcomes of Theorems \ref{thm_CE1}, \ref{thm_CE2}, \ref{thm_SComp}, and Conjectures \ref{conj:WC}, \ref{conj:Bi_stab} provided clear evidence of these phenomena. To further validate our numerical results, we utilized Figures [\ref{fig:CE1}, \ref{fig:CE2}, \ref{fig:SComp}, \ref{fig:WC}, \ref{fig:Bi_Stab}].

Theoretical results were rigorously validated through numerical experiments employing various heterogeneous fear functions. Each figure caption includes details regarding the parameters used for these simulations and their relevance to specific theorems. It is important to note that all parameter choices remained within the range $[0,5]$, consistent with the model and its comparison to the logistic equation, indicating that any species' population cannot exceed unity.

\section{Summary and Conclusion}
In this paper, we are the first to propose an allelopathic phytoplankton competition model influenced by the fear effect, where the parameters $k$ and $m$ denote the fear effect and the toxic release rate, respectively. Our study shows that $k$ and $m$ perturb the classical Lotka-Volterra competition model to cause rich dynamics. Meanwhile, $k$ and $m$ can significantly impact species density biologically.

First, we give the conditions for persistence for system \eqref{2}. When the persistence condition is satisfied, the two species will coexist. System \eqref{2} has three boundary equilibria. To study the positive equilibria, we construct a cubic function \eqref{16}. By analyzing the original image of this function as well as the image of the corresponding derivative function, we find that there are at most two positive equilibria of system \eqref{2} and give the existence conditions for the corresponding cases.

The next step is to analyze the stability of the equilibria.  We investigate the Jacobian matrix corresponding to the boundary equilibria $E_0$, $E_1$, and $E_2$, respectively. By analyzing the traces and the determinant of the matrix, it is found that $E_0$ is always a source. The fear effect $k$ and the interspecific competition rate $c$ which the toxic species is subjected will affect $E_1$ and $E_2$, respectively. Furthermore, when the toxin release rate $m$ reaches a certain threshold, either E1 or E2 will turn into a degraded equilibrium point.

For the positive equilibria, we have used \eqref{19} to study the relationship between the determinant of its Jacobian matrix as well as the slope of the tangent line and further obtain that $E_{1*}$ is a saddle point and $E_{2*}$ is a stable node. At the point $E_{3*}$, we obtained that its determinant equals 0, so we translated this point to the origin and performed Taylor's expansion. Finally, we used Theorem 7 in Chapter 2 to prove that $E_{3*}$ is a saddle-node. In particular, we prove no closed orbit for system \eqref{2}. Combined with the persistence condition, the locally stable positive equilibria $E_{2*}$ is also globally stable in system \eqref{2}.

In addition, by varying the fear effect $k$ or the interspecific competition rate $c$ to which toxic species is subjected, system \eqref{2} will experience transcritical bifurcation around $E_1$ or $E_2$. If the toxic release rate $m = 1-a-k$, the transcritical bifurcation experienced around $E_2$ will turn into a pitchfork bifurcation. When the toxic release rate $m$ is used directly as a bifurcation parameter, it results in a saddle-node bifurcation of system \eqref{2} in the first quadrant.

In essence, these results are seen in the spatially explicit case as well. For large fear coefficient extinction (for certain initial data) is seen for the non-toxic fearful species, see Theorem \ref{thm_CE1}. Depending on the interplay of other parameters, one sees a strong competition type setting, see Theorem \ref{thm_SComp}. Future work will explore a spatially heterogeneous toxic release rate $m$ (perhaps even one that causes degeneracy), as well as different forms of this rate, including the non-smooth case \cite{P21, K20}. We will also explore global stability of the interior equilibrium in the PDE case, as well as the existence of non-constant steady states.

To summarize all of the above analysis, the two species can coexist only if the fear effect $k$ is within the interval $(0,k^*)$. As for the toxic release rate $m$, it does not directly change the survival of the non-toxic species but only affects the species' density. We can conclude that in the allelopathic phytoplankton competition model, the real cause of the extinction of non-toxic species is the fear of toxic species compared to toxins. This article has some guidance for the conservation of species diversity.

\appendix{}
\begin{align*}
g_{11}=&\frac{\left(E^{2} m +1\right) Q_1 b}{Q_2^{2}}, \quad g_{02}=\frac{\left(2 E k +1\right) \left(-1+E \right)Q_3}{Q_2^{2} \left(E k +1\right)^{4} \left(E^{2} m +1\right)^{2}},\quad f_{20}=-\frac{Q_4}{Q_2^{2} \left(2 E k +1\right)^{2}}, \\
f_{11}=&-\frac{b E \left(E k +1\right)^{2} \left(E^{2} m +1\right) Q_5}{Q_2^{2} \left(2 E k +1\right)},\quad f_{02}=-\frac{Q_6}{Q_2^{2} \left(E k +1\right)^{2} \left(E^{2} m +1\right)},\\
Q_1=&4 E^{6} b k^{3} m -6 E^{5} b k^{3} m +7 E^{5} b k^{2} m -12 E^{4} bk^{2} m +4 E^{4} b k m +4 E^{4} k^{2} m -2 E^{3} bk^{3}-3 E^{3} bk^{2}-8 E^{3} b k m \\
&-4 E^{3} k^{2} m +E^{3} b m +4 E^{3} k^{2}+4 E^{3} k m -2 E^{2} bk^{2}-4 E^{2} b k -2 E^{2} b m -4 E^{2} k^{2}-4 E^{2} k m +4 E^{2} k \\
&+E^{2} m -b E -4 E k -E m +E -1,\\
Q_2=&E^{5} bk^{2} m +2 E^{4} b k m +E^{3} bk^{2}+E^{3} b m -2 E^{3} k m +2 E^{2} b k +2 E^{2} k m -2 E^{2} k -E^{2} m +b E \\
&+2 E k +E m -E +1,\\
Q_3=&3 E^{11} b^{2} k^{5} m^{3}+2 E^{10} b^{2} k^{5} m^{2}+12 E^{10} b^{2} k^{4} m^{3}+7 E^{9} b^{2} k^{5} m^{2}+9 E^{9} b^{2} k^{4} m^{2}+19 E^{9} b^{2} k^{3} m^{3}-4 E^{9} bk^{4} m^{3}\\
&+4 E^{8} b^{2} k^{5} m +27 E^{8} b^{2} k^{4} m^{2}+16 E^{8} b^{2} k^{3} m^{2}+15 E^{8} b^{2} k^{2} m^{3}-12 E^{8} b k^{4} m^{2}-12 E^{8} bk^{3} m^{3}+5 E^{7} b^{2} k^{5} m \\
&+18 E^{7} b^{2} k^{4} m +41 E^{7} b^{2} k^{3} m^{2}+14 E^{7} b^{2} k^{2} m^{2}+6 E^{7} b^{2} km^{3}-8 E^{7} bk^{4} m -36 E^{7} bk^{3} m^{2}-13 E^{7} bk^{2} m^{3}\\
&+8 E^{7} k^{3} m^{3}+2 E^{6} b^{2} k^{5}+18 E^{6} b^{2} k^{4} m +32 E^{6} b^{2} k^{3} m +31 E^{6} b^{2} k^{2} m^{2}-8 E^{6} b \,k^{4} m -8 E^{6} k^{3} m^{3}+E^{5} b^{2} k^{5}\\
&+6 E^{6} b^{2} k \,m^{2}+E^{6} b^{2} m^{3}-24 E^{6} b \,k^{3} m -39 E^{6} b \,k^{2} m^{2}-6 E^{6} b k \,m^{3}+24 E^{6} k^{3} m^{2}+12 E^{6} k^{2} m^{3}+9 E^{5} b^{2} k^{4}\\
&+25 E^{5} b^{2} k^{3} m +4 E^{5} b \,k^{4} m +28 E^{5} b^{2} k^{2} m +12 E^{5} b^{2} k \,m^{2}-8 E^{5} b \,k^{4}-24 E^{5} b \,k^{3} m -24 E^{5} k^{3} m^{2}\\
&-12 E^{5} k^{2} m^{3}+3 E^{4} b^{2} k^{4}+E^{5} b^{2} m^{2}-26 E^{5} b \,k^{2} m -18 E^{5} b k \,m^{2}-E^{5} b \,m^{3}+24 E^{5} k^{3} m +36 E^{5} k^{2} m^{2}\\
&+6 E^{5} k \,m^{3}+16 E^{4} b^{2} k^{3}+17 E^{4} b^{2} k^{2} m +4 E^{4} b \,k^{4}+12 E^{4} b \,k^{3} m +12 E^{4} b^{2} k m +2 E^{4} b^{2} m^{2}-24 E^{4} b \,k^{3}\\
&-26 E^{4} b \,k^{2} m -24 E^{4} k^{3} m -36 E^{4} k^{2} m^{2}-6 E^{4} k \,m^{3}+3 E^{3} b^{2} k^{3}-12 E^{4} b k m -3 E^{4} b \,m^{2}+8 E^{4} k^{3}\\
&+36 E^{4} k^{2} m +18 E^{4} k \,m^{2}+E^{4} m^{3}+14 E^{3} b^{2} k^{2}+6 E^{3} b^{2} k m +12 E^{3} b \,k^{3}+13 E^{3} b \,k^{2} m +2 E^{3} b^{2} m \\
&-26 E^{3} b \,k^{2}-12 E^{3} b k m -8 E^{3} k^{3}-36 E^{3} k^{2} m -18 E^{3} k \,m^{2}-E^{3} m^{3}+k^{2} b^{2} E^{2}-2 E^{3} b m +12 E^{3} k^{2}\\
&+18 E^{3} k m +3 E^{3} m^{2}+6 E^{2} b^{2} k +E^{2} b^{2} m +13 E^{2} b \,k^{2}+6 E^{2} b k m -12 E^{2} b k -2 E^{2} b m -12 E^{2} k^{2}\\
&-18 E^{2} k m -3 E^{2} m^{2}+6 E^{2} k +3 E^{2} m +E \,b^{2}+6 E b k +E b m -2 b E -6 E k -3 E m +E +b -1,\\
Q_4=&\left(E^{2} m +1\right)^{3} \left(E k +1\right)^{5} \left(-1+E \right) \left(3 E^{2} k^{2} m +3 E k m +k^{2}+m \right) E^{2} b^{2},\\
Q_5=&E^{9} b^{2} k^{4} m^{2}+4 E^{8} b^{2} k^{3} m^{2}+2 E^{7} b^{2} k^{4} m +6 E^{7} b^{2} k^{2} m^{2}+8 E^{6} b^{2} k^{3} m -2 E^{6} b \,k^{3} m^{2}+4 E^{6} b^{2} k \,m^{2}\\
&-4 E^{6} b \,k^{3} m -3 E^{6} b \,k^{2} m^{2}+E^{5} b^{2} k^{4}+12 E^{5} b^{2} k^{2} m +4 E^{5} b \,k^{3} m -2 E^{5} b \,k^{2} m^{2}+E^{5} b^{2} m^{2}-10 E^{5} b \,k^{2} m\\
& -4 E^{5} b k \,m^{2}+8 E^{5} k^{2} m^{2}+4 E^{4} b^{2} k^{3}-4 E^{4} b \,k^{3} m +8 E^{4} b^{2} k m -4 E^{4} b \,k^{3}+4 E^{4} b \,k^{2} m -12 E^{4} k^{2} m^{2}\\
&-8 E^{4} b k m -E^{4} b \,m^{2}+12 E^{4} k^{2} m +8 E^{4} k \,m^{2}+6 E^{3} b^{2} k^{2}+4 E^{3} b \,k^{3}-4 E^{3} b \,k^{2} m +4 E^{3} k^{2} m^{2}+2 E^{3} b^{2} m \\
&-10 E^{3} b \,k^{2}-16 E^{3} k^{2} m -12 E^{3} k \,m^{2}-2 E^{2} b \,k^{3}-2 E^{3} b m +4 E^{3} k^{2}+12 E^{3} k m +2 E^{3} m^{2}+4 E^{2} b^{2} k \\
&+7 E^{2} b \,k^{2}+4 E^{2} k^{2} m +4 E^{2} k \,m^{2}-8 E^{2} b k -4 E^{2} k^{2}-16 E^{2} k m -3 E^{2} m^{2}-2 E b \,k^{2}+4 E^{2} k +3 E^{2} m \\
&+b^{2} E +4 E b k +4 E k m +E \,m^{2}-2 E b -4 E k -4 E m +E +b +m -1,\\
Q6=&E^{14} b^{3} k^{6} m^{3}+6 E^{13} b^{3} k^{5} m^{3}+3 E^{12} b^{3} k^{6} m^{2}+15 E^{12} b^{3} k^{4} m^{3}-E^{12} b^{2} k^{5} m^{3}+18 E^{11} b^{3} k^{5} m^{2}+E^{11} b^{2} k^{5} m^{3}\\
&+20 E^{11} b^{3} k^{3} m^{3}-2 E^{11} b^{2} k^{5} m^{2}-6 E^{11} b^{2} k^{4} m^{3}+3 E^{10} b^{3} k^{6} m +45 E^{10} b^{3} k^{4} m^{2}+E^{10} b^{2} k^{5} m^{2}\\
&+6 E^{10} b^{2} k^{4} m^{3}+15 E^{10} b^{3} k^{2} m^{3}-9 E^{10} b^{2} k^{4} m^{2}-13 E^{10} b^{2} k^{3} m^{3}+18 E^{9} b^{3} k^{5} m +E^{9} b^{2} k^{5} m^{2}+60 E^{9} b^{3} k^{3} m^{2}\\
&-4 E^{9} b^{2} k^{5} m +13 E^{9} b^{2} k^{3} m^{3}-4 E^{9} b \,k^{4} m^{3}+E^{8} b^{3} k^{6}+6 E^{9} b^{3} k \,m^{3}-16 E^{9} b^{2} k^{3} m^{2}-13 E^{9} b^{2} k^{2} m^{3}\\
&-4 E^{9} b \,k^{4} m^{2}+45 E^{8} b^{3} k^{4} m +5 E^{8} b^{2} k^{5} m +9 E^{8} b^{2} k^{4} m^{2}+4 E^{8} b \,k^{4} m^{3}+45 E^{8} b^{3} k^{2} m^{2}-18 E^{8} b^{2} k^{4} m\\
& -7 E^{8} b^{2} k^{3} m^{2}+13 E^{8} b^{2} k^{2} m^{3}-12 E^{8} b \,k^{3} m^{3}+6 E^{7} b^{3} k^{5}-E^{7} b^{2} k^{5} m +E^{8} b^{3} m^{3}-14 E^{8} b^{2} k^{2} m^{2}\\
&-6 E^{8} b^{2} k \,m^{3}-4 E^{8} b \,k^{4} m -12 E^{8} b \,k^{3} m^{2}+8 E^{8} k^{3} m^{3}+60 E^{7} b^{3} k^{3} m -2 E^{7} b^{2} k^{5}+18 E^{7} b^{2} k^{4} m \\
&+23 E^{7} b^{2} k^{3} m^{2}+12 E^{7} b \,k^{3} m^{3}+18 E^{7} b^{3} k \,m^{2}-32 E^{7} b^{2} k^{3} m -11 E^{7} b^{2} k^{2} m^{2}+6 E^{7} b^{2} k \,m^{3}-13 E^{7} b \,k^{2} m^{3}\\
&-16 E^{7} k^{3} m^{3}+15 E^{6} b^{3} k^{4}+3 E^{6} b^{2} k^{5}+4 E^{6} b \,k^{4} m^{2}-6 E^{7} b^{2} k \,m^{2}-E^{7} b^{2} m^{3}-12 E^{7} b \,k^{3} m -13 E^{7} b \,k^{2} m^{2}\\
&+24 E^{7} k^{3} m^{2}+12 E^{7} k^{2} m^{3}+45 E^{6} b^{3} k^{2} m -9 E^{6} b^{2} k^{4}+25 E^{6} b^{2} k^{3} m +25 E^{6} b^{2} k^{2} m^{2}+13 E^{6} b \,k^{2} m^{3}\\
&+8 E^{6} k^{3} m^{3}-E^{5} b^{2} k^{5}+3 E^{6} b^{3} m^{2}-28 E^{6} b^{2} k^{2} m -6 E^{6} b^{2} k \,m^{2}+E^{6} b^{2} m^{3}-4 E^{6} b \,k^{4}-6 E^{6} b k \,m^{3}\\
&-48 E^{6} k^{3} m^{2}-24 E^{6} k^{2} m^{3}+20 E^{5} b^{3} k^{3}+12 E^{5} b^{2} k^{4}+7 E^{5} b^{2} k^{3} m +4 E^{5} b \,k^{4} m +12 E^{5} b \,k^{3} m^{2}-E^{6} b^{2} m^{2}\\
&-13 E^{6} b \,k^{2} m -6 E^{6} b k \,m^{2}+24 E^{6} k^{3} m +36 E^{6} k^{2} m^{2}+6 E^{6} k \,m^{3}+18 E^{5} b^{3} k m -16 E^{5} b^{2} k^{3}+17 E^{5} b^{2} k^{2} m \\
&+12 E^{5} b^{2} k \,m^{2}+4 E^{5} b \,k^{4}+6 E^{5} b k \,m^{3}+24 E^{5} k^{3} m^{2}+12 E^{5} k^{2} m^{3}-3 E^{4} b^{2} k^{4}-12 E^{5} b^{2} k m -E^{5} b^{2} m^{2}\\
&-12 E^{5} b \,k^{3}-E^{5} b \,m^{3}-48 E^{5} k^{3} m -72 E^{5} k^{2} m^{2}-12 E^{5} k \,m^{3}+15 E^{4} b^{3} k^{2}+19 E^{4} b^{2} k^{3}+11 E^{4} b^{2} k^{2} m \\
&+12 E^{4} b \,k^{3} m +13 E^{4} b \,k^{2} m^{2}-6 E^{5} b k m -E^{5} b \,m^{2}+8 E^{5} k^{3}+36 E^{5} k^{2} m +18 E^{5} k \,m^{2}+E^{5} m^{3}+3 E^{4} b^{3} m\\
& -14 E^{4} b^{2} k^{2}+6 E^{4} b^{2} k m +2 E^{4} b^{2} m^{2}+12 E^{4} b \,k^{3}+E^{4} b \,m^{3}+24 E^{4} k^{3} m +36 E^{4} k^{2} m^{2}+6 E^{4} k \,m^{3}\\
&-3 E^{3} b^{2} k^{3}-2 E^{4} b^{2} m -13 E^{4} b \,k^{2}-16 E^{4} k^{3}-72 E^{4} k^{2} m -36 E^{4} k \,m^{2}-2 E^{4} m^{3}+6 E^{3} b^{3} k +15 E^{3} b^{2} k^{2}\\
&+6 E^{3} b^{2} k m +13 E^{3} b \,k^{2} m +6 E^{3} b k \,m^{2}-E^{4} b m +12 E^{4} k^{2}+18 E^{4} k m +3 E^{4} m^{2}-6 E^{3} b^{2} k +E^{3} b^{2} m \\
&+13 E^{3} b \,k^{2}+8 E^{3} k^{3}+36 E^{3} k^{2} m +18 E^{3} k \,m^{2}+E^{3} m^{3}-E^{2} b^{2} k^{2}-6 E^{3} b k -24 E^{3} k^{2}-36 E^{3} k m -6 E^{3} m^{2}\\
&+E^{2} b^{3}+6 E^{2} b^{2} k +E^{2} b^{2} m +6 E^{2} b k m +E^{2} b \,m^{2}+6 E^{3} k +3 E^{3} m -E^{2} b^{2}+6 E^{2} b k +12 E^{2} k^{2}+18 E^{2} k m \\
&+3 E^{2} m^{2}-E^{2} b -12 E^{2} k -6 E^{2} m +E \,b^{2}+E b m +E^{2}+E b +6 E k +3 E m -2 E +1.
\end{align*}

\end{document}